\documentclass[10pt]{amsart}
\usepackage{latexsym}
\usepackage{amssymb}
\usepackage{amscd}
\usepackage{epsf}
\usepackage{amsmath,amssymb,graphicx,amsthm,mathrsfs,verbatim}
\usepackage{pstricks}

\begin{document}

\theoremstyle{plain}
\newtheorem{Thm}{Theorem}[section]
\newtheorem{TitleThm}[Thm]{}
\newtheorem{Corollary}[Thm]{Corollary}
\newtheorem{Proposition}[Thm]{Proposition}
\newtheorem{Lemma}[Thm]{Lemma}
\newtheorem{Conjecture}[Thm]{Conjecture}
\theoremstyle{definition}
\newtheorem{Definition}[Thm]{Definition}
\theoremstyle{definition}
\newtheorem{Example}[Thm]{Example}
\newtheorem{TitleExample}[Thm]{}
\newtheorem{Remark}[Thm]{Remark}
\newtheorem{SimpRemark}{Remark}
\renewcommand{\theSimpRemark}{}

\numberwithin{equation}{section}

\newcommand{\C}{{\mathbb C}}
\newcommand{\Q}{{\mathbb Q}}
\newcommand{\R}{{\mathbb R}}
\newcommand{\Z}{{\mathbb Z}}
\newcommand{\mbS}{{\mathbb S}}
\newcommand{\mbU}{{\mathbb U}}
\newcommand{\mbO}{{\mathbb O}}
\newcommand{\mbG}{{\mathbb G}}
\newcommand{\mbH}{{\mathbb H}}

\newcommand{\flushpar}{\par \noindent}

\newcommand{\proj}{{\rm proj}}
\newcommand{\coker}{{\rm coker}\,}
\newcommand{\Sol}{{\rm Sol}}
\newcommand{\supp}{{\rm supp}\,}
\newcommand{\codim}{{\operatorname{codim}}}
\newcommand{\sing}{{\operatorname{sing}}}
\newcommand{\Tor}{{\operatorname{Tor}}}
\newcommand{\Hom}{{\operatorname{Hom}}}
\newcommand{\wt}{{\operatorname{wt}}}
\newcommand{\rk}{{\operatorname{rk}}}
\newcommand{\dlog}{{\operatorname{Derlog}}}
\newcommand{\Olog}[2]{\Omega^{#1}(\text{log}#2)}
\newcommand{\produnion}{\cup \negmedspace \negmedspace 
\negmedspace\negmedspace {\scriptstyle \times}}
\newcommand{\pd}[2]{\dfrac{\partial#1}{\partial#2}}

\def \ba {\mathbf {a}}
\def \bb {\mathbf {b}}
\def \bc {\mathbf {c}}
\def \bd {\mathbf {d}}
\def \bone {\boldsymbol {1}}
\def \bg {\mathbf {g}}
\def \bG {\mathbf {G}}
\def \bh {\mathbf {h}}
\def \bi {\mathbf {i}}
\def \bj {\mathbf {j}}
\def \bk {\mathbf {k}}

\def \bm {\mathbf {m}}
\def \bn {\mathbf {n}}
\def \bt {\mathbf {t}}
\def \bu {\mathbf {u}}
\def \bv {\mathbf {v}}
\def \bV {\mathbf {V}}
\def \bx {\mathbf {x}}
\def \bw {\mathbf {w}}
\def \b1 {\mathbf {1}}
\def \bga {\boldsymbol \alpha}
\def \bgb {\boldsymbol \beta}
\def \bgg {\boldsymbol \gamma}

\def \itc {\text{\it c}}
\def \ite {\text{\it e}}
\def \ith {\text{\it h}}
\def \iti {\text{\it i}}
\def \itj {\text{\it j}}
\def \itm {\text{\it m}}
\def \itM {\text{\it M}} 
\def \itn {\text{\it n}}
\def \ithn {\text{\it hn}}
\def \itt {\text{\it t}}

\def \cA {\mathcal{A}}
\def \cB {\mathcal{B}}
\def \cC {\mathcal{C}}
\def \cD {\mathcal{D}}
\def \cE {\mathcal{E}}
\def \cF {\mathcal{F}}
\def \cG {\mathcal{G}}
\def \cH {\mathcal{H}}
\def \cK {\mathcal{K}}
\def \cL {\mathcal{L}}
\def \cM {\mathcal{M}}
\def \cN {\mathcal{N}}
\def \cO {\mathcal{O}}
\def \cP {\mathcal{P}}
\def \cS {\mathcal{S}}
\def \cT {\mathcal{T}}
\def \cU {\mathcal{U}}
\def \cV {\mathcal{V}}
\def \cW {\mathcal{W}}
\def \cX {\mathcal{X}}
\def \cY {\mathcal{Y}}
\def \cZ {\mathcal{Z}}

\def \ga {\alpha}
\def \gb {\beta}
\def \gg {\gamma}
\def \gd {\delta}
\def \ge {\epsilon}
\def \gevar {\varepsilon}
\def \gk {\kappa}
\def \gl {\lambda}
\def \gs {\sigma}
\def \gt {\tau}
\def \gw {\omega}
\def \gz {\zeta}
\def \gG {\Gamma}
\def \gD {\Delta}
\def \gL {\Lambda}
\def \gS {\Sigma}
\def \gW {\Omega}

\def \dim {{\rm dim}\,}
\def \mod {{\rm mod}\;}
\def \rank {{\rm rank}\,}

\newcommand{\ds}{\displaystyle}
\newcommand{\vf}{\vspace{\fill}}
\newcommand{\vect}[1]{{\bf{#1}}}
\def\R{\mathbb R}
\def\C{\mathbb C}
\def\CP{\mathbb{C}P}
\def\RP{\mathbb{R}P}
\def\N{\mathbb N}

\def\Sym{\mathrm{Sym}}
\def\Sk{\mathrm{Sk}}
\def\GL{\mathrm{GL}}
\def\Diff{\mathrm{Diff}}
\def\id{\mathrm{id}}
\def\Pf{\mathrm{Pf}}
\def\sll{\mathfrak{sl}}
\def\g{\mathfrak{g}}
\def\h{\mathfrak{h}}
\def\k{\mathfrak{k}}
\def\t{\mathfrak{t}}
\def\OcN{\mathscr{O}_{\C^N}}
\def\Ocn{\mathscr{O}_{\C^n}}
\def\Ocm{\mathscr{O}_{\C^m}}
\def\Ocnz{\mathscr{O}_{\C^n,0}}
\def\Derlog{\mathrm{Derlog}\,}
\def\expdeg{\mathrm{exp\,deg}\,}

\title[Schubert Decomposition for Milnor Fibers]
{Schubert Decomposition for Milnor Fibers of the Varieties of Singular 
Matrices}
\author[James Damon]{James Damon$^1$}

\thanks{(1) Partially supported by the National Science Foundation grant 
DMS-1105470}
\address{Department of Mathematics, University of North Carolina, Chapel 
Hill, NC 27599-3250, USA
}



\begin{abstract}
We consider the varieties of singular $m \times m$ complex matrices 
which may be either general, symmetric or skew-symmetric (with $m$ 
even).  For these varieties we have shown in another paper that they had 
compact \lq\lq model submanifolds\rq\rq\,  for the homotopy types of the 
Milnor fibers which are classical symmetric spaces in the sense of Cartan.  
In this paper we use these models, combined with results due to a number of 
authors concerning the Schubert decomposition of Lie groups and symmetric 
spaces via the Cartan model, together with Iwasawa decomposition, to give 
cell decompositions of the global Milnor fibers.
  \par 
The Schubert decomposition is in terms of \lq\lq unique ordered 
factorizations\rq\rq\, of matrices in the Milnor fibers as products of \lq\lq 
pseudo-rotations\rq\rq.  In the case of symmetric or 
skew-symmetric matrices, this factorization has the form of iterated \lq\lq 
Cartan conjugacies\rq\rq\, by pseudo-rotations. The decomposition respects 
the towers of Milnor fibers and symmetric spaces ordered by inclusions.  
Furthermore, the \lq\lq Schubert cycles\rq\rq, which are the closures of 
the Schubert cells, are images of products of suspensions of projective 
spaces (complex, real, or quaternionic as appropriate).  In the cases of 
general or skew-symmetric matrices the Schubert cycles have fundamental 
classes, and for symmetric matrices $\mod 2$ classes,  which give a basis 
for the homology.  They are also shown to correspond to the cohomology 
generators for the symmetric spaces.  For general matrices the duals of the 
Schubert cycles are represented as explicit monomials in the generators of 
the cohomology exterior algebra; and for symmetric matrices they are 
related to Stiefel-Whitney classes of an associated real vector bundle.\par 
Furthermore, for a matrix singularity of any of these types. the pull-backs of 
these cohomology classes generate a characteristic subalgebra of the 
cohomology of its Milnor fiber.    \par
We also indicate how these results extend to exceptional orbit 
hypersurfaces, complements and links, including a characteristic subalgebra 
of the cohomology of the complement of a matrix singularity. 
\end{abstract} 
\maketitle
\par
\section*{Preamble: Motivation from the Work of Brieskorn}  
\label{S:sec0a}
\par
After Milnor developed the basic theory of the Milnor fibration and the 
properties of Milnor fibers and links for isolated hypersurface 
singularities, Brieskorn was involved in fundamental ways in developing a 
more complete theory of isolated hypersurface singularities.  Furthermore 
through the work of his many students the theory was extended to isolated 
complete intersection singularities.  \par 
For isolated hypersurface singularities Brieskorn developed the 
importance of the intersection pairing on the Milnor fiber \cite{Br}.  This 
includes the computation of the intersection index for Pham-Brieskorn 
singularities, leading to the discovery that for a number of these 
singularities the link is an exotic topological sphere.  He also 
demonstrated in a variety of ways that group theory in various forms 
plays an essential role in understanding the structure of singularities.  
This includes the relation between the monodromy and the Milnor fiber 
cohomology by the Gauss-Manin connection, and including the intersection 
pairing \cite{Br2}.  This includes the relation with Lie groups, especially 
for the ADE classification for simple hypersurface singularities, where he 
identified the intersection pairing with the Dynkin diagrams for the 
corresponding Lie groups.  He also gave the structure of the discriminant 
for the versal unfoldings using the Weyl quotient map on the subregular 
elements of the Lie group \cite{Br3}.  In combined work with Arnold 
\cite{Br4}, he further showed that for the simple ADE singularities the 
complement of the discriminant is a $K(\pi, 1)$.  He continued on beyond 
the simple singularities to understand the corresponding structures for 
unimodal singularities \cite{Br5}, setting the stage for further work in 
multiple directions.   \par
The approaches which he initiated provide models for approaching 
questions for highly nonisolated hypersurface singularities which are used 
in this paper.  For matrix singularities, the high dimensional singular set 
means that the Milnor fiber, complement and link have low connectivity 
and hence can have (co)homology in many degrees \cite{KMs}.  To handle 
this complexity for matrix singularities of the various types, Lie group 
methods are employed to answer these questions.  Partial answers were 
already given in \cite{D3}, including determining the (co)homology of the 
Milnor fibers using representations as symmetric spaces.  This continues 
here by obtaining geometric models for the homology classes, 
understanding the analogue of the intersection pairing on the Milnor fiber 
via a Schubert decomposition, determining the structure of the link and 
complement, and their relations with the cohomology structure.  We see 
that there is the analogue of the ADE classification which is given for the matrix 
singularities by the ABCD classification for the infinite families of simple Lie groups.  
We also indicate how these geometric methods extend to complements and 
links, including more general exceptional orbit hypersurfaces for 
prehomogeneous spaces.  

\section*{Introduction}  
\label{S:sec0} 
\par
In this paper we derive the Schubert cell decomposition of the Milnor fibers 
of the varieties of singular matrices for $m \times m$ complex matrices 
which may be either general, symmetric, or skew-symmetric (with $m$ 
even).  We show that there is a homology basis obtained from \lq\lq 
Schubert cycles\rq\rq, which are the closures of these cells.  We further 
identify these homology classes with the cohomology.  For general matrices 
we identify the correspondence with monomials of the generators for the 
exterior cohomology algebra and for symmetric matrices we identify the 
Schubert classes with monomials in the 
Stiefel-Whitney classes of an associated vector bundle.  We also indicate how 
these results extend to more general exceptional orbit varieties and for the 
complements and links for all of these cases.  Furthermore, for general 
matrix singularities defined from these matrix types, we define 
characteristic subalgebras of the cohomology of the Milnor fibers and 
complements representing them as modules over these subalgebras. \par
 In \cite{D3} we computed the topology of the exceptional orbit 
hypersurfaces for classes of prehomogeneous spaces which include
these varieties of singular matrices.  
This included the topology of the Milnor fiber, link, and complement.  This 
used the representation of the complements and the global Milnor fibers as 
homogeneous spaces which are homotopy equivalent to compact models 
which are classical symmetric spaces studied by Cartan.  These symmetric 
spaces have representations as \lq\lq Cartan models\rq\rq, which can be 
identified as compact submanifolds of the global Milnor fibers.  \par 
We use the Schubert decomposition for these symmetric spaces 
developed by Kadzisa-Mimura \cite{KM} building on the earlier results for Lie 
groups  and Stiefel manifolds by J. H. C. Whitehead \cite{W}, C.E. Miller, 
\cite{Mi}, I. Yokota \cite{Y}.  This allows us to give a Schubert decomposition 
for the compact models of the Milnor fibers, which together with Iwasawa 
decomposition provides a cell decomposition for the global Milnor fibers in 
terms of the Schubert decomposition for these symmetric spaces. \par
The Schubert decompositions are in terms of cells defined by the unique 
\lq\lq ordered factorizations\rq\rq\, of matrices in the Milnor fibers into 
\lq\lq pseudo-rotations\rq\rq\, of types depending on the matrix type, and 
their relation to a flag of subspaces.  For symmetric or 
skew-symmetric matrices, this factorization has the form of iterated \lq\lq 
Cartan conjugacies\rq\rq\, by the pseudo-rotations.  These are given by a 
modified form of conjugacy which acts on the Cartan models.  \par
The Schubert decomposition is then further related to the co(homology) of 
the global Milnor fibers.  We do so by showing the Schubert cycles for the 
symmetric spaces are images of products of suspensions of projective 
spaces of various types (complex, real, and quaternionic as appropriate).  This 
allows us to relate the duals of the fundamental classes of the Schubert 
cycles ($\mod 2$ classes for symmetric matrices) to the cohomology 
classes given for Milnor fibers in \cite{D1}.  These are given for the 
different matrix types and various coefficients as exterior algebras.  In the 
symmetric matrix case the cohomology with $\Z/2\Z$ coefficients is  given 
as an exterior algebra on the Stiefel-Whitney classes of an associated 
real vector bundle.  For coefficient fields of characteristic zero the generators 
are classes which transgress to characteristic classes of appropriate types.  
\par 
We further indicate how these methods also apply to exceptional orbit 
hypersurfaces in \cite{D3} and how they further extend to the complements 
of the varieties and their links.  \par
Lastly, we show that for matrix singularities of these matrix types, we can 
pull-back the cohomology algebras of the global Milnor fibers to identify 
characteristic subalgebras of the Milnor fibers for these matrix singularities.  
This represents the cohomology of the Milnor fiber of a matrix singularity of 
any of these types as a module over the corresponding characteristic subalgebra.  
We also indicate how this also holds for the cohomology of the complement.  
\par
\vspace{10ex}
\flushpar
\centerline{CONTENTS}
\par
\vspace{2ex}
\flushpar
\S 1 \quad Cell Decomposition for Global Milnor Fibers in Terms of their 
Compact Models 
\par
\vspace{1ex}
\flushpar
\S 2 \quad Cartan Models for the Symmetric Spaces 
\par
\vspace{1ex}
\flushpar
\S 3 \quad Schubert Decomposition for Compact Lie Groups
\par
\vspace{1ex}
\flushpar
\S 4 \quad Schubert Decomposition for Symmetric Spaces 
\par
\vspace{1ex}
\flushpar
\S 5 \quad Schubert Decomposition for Milnor Fibers
\par
\vspace{1ex}
\flushpar
\S 6 \quad Representation of the Dual Classes in Cohomology
\par
\vspace{1ex}
\flushpar
\S 7 \quad Characteristic Subalgebra in the Milnor Fiber Cohomology of 
General Matrix \par
\hspace{1em} Singularities
\par
\vspace{1ex}
\flushpar
\S 8 \quad Extensions to Exceptional Orbit Varieties, Complements, and 
Links
\par
\vspace{5ex}
\section{Cell Decomposition for Global Milnor Fibers in Terms of their 
Compact Models }
\label{S:sec1} 
\par
We consider the varieties of singular $m \times m$ complex matrices 
which may be either general, symmetric, or skew-symmetric (with $m$ 
even).  In \cite{D1} we investigated the topology of these singularities, 
including the topology of the Milnor fiber, link and complement.  This was 
done by viewing them as the exceptional orbit varieties obtained by the 
representation of a complex linear algebraic group $G$ on a complex vector 
space $V$ with open orbit.  For example this includes the cases where $V = 
M$ is one of the spaces of complex matrices $M = Sym_m$ or $M = Sk_m$ 
(for $m = 2k$) acted on by $\GL_m(\C)$ by $B\cdot A = B A B^T$, or , $M = 
M_{m, m}$ and $\GL_m(\C)$ acts by left multiplication.  Each of these 
representations have open orbits and the resulting prehomogeneous space 
has an exceptional orbit variety $\cE$ which is a hypersurface of singular 
matrices.  \par
\begin{Definition} 
\label{Def1.1}
The {\em determinantal hypersurface} for the space of $m \times m$ 
symmetric or general matrices, denoted by $M = Sym_m$ or $M = M_{m, 
m}$ is the hypersurface of singular matrices defined by $\det : M \to \C$ 
and denoted by $\cD_m^{(sy)}$ for $M = Sym_m$, or $\cD_m$ for $M = 
M_{m, m}$.  For the space of $m \times m$ skew-symmetric matrices 
$M = Sk_m$ (for $m = 2k$) the determinantal hypersurface of singular 
matrices is defined by the Pfaffian $\Pf : Sk_m \to \C$, and is denoted by 
$\cD_m^{(sk)}$.  In the following we uniformly denote any of these functions 
as $f$.  
\end{Definition} 
Then, we showed in \cite{D3} that the Milnor fibers for each of these 
singularities at $0$ are diffeomorphic to their global Milnor fibers
$f^{-1}(1)$ which are denoted by: $F_m$ for general case, $F_m^{(sy)}$ for 
the symmetric case, and $F_m^{(sk)}$ for the skew-symmetric case.  Then, 
we show in Theorem 3.1 in \cite[\S 3]{D3} that each global Milnor fiber is 
acted on transitively by a linear algebraic group and so is a homogeneous 
space.  In particular, 
$F_m = SL_m(\C)$, $F_m^{(sy)} \simeq SL_m(\C)/SO_m(\C)$, and 
$F_{2m}^{(sk)} \simeq SL_{2m}(\C)/Sp_m(\C)$.  Moreover, these spaces 
have as deformation retracts spaces which are symmetric spaces of 
classical type studied by Cartan:  $SL_m(\C)$ has as deformation retract 
$SU_m$; $SL_m(\C)/SO_m(\C)$ has as deformation retract $SU_m/SO_m$; 
and $SL_{2m}(\C)/Sp_m(\C)$ has as deformation retract $SU_{2m}/Sp_m$. 
These are compact models for the Milnor fibers and we denote them as 
$F_m^{c}$, $F_m^{(sy)\, c}$, and $F_{2m}^{(sk)\, c}$ respectively.
\par 
 This allowed us to obtain the rational (co)homology (and integer cohomology 
for the general and skew-symmetric cases and the $\Z/2\Z$ cohomology 
for the symmetric cases), as well as using the Bott periodicity theorem to 
compute the homotopy groups in the stable range. 
\par 
We will now further use the cell decompositions  of the symmetric spaces 
together with Iwasawa decomposition to give the cell decompositions for the 
global Milnor fibers.  We recall the Iwasawa decomposition for $SL_m(\C)$ 
has the form $KAN$ where $K = SU_m$, $A_m$ consists of diagonal 
matrices with real positive entries of $\det = 1$, and $N_m$ is the nilpotent 
group of upper triangular complex matrices with $1$\rq s on the diagonal.  In 
particular, this means that the map $SU_m \times A_m \times N_m \to 
SL_m(\C)$ sending $(U, B, C) \mapsto U\cdot  B\cdot C$ is a real algebraic 
diffeomorphism.  Alternatively $A_m\cdot N_m$ consists of the upper 
triangular matrices of $\det = 1$ with complex entries except having real 
positive entries on the diagonal.  As a manifold it is diffeomorphic to a 
Euclidean space of real dimension $2\binom{m}{2} +m-1$.   We denote this 
subgroup of $SL_m(\C)$ as $\Sol_m$, which is a real solvable subgroup of 
$SL_m(\C)$.  
\par
For any of the preceding cases, let $F$ denote the Minor fiber and $Y$ the 
compact symmetric space associated to it.  Suppose that $Y$ has a cell 
decomposition with open cells $\{ e_i: I = 1, \dots , r\}$.  Then, we have the 
following simple proposition.
\begin{Proposition}
\label{Prop1.1}
With the preceding notation, the cell decomposition of $F$ is given by 
$\{ e_i \cdot \Sol_m: I = 1, \dots , r\}$.  Moreover, if the closure 
$\bar{e_i}$ has a fundamental homology class (for Borel-Moore homology) 
then $\overline{e_i \cdot \Sol_m}  = \bar{e_i} \cdot \Sol_m$ has a 
fundamental homology class with the same Poincar\'{e} dual.
\end{Proposition} 

\begin{proof}
By the Iwasawa decomposition $Y \times \Sol_m \simeq F$ via $(U, B) 
\mapsto U\cdot B$.  Hence, if for $i \neq j$, $e_i \cap e_j = \emptyset$, 
then  $(e_i \times \Sol_m) \cap (e_j \times \Sol_m)  = \emptyset$ and 
$(e_i \cdot \Sol_m) \cap (e_j \cdot \Sol_m)  = \emptyset$.  
Also, as $Y = \cup_{i} e_i$ is a disjoint union, so also is $F = \cup_{i} 
e_i\cdot \Sol_m$.  Third, each $e_i \times \Sol_m$ is homeomorphic to a 
cell of dimension $\dim_{\R}(e_i) + 2\binom{m}{2} +m-1$.  Thus, $F$ is a 
disjoint union of the cells $e_i\cdot \Sol_m$.  Lastly, $\bar{e_i} = e_i 
\cup_{j_i} e_{j_i}$ where the last union is over cells of dimension less than 
$\dim e_i$.  Hence, $\bar{e_i  \cdot \Sol_m}  = \bar{e_i} \cdot \Sol_m = 
(e_i \cdot \Sol_m) \cup_{j_i} (e_{j_i}\cdot \Sol_m)$.  Hence this is a cell 
decomposition. \par
Then, $\bar{e_i}$ is a singular manifold with open smooth manifold $e_i$.  If 
it has a Borel-Moore fundamental class, which restricts to that of $e_i$, 
then so does $\overline{e_i  \cdot \Sol_m}$ have a fundamental class that 
restricts to that for $e_i  \cdot \Sol_m \simeq  e_i \times \Sol_m$.  Then, 
as $\bar{e_i}$ is the pull-back of $\overline{e_i  \cdot \Sol_m}$ under the 
map $ \iti : Y \to Y \times \Sol_m \simeq F$ which is transverse to 
$\bar{e_i}  \times \Sol_m \simeq \overline{e_i  \cdot \Sol_m}$, by a 
fiber-square argument for Borel-Moore homology, the Poincar\'{e} dual of 
$\overline{e_i  \cdot \Sol_m}$ pulls-back via $\iti^*$ to the Poincar\'{e} dual 
of $\bar{e_i}$.  As $\iti$ is a homotopy equivalence, via the isomorphism 
$\iti^*$ the Poincar\'{e} duals agree.
\end{proof}
\par
\section{Cartan Models for the Symmetric Spaces}
\label{S:sec2} 
\par 
\subsection*{The General Cartan Model} \hfill 
\par
By Cartan, a symmetric space is defined by a Lie group G with an involution 
$\gs : G \to G$ so that the symmetric space is given by the quotient space 
$G/G^{\gs}$, where $G^{\gs}$ denotes the subgroup of $G$ invariant under 
$\gs$.  Furthermore this space can be embedded into the Lie group $G$. The 
embedding is called the {\em Cartan model}.  It is defined as follows, where 
we follow the approach of Kadzisa-Mimura \cite{KM} and the references 
therein. 
They introduce two subsets $M$ and $N$ of $G$ defined by:
$$ M \,\, = \,\, \{ g \gs(g^{-1}): g \in G\} \qquad \text{and} \qquad N \,\,  = \, \, 
\{ g \in G:  \gs(g^{-1}) = g\}.$$
Then, we have $G/G^{\gs} \simeq M \subset N$.  The inclusion is the obvious 
one, and the homeomorphism is given by $g \mapsto g \gs(g^{-1})$.  Via this 
homeomorphism, we may identify the symmetric space $G/G^{\gs}$ with the 
subset $M \subset G$.  The subspace $N$ is closed in G, and it can be shown 
that $M$ is the connected component of $N$ containing the identity element.  
In the three cases we consider, it will be the case that $M = N$.  
\par  
We also note that while $M$ and $N$ are subspaces of $G$, they are not 
preserved under products nor conjugacy; however they do have the 
following properties. \flushpar
{\it Further Properties of the Cartan Model:} \hfill 
\flushpar
\begin{itemize}
\item[i)]  there is an action of $G$ on both $M$ and $N$ defined by $g \cdot 
h = g h \gs(g^{-1})$ and on $M$ it is transitive;
\item[ii)] the homeomorphism $G/G^{\gs} \simeq M$ is $G$-equivariant 
under left multiplication on $G/G^{\gs}$ and the preceding action on $M$;
\item[iii)] both $M$ and $N$ are invariant under taking inverses; and 
\item[iv)]  if $g, h \in N$ commute then $gh \in N$. 
\end{itemize}
For $U_n$, $g^* = g^{-1}$ so an alternative way to write the action in i) is 
given by $g \mapsto h\cdot g \cdot \gs(h^*)$.   We will refer to this action 
as {\em Cartan conjugacy}. 
\par 
Then, Kadzisa-Mimura use the cell decompositions for various $G$ to give the 
cell decompositions for $M$ and hence the symmetric space $G/G^{\gs}$.  
There is one key difference with what we will do versus what Kadzisa-Mimura 
do.  They give the cell decomposition; however we also want to represent the 
closed cells where possible as the images of specific singular manifolds, 
specifically products of suspensions of projective spaces of various types 
and to relate the fundamental homology classes to corresponding classes in 
cohomology.  Together with the reasoning in \S \ref{S:sec1} and the 
identification of the global Milnor fibers with the Cartan models, we will then 
be able to give the Schubert decomposition for the global Milnor fibers and 
identify the Schubert homology classes with dual cohomology classes.  \par

\subsection*{The Cartan Models for $SU_m$, $SU_m/SO_m$, and 
$SU_{2m}/Sp_m$} \hfill
\par
For the three cases we consider: 
$SU_m$ , $SU_m/SO_m$, $SU_{2m}/Sp_m$, we first observe that the exact 
sequence of groups \eqref{CD2.1} does not split  
\begin{equation}
\label{CD2.1}
\begin{CD} 
 {1} @>>> {SU_m} @>>> {U_m} @>\det >> {S^1}  @>>>  {1} \, .
\end{CD}  
\end{equation}
However, it does split as manifolds $U_m \simeq S^1 \times SU_m$ sending 
$$ C \mapsto (\det(C), I_{1, m-1}(\det(C))\cdot C),$$ 
where $I_{1, m-1}(\det(C)^{-1})$ is the $m \times m$ diagonal matrix with $1$\rq s 
on the diagonal except in the first position where it is $\det(C)^{-1}$.  Thus, 
topological statements about $U_m$ have 
corresponding statements about $SU_m$ and conversely.\par 
We first give the representation for the symmetric spaces.  For $SU_m$ we 
just use itself as a compact Lie group.  \par 
Next, for $SU(m)/SO(m)$ we let the involution $\gs$ on 
$SU(m)$ be defined by $C \mapsto \overline{C}$.  We see that $\gs(C) = C$ is 
equivalent to $C = \overline{C}$.  Thus $C$ is a real matrix which is unitary; and 
hence $C$ is real orthogonal.  As $\det(C) = 1$, we see that $SU_m^{\gs} = 
SO_m$.  \par
The third case is  $SU_{2n}/Sp_n$ for $m = 2n$.  In this case, the involution 
$\gs$ on $SU_{2n}$ sends $C \mapsto J_n\overline{C}J_n^*$ where $J_n$ is the 
$2n \times 2n$ block diagonal matrix with $2 \times 2$ diagonal blocks 
$\begin{pmatrix} 0 & 1 \\ -1 & 0 \end{pmatrix}$.  As $J_n^* = J_n^T = -J_n 
= J_n^{-1}$, then $\gs(C) = C$ is equivalent to $J_n\overline{C}J_n = -C$, or as 
$C^{-1} = \overline{C}^T$ we can rearrange to obtain $C^T J_n C = J_n$ (or 
alternatively $C J_n C^T = J_n$), which implies that $C$ leaves invariant the 
bilinear form $(v, w) = v^T J_n w$ (for column vectors $v$ and $w$) and so is 
an element of $Sp_n(\C)$, and so an element of $Sp_n = SU_{2n}\cap 
Sp_n(\C)$.  \par 
The corresponding Cartan models are then given as follows.  We denote the 
Cartan models by respectively: $\cC_m$, $\cC_m^{(sy)}$, and 
$\cC_m^{(sk)}$.  \par
First, for $G = SU_m$, which is itself a symmetric space, and we let $\cC_m 
= SU_m$.  In this case, Cartan conjugacy is replaced by left multiplication.  
\par
Second, for $SU_m/SO_m$ we claim 
\begin{equation}
\label{Eqn2.1}
 \cC_m^{(sy)} \,\, \overset{def}{=} \,\, \{ C\cdot C^T: C \in SU_m\} \,\, = 
\,\, \{B \in SU_m : B = B^T\} \, . 
\end{equation} 
The inclusion of the LHS in the RHS is immediate.  For the converse, we note 
that if $B \in SU_m$ and $B = B^T$, then by the following Lemma given in 
\cite{KM} there is an orthonormal basis of eigenvectors which are real 
vectors so we may write $B = ACA^{-1}$ with $A$ an orthogonal matrix and 
$C$ a diagonal matrix with diagonal entries $\gl_j$ so that 
$| \gl_j | = 1$.  Thus, $A^{-1} = A^T$, and so $B = ADA^T\cdot ADA^T$ with 
$D$ a diagonal matrix with entries $\sqrt{\gl_j}$.  
\begin{Lemma}
\label{Lem2.1}
If $B \in SU_m$ and $B = B^T$ then there is a real orthonormal basis of 
eigenvectors for $B$.  
\end{Lemma}
This is a simple consequence of the eigenspaces being invariant under 
conjugation, which is easily seen to follow from the conditions.  In this case, 
Cartan conjugacy by $A$ on $B$ is checked to be given by $B \mapsto 
A\cdot B \cdot A^T$. 
\par
Third, for $SU_{2n}/Sp_n$ with $m = 2n$,  we may directly verify 
\begin{equation}
\label{Eqn2.2}
 \cC_m^{(sk)} \,\, \overset{def}{=} \,\, \{ C\cdot J_n\cdot C^T\cdot 
J_n^*: C \in SU_{2n}\} \,\, = \,\, \{ B \in SU_{2n} : (B\cdot J_n)^T = -
B\cdot J_n\} \, . 
\end{equation}
Then, Cartan conjugacy by $A$ on $B$ is given by $B \mapsto A\cdot 
(B\cdot J_n) \cdot A^T\cdot J_n^{-1}$, with $B\cdot J_n$ 
skew-symmetric for $B \in \cC_m^{(sk)}$.
\par
Hence, from \eqref{Eqn2.1}, we have the compact model for $F_m^{(sy)}$ as 
a subspace is given by $F_m^{(sy)\, c} = SU_m \cap Sym_{m}(\C)$ and the 
Cartan model for the symmetric space $SU_m/SO_m$ is given by $F_m^{(sy)\, 
c}$ itself.  Similarly, from \eqref{Eqn2.2}, we have the compact model for 
$F_m^{(sk)}$ with $m =2n$ as a subspace is given by $F_m^{(sk)\, c} = SU_{m} 
\cap Sk_{m}(\C)$ and the Cartan model for the symmetric space 
$SU_{2n}/Sp_n$ is given by $F_m^{(sk)\, c}\cdot J_n^{-1}$.  
\par
\begin{Remark}
\label{Rem2.3}
Frequently for all three cases, we will want to apply a Cartan conjugate for 
an element of $U_n$ instead of $SU_n$.  The formula for the Cartan 
conjugate remains the same and the corresponding symmetric spaces are 
$U_n$, $U_n/O_n$, and $U_{2n}/Sp_n$.  By the properties of Cartan 
conjugacy, an iteration of Cartan conjugacy by elements  $A_i \in U_n$ 
whose product belongs to $SU_n$ will be a Cartan conjugate by an element of 
$SU_n$ and preserve the Cartan models of interest to us. 
\end{Remark}
\flushpar
\subsubsection*{Tower Structures of Global Milnor fibers and Symmetric 
Spaces by Inclusion} \hfill 
\par 
Lastly, these global Milnor fibers, symmetric spaces and compact models 
form towers via inclusions: i) sending $A \mapsto \begin{pmatrix} A & 0 \\ 0 
& 1 \end{pmatrix}$ for $SU_m \subset SU_{m+1}$, $F_m \subset 
F_{m+1}$, or $F_m^{(sy)} \subset F_{m+1}^{(sy)}$  which induce inclusions 
of the symmetric spaces $SU_m$ and $SU_m/SO_m$ and corresponding 
global Milnor fibers, or  ii) sending $A \mapsto \begin{pmatrix} A & 0 \\ 0 & I_2 
\end{pmatrix}$ for the $2 \times 2$ identity matrix $I_2$ for $SU_m 
\subset SU_{m+2}$ for $m = 2n$ and the corresponding symmetric spaces 
$SU_{2n}/Sp_n$ and Milnor fibers $F_m^{(sk)} \subset F_{m+2}^{(sk)}$.  
The Schubert decompositions will satisfy the additional property that they 
respect the inclusions. \par 

We summarize these results by the following.
\begin{Proposition}
\label{Prop2.3}
For the varieties of singular $m \times m$ complex matrices which are 
either general, symmetric or skew-symmetric, their global Milnor fibers, 
representations as homogeneous spaces, compact models given as 
symmetric spaces and Cartan models are summarized in Table 
\ref{Mil.fib.sym.sp}.
\end{Proposition}
\par
\vspace{2ex}
\begin{table}[h]
\begin{tabular}{|l|c|c|c|l|}
\hline
Milnor   & Quotient  & Symmetric  & Compact Model & Cartan \\
 Fiber $F_m^{(*)}$  &  Space  &  Space  & $F_m^{(*)\, c}$ &  Model   \\
\hline
$F_m$  & $SL_m(\C)$ & $SU_m$  &  $SU_m$  & $F_m^{c}$ \\
\hline
$F_m^{(sy)}$  &  $SL_m(\C)/SO_m(\C)$  & $SU_m/SO_m$  &  $SU_m \cap 
Sym_{m}(\C)$  &  $F_m^{(sy)\, c}$ \\
\hline
$F_m^{(sk)}, m = 2n$ &  $SL_{2n}(\C)/Sp_{n}(\C)$  & $SU_{2n}/Sp_n$  &  
$SU_{m} \cap Sk_{m}(\C)$  & $F_{m}^{(sk)\, c}\cdot J_n^{-1}$  \\
\hline
\end{tabular}
\caption{Global Milnor fiber, its representation as a homogenenous space, 
compact model as a symmetric space, compact model as subspace and 
Cartan model.}
\label{Mil.fib.sym.sp}
\end{table}
\par
\section{Schubert Decomposition for Compact Lie Groups }
\label{S:sec3} 
\par
We recall the \lq\lq Schubert decomposition\rq\rq for compact Lie groups, 
concentrating on $SU_n$.  The cell decompositions of certain compact Lie 
groups, especially $SO_n$ and $U_n$ and $SU_n$ were carried out by C. E. 
Miller \cite{Mi} and I. Yokota \cite{Y}, building on the work of J. H. C. 
Whitehead \cite{W} for the cell decomposition of Stiefel varieties.   
In the case of Grassmannians, the Schubert decomposition is in terms of the 
dimensions of the intersections of the subspaces with a given fixed flag of 
subspaces.  For these Lie groups, elements are expressed as ordered 
products of (complex) \lq\lq pseudo rotations\rq\rq about complex 
hyperplanes (or reflections about real hyperplanes in the case of $SO_n$).  
The cell decomposition is based on the subspaces of a fixed flag that contain 
the orthogonal lines to the hyperplane axes of rotation (or reflection).  We 
will concentrate on the complex case which is relevant to our situation.  \par
\subsection*{(Complex) Pseudo-Rotations} \hfill 
\par
We note that given a complex $1$--dimensional subspace $L \subset \C^n$, 
we can define a \lq\lq (complex) pseudo-rotation\rq\rq about the orthogonal 
hyperplane $L^{\perp}$ as follows.  Let $x \in L$ be a unit vector.  As $L$ is 
complex we have a positive sense of rotation through an angle $\theta$ 
given by $x \mapsto e^{i\theta} x$.  We extend this to be the identity on 
$L^{\perp}$.  This is given by the following formula for any $x^{\prime} \in 
\C^n$:
$$ A_{(\theta, x)}(x^{\prime}) \,\, = \,\,  x^{\prime} - ((1 - e^{i\theta}) 
<x^{\prime}, x>)\, x \, .  $$   
This is not a true rotation as a complex linear transformation so we refer to 
this as a \lq\lq pseudo-rotation\rq\rq.  Then, $A_{(\theta, x)}$ can be 
written in matrix form as 
$A_{(\theta, x)} = (I_n - (1 - e^{i\theta})\, x\cdot \bar{x}^T)$ for $x$ an 
$n$-dimensional column vector.  
\begin{Remark}
\label{Rem3.0}
	In the special case that $A_{(\theta, x)}$ has finite order as an 
element of the group $U_n$, it is called a \lq\lq complex reflection\rq\rq. 
\end{Remark}
\par 
We observe a few simple properties of pseudo-rotations:
\begin{itemize}
\item[i)] $A_{(\theta, x)}$ only depends on $L = <x>$, so we will also feel 
free to use the alternate notation $A_{(\theta, L)}$;
\item[ii)] $A_{(\theta, x)}$ is a unitary transformation with 
$\det(A_{(\theta, x)}) = e^{i\theta}$;
\item[iii)]  if $B \in U_n$, then 
$B\cdot A_{(\theta, x)}\cdot B^{-1} = A_{(\theta, Bx)}$ is again a 
pseudo-rotation; and
\item[iv)] $\overline{A_{(\theta, x)}} = A_{(-\theta, \bar{x})}$; 
$A^{-1}_{(\theta, x)} = A_{(-\theta, x)}$; and $A^{T}_{(\theta, x)} = 
A_{(\theta, \bar{x})}$.
\end{itemize}
 \par
\subsection*{Ordered Factorizations in $SU_m$ and Schubert Symbols} \hfill
\par
Then, given any $B \in SU_n$, we may diagonalize $B$ using an orthonormal 
basis $\{ v_1, \dots , v_n\}$ so if $C$ denotes the unitary matrix with the 
$v_i$ as columns, then we may write $B = C D C^{-1}$ where $D$ is a 
diagonal matrix with diagonal entries $\gl_i$ of unit length so that $\prod_{i 
= 1}^{n} \gl_i = 1$.  This can be restated as saying that $B$ is a product of 
pseudo-rotations about the hyperplanes $<v_j>^{\perp}$ with angles 
$\theta_j$ where $\gl_j = e^{i\theta_j}$.  Thus, $B = \prod_{j = 1}^{n}  
A_{(\theta_j, v_j)}$.  However, we note that as certain eigenspaces may 
have dimension $ > 1$, the terms and their order in the product are not unique.  \par
There is a method introduced by Whitehead and used by Miller and Yokota for 
obtaining a unique factorization leading to the Schubert decomposition in 
$SU_n$.  The product is rewritten as a product of different 
pseudo-rotations whose lines satisfy certain inclusion relations for a fixed 
flag leading to an ordering of the pseudo-rotations.  We let $0 \subset \C 
\subset \C^2 \subset \dots \subset \C^n$ denote the standard flag.  Then, 
if $L = \, <x> \, \subset \C^k$ and $L = \, <x> \, \not\subset \C^{k-1}$, we 
will say that {\em $x$ and $L$ minimally belong to $\C^k$ }and introduce the 
notation $x \in_{\min} \C^k$ or $L \subset_{\min} \C^k$.  If $x = (x_1, x_2, 
\dots , x_n)$ then $x \in_{\min} \C^k$ iff $x_{k+1} = \cdots  = x_n = 0$ 
and $x_k \neq 0$.  We observe two simple properties: if $x \in_{\min} \C^k$ 
then $\bar{x} \in_{\min} \C^k$; and if $x^{\prime} \in_{\min} 
\C^{k^{\prime}}$ with $k^{\prime} < k$, then $A_{(\theta, x^{\prime})}(x) 
\in_{\min} \C^k$. 
\par
Then to rewrite the product in a different form, we proceed, as in the other 
papers, to follow Whitehead with the following lemma.
\begin{Lemma}
\label{Lem3.1}
Suppose that we have two pseudo-rotations $A_{(\theta, x)}$ and 
$A_{(\theta^{\prime}, x^{\prime})}$ with $x \in_{\min} \C^{m}$ and 
$x^{\prime} \in_{\min} \C^{m^{\prime}}$.  
\begin{itemize}
\item[1)] If $m > m^{\prime}$, then 
\begin{equation}
\label{Eqn3.1}
 A_{(\theta, x)} \cdot A_{(\theta^{\prime}, x^{\prime})}\,\, = \,\, 
A_{(\theta^{\prime}, x^{\prime})} \cdot 
A_{(\theta, \tilde{x})}  
\end{equation}
where $\tilde{x} = A_{(\theta^{\prime}, x^{\prime})}^{-1}(x)$. 
\item[2)]  If $m = m^{\prime}$, and $<x> \neq <x^{\prime}>$ let $W = < x, 
x^{\prime}>$, which has dimension $2$, and let $L = \, <\tilde{x}> \, = W 
\cap \,\C^{m-1}$, with $\tilde{x} \in_{\min} \C^k$ for $k \leq m-1$.  Then, 
there exist pseudo-rotations $A_{(\tilde{\theta}, \tilde{x})}$ and 
$A_{(\tilde{\theta}^{\prime}, \tilde{x}^{\prime})}$ with $\tilde{x} \in_{\min} 
\C^{k}$ and $\tilde{x}^{\prime} \in_{\min} \C^{m}$ such that 
\begin{equation}
\label{Eqn3.1b}
A_{(\theta, x)} \cdot A_{(\theta^{\prime}, x^{\prime})}\,\, = \,\, 
 A_{(\tilde{\theta}, \tilde{x})}  \cdot A_{(\tilde{\theta}^{\prime}, 
\tilde{x}^{\prime})}\, .
\end{equation}
Moreover, for generic $x, x^{\prime} \in_{\min} \C^m$, $\tilde{x} \in_{\min} 
\C^{m-1}$.
 \end{itemize}
\end{Lemma}
\begin{proof}
\par
For 1), by property iii) of pseudo-rotations, $A_{(\theta^{\prime}, 
x^{\prime})}^{-1}\cdot A_{(\theta, x)} \cdot A_{(\theta^{\prime}, 
x^{\prime})}$ is a pseudo-rotation of the form $A_{(\theta, \tilde{x})}$ with 
$\tilde{x} = A_{(\theta^{\prime}, x^{\prime})}^{-1}(x)$.  
Also, both $A_{(\theta, x)}$ and $A_{(\theta^{\prime}, x^{\prime})}$ are 
the identity on $\C^{m \perp}$; hence $\tilde{x} \in_{\min} \C^m$.  \par
For 2), if $<x> \, =\, <x^{\prime}>$, then the pseudo-rotations commute.  
Next, suppose these lines differ so the complex subspace $W$ spanned by 
$x$ and $x^{\prime}$ is $2$-dimensional.  Then, $\dim_{\C} W \cap 
\C^{m-1} = 1$.  We denote it by $L$ and let it be spanned by a unit vector 
$\tilde{x}$ with say $\tilde{x} \in_{\min} \C^k$ for $k \leq m-1$ (and 
generically $k = m-1$).  We note that both pseudo-rotations are the identity 
on $W^{\perp}$.  Also, $W \subset \C^{m}$.  It is sufficient to consider the 
pseudo-rotations restricted to $W \simeq \C^2$ with $\tilde{x}$ denoted by 
$e_2$ and orthogonal unit vector $e_1$.  Then, let $(A_{(\theta, x)} \cdot 
A_{(\theta^{\prime}, x^{\prime})})^{-1}(e_1) = v$.  Then, we want a 
pseudo-rotation on $W$ that sends $e_1 \mapsto v$.  If $v \neq -e_1$, then 
reflection about the complex line spanned by $e_1 + v$, is a 
pseudo-rotation by $\pi$ and sends $e_1$ to $v$.  If $v = -e_1$, then 
reflection about the complex line spanned by $e_2$ works instead.  If we 
denote this reflection by 
$A_{(\pi, \tilde{x}^{\prime})}$, then $A_{(\theta, x)} \cdot 
A_{(\theta^{\prime}, x^{\prime})} \cdot A_{(\pi, \tilde{x}^{\prime})}$ is a 
unitary transformation which fixes $e_1$ and is hence a pseudo-rotation 
about the line $<e_1>$ and so sends $e_2  = \tilde{x}$ to 
$e^{i\tilde{\theta}} \tilde{x}$ for some angle 
$\tilde{\theta}$. Thus,   
$$ A_{(\theta, x)} \cdot A_{(\theta^{\prime}, x^{\prime})} \,\, = \,\,  
A_{(\tilde{\theta}, \tilde{x})}  \cdot A_{(\tilde{\theta}^{\prime}, 
\tilde{x}^{\prime})}\, $$
giving the result. 
\end{proof}
\par 
This allows us to rewrite a product of pseudo-rotations as a product where 
the lines are minimally contained in successively larger subspaces of the 
flag. \par 
\subsubsection*{Whitehead Algorithm for ordered factorization of Unitary 
matrices} 
Given $B \in SU_n$, we may write $B = \prod_{j = 1}^{k} A_{(\theta_j, 
x_j)}$, with the $\{x_j \}$ an orthonormal set of vectors with say $x_j 
\in_{\min} \C^{m_j}$.  Note that $k$ may be less than $n$ as we may 
exclude the eigenvectors $x_j^{\prime}$ with eigenvalue $1$, which give 
$A_{(0, x_j^{\prime})} = I_n$.  Then, we may use Lemma \ref{Lem3.1} to 
reduce the product into a standard form as follows. 
For the sequence $(m_1, m_2, \dots , m_k)$, we find the largest $j$ so that 
$m_j \geq m_{j+1}$.  If $m_j > m_{j+1}$ then by 1) of Lemma \ref{Lem3.1}, 
we may replace $A_{(\theta_j, x_j)} \cdot A_{(\theta_{j+1}, x_{j+1})}$ by 
$A_{(\theta_{j+1}, x_{j+1})} \cdot A_{(\theta_j, \tilde{x}_j)}$, with 
$\tilde{x}_j \in_{\min} \C^{m_j}$.  If instead $m_j = m_{j+1}$, then by 2) of 
Lemma \ref{Lem3.1}, we may instead replace the product by 
$A_{(\theta_j^{\prime}, x_j^{\prime})} \cdot A_{(\theta_{j+1}^{\prime}, 
x_{j+1}^{\prime})}$, where $x_{j+1}^{\prime} \in_{\min} \C^{m_j}$ and 
$x_{j}^{\prime} \in_{\min} \C^{\ell}$, where $\ell < m_j$ satisfies $(<x_j, 
x_{j+1}> \cap\, \C^{m_j}) \subset_{\min} \C^{\ell}$.  \par
Then, we relabel the angles and vectors to be $(\theta_j, x_j)$, where now 
$m_j < m_{j+1} < \cdots < m_k$.  Then, we may repeat the procedure until 
we obtain $m_1 < m_2 < \cdots < m_k$.  We summarize the final result of 
this process.
\begin{Lemma}
\label{Lem3.3} 
Given $B \in SU_n$, it may be written as a product 
\begin{equation}
\label{Eqn3.3}
B \, \, = \,\,  A_{(\theta_1, x_1)} \cdot A_{(\theta_2, x_2)} \cdots 
A_{(\theta_k, x_k)}\, ,
\end{equation}
 with 
$x_j \in_{\min} \C^{m_j}$ and $1 \leq m_1 < m_2 < \cdots < m_k \leq n$, 
and each $\theta_i \not \equiv 0  \,\mod 2\pi$.
\end{Lemma}
\par
If $B$ has the form given in Lemma \ref{Lem3.3} with $m_1 > 1$, then we 
will say that $B$ has Schubert type $\bm = (m_1, m_2, \cdots , m_k)$ and 
write $\bm(B) = \bm$.  If instead $m_1 = 1$ then as $\det(B) = 1$
$$B = A_{(-\tilde{\theta}, e_1)} \cdot A_{(\theta_2, x_2)} \cdot 
A_{(\theta_2, x_2)} \cdots A_{(\theta_k, x_k)}$$  
where $\tilde{\theta} \equiv \sum_{j = 2}^{k} \theta_j \mod 2\pi$ and we 
instead denote $\bm(B) = (m_2, \cdots , m_k)$.  For the case of an empty 
sequence with $k = 0$, we associate the unique identity element $I$. We 
refer to the tuple $\bm = (m_1, m_2, \cdots , m_k)$ as the Schubert 
symbol of $B$.  It will follow from Theorem \ref{Thm3.5} that this 
representation is unique.  
\par
There is also an alternative way to obtain a factorization \eqref{Eqn3.3} 
where instead $x_j \in_{\min} \C^{m_j^{\prime}}$ with a decreasing sequence
$m^{\prime}_1 > m^{\prime}_2 > \cdots > m^{\prime}_k$.  In fact, if we give a 
representation for $B^{-1}$ as in \eqref{Eqn3.3} with the $m_i$ increasing, 
then taking inverses gives a product of $A_{(\theta_i, x_i)}^{-1} = A_{(-
\theta_i, x_i)}$ in decreasing order.  There is a question for a given $B \in 
SU_n$ about the relation between the increasing and decreasing symbols.  
The relation between these is a consequence of the following lemma which is 
basically that given in \cite[Prop. 4.5]{KM} and is a consequence of the 
uniqueness of the Schubert symbol for one direction of ordering. 
\begin{Lemma}
\label{Lem3.4}
Suppose $x_i \in_{\min} \C^{m_i}$, for $1 \leq i \leq k$ and $m_1 < m_2 < 
\cdots < m_k$; and $y_j \in_{\min} \C^{m_j^{\prime}}$, for $1 \leq j \leq 
k^{\prime}$ and $m_1^{\prime} < m_2^{\prime} < \cdots < m_k^{\prime}$.  
Also, suppose $\theta_i, \theta_i^{\prime} \not \equiv 0 \, \mod 2\pi$ for 
each $i$.  Let $A_i = A_{(\theta_i, x_i)}$ and $B_j = A_{(\theta_j^{\prime}, 
y_j)}$.  If
$$A_{1} \cdot A_{2} \cdots A_{k} \,\, = \,\, B_{k^{\prime}} \cdot 
B_{k^{\prime}-1} \cdots B_{1}$$
then the following hold:
\begin{itemize}
\item[a)]  $k = k^{\prime}$ and $(m_1, m_2,\dots , m_k) =  (m_1^{\prime}, 
\dots , m_{k^{\prime}}^{\prime})$;
\vspace{1ex}
\item[b)] $A_{i} \,\, = \,\, B_{1}^{-1} \cdot B_{2}^{-1} \cdots  B_{i-1}^{-1} 
\cdot B_{i} \cdot B_{i-1} \cdots B_{1}$ for $1 \leq i \leq k$; and
\vspace{1ex}
\item[c)]  $B_{i} \,\, = \,\, A_{1} \cdot A_{2} \cdots  A_{i-1} \cdot A_{i} 
\cdot A_{i-1}^{-1} \cdots A_{1}^{-1}$ for $1 \leq i \leq k$.
\end{itemize}
In the cases of $k = 1$ in b) and c), we let $A_0 = B_0 = I_m$ so they are 
understood to be $A_1 = B_1$. 
\end{Lemma}
\begin{proof}
We let $C_i$ denote the RHS of the equation in b) but for $1 \leq i \leq 
k^{\prime}$.  Since $B_{i-1} \cdot B_{i-2} \cdots B_{1}$ leaves 
pointwise invariant $(\C^{m_i^{\prime}})^{\perp}$, we conclude $B_{i-1} 
\cdot B_{i-2} \cdots B_{1}(y_i) = y_i^{\prime} \in_{min} 
\C^{m_i^{\prime}}$; hence by property iii) for pseudo rotations, $C_i = 
A_{(\theta_i^{\prime}, y_i^{\prime})}$.  Thus, we have that $A$ has two 
different Schubert factorizations with increasing Schubert symbols $(m_1, 
m_2,\dots , m_k)$ and  $(m^{\prime}_{1}, \dots , 
m ^{\prime}_{k^{\prime}})$.  By the uniqueness of the Schubert symbols, we 
obtain a). \par
Furthermore, by the uniqueness of the Schubert decomposition stated in 
Theorem \ref{Thm3.5} (for increasing Schubert decomposition) and Remark 
\ref{Rem3.6}, it then furthermore follows that $A_i = C_i$ for all $i$ so b) 
holds.  Lastly, the uniqueness of the increasing order Schubert 
decomposition implies by taking inverses that we also have uniqueness of 
decreasing order Schubert decomposition.  Then, the corresponding analogue 
of the argument for b) yields c).
\end{proof}
We then have the following corollary
\begin{Corollary}
\label{Cor3.5}
If $B \in SU_n$, then 
$$\bm(B) \,\, = \,\,  \bm(B^{-1})  \,\, = \,\,  \bm(\overline{B}) \,\, = \,\,  
\bm(B^T)\, . $$
\end{Corollary}
\begin{proof}
Given an increasing Schubert factorization
$B = A_{1} \cdot A_{2} \cdots A_{k}$ for $A_i = A_{(\theta_i, x_i)}$ with 
Schubert symbol $\bm = (m_1, m_2,\dots , m_k)$, then $B^{-1}  = A_{k} 
\cdot A_{k-1} \cdots A_{1}$ is a Schubert factorization for decreasing 
order.  This has the decreasing Schubert symbol $(m_k, m_{k-1},\dots , 
m_1)$, and hence $B^{-1}$ has the same increasing Schubert symbol $\bm$. \par
Next, $\overline{B} = \overline{A_{1}} \cdot \overline{A_{2}} \cdots 
\overline{A_{k}}$, and by property iv) of pseudo-rotations $\overline{A_{i}} 
= A_{(-\theta_i, \bar{x_i})}$ so the Schubert Symbol is the same.  \par 
Lastly, as $B \in SU_n$, $B^T = \overline{B^{-1}}$, which combined with the 
two other properties implies that it has the same Schubert symbol.
\end{proof}
\par
\begin{Remark}
\label{Rem3.8}
We will use the increasing order for the Schubert symbol to be in agreement 
with that used for the Schubert decomposition as in 
Milnor-Stasheff \cite{MS}.  In fact, if $A = A_{1} \cdot A_{2} \cdots A_{k}$ 
for $A_i = A_{(\theta_i, x_i)}$ with Schubert symbol $\bm = (m_1, 
m_2,\dots , m_k)$, and we let $V = \C <x_1, \dots , x_k>$, then $\dim_{\C} 
V \cap \C^{m_i} = i$ so $V$ as an element of the Grassmannian 
$G_k(\C^n)$ would also have Schubert symbol $\bm$.  In \cite{KM}, the 
decreasing order Schubert symbol is used; however, we easily change 
between the two. 
\end{Remark}
 \par
We next state the form of the Schubert decomposition given in terms of the 
Schubert factorization giving the Schubert types for elements of $SU_n$.
\subsection*{Schubert Decomposition for $SU_n$} \hfill
\par

In describing the Schubert decomposition for $SU_n$, we are giving a version 
of that contained in \cite{W}, \cite{Mi}, \cite{Y} and summarized in \cite{KM} 
(but using instead an increasing order). \par
Given an increasing sequence $m_1 < m_2 < \dots < m_k$ with $1 < m_1 $ 
and $m_k \leq n$, which we denote by $\bm = (m_1, m_2, \dots , m_k)$, we 
define a map 
$$ \psi_{\bm} : S \CP^{m_1 - 1} \times S \CP^{m_2 - 1} \times \cdots \times S 
\CP^{m_k - 1} \longrightarrow SU_n\, , $$
where $SX$ denotes the suspension of $X$.  This is given as follows: \par
First, we define a simpler map for $m \leq n$, $I = [0, 1]$ and a complex line 
$L \subset \C^m$,
$\tilde {\psi}_m : I \times \CP^{m-1} \to SU_n$ defined by $\tilde 
{\psi}_m(t, L) = A_{(2\pi t, L)}$.  Since $A_{(0, L)} = A_{(2\pi, L)} = I_n$ 
independent of $L$, this descends to a map $\psi_m : S\CP^{m-1} \to 
SU_n$.  Then, we define 
\begin{align}
\label{Eqn3.4}
\psi_{\bm}((t_1, L_1), \dots  , (t_k, L_k)) \,\,  &= \,\, A_{(-2\pi \tilde{t}, 
e_1)}\cdot {\psi}_{m_1}(t_1, L_1) \cdot  {\psi}_{m_2}(t_2, L_2) \cdots 
{\psi}_{m_k}(t_k, L_k)  \notag \\ 
 \,\,  &= \,\, A_{(-2\pi \tilde{t}, e_1)}\cdot A_{(2\pi t_1, L_1)} \cdot 
A_{(2\pi t_2, L_2)} \cdots A_{(2\pi t_k, L_k)} \, . 
\end{align}
where $\tilde{t} = \sum_{j = 1}^{k}t_j$.  We note that the first factor 
$A_{(-2\pi \tilde{t}, e_1)}$ ensures the product is in $SU_n$ as in the 
splitting for \eqref{CD2.1}.  \par
We observe that each $I \times \CP^{m-1}$ has an open dense cell 
$$ E_m = (0, 1) \times \{x = (x_1, \dots, x_m, 0 , \dots 0) : (x_1, \dots, 
x_m) \in S^{2m-1} \mbox{ and } x_m > 0 \} $$
which is of dimension $2m-1$ (as $x_m = \sqrt{1 - \sum_{j = 1}^{m-1} 
|x_j|^2}$ \,).  Also, if $x = (x_1, \dots, x_m, 0 , \dots 0)$ with $x_m > 0$, 
then $x \in_{\min} \C^m$.  \par 
We now introduce some notation and denote 
$$\tilde{S}_{\bm} = S \CP^{m_1 - 1} \times S \CP^{m_2 - 1} \times \cdots \times 
S \CP^{m_k - 1}\, ;$$
also, we consider the corresponding cell $E_{\bm} = E_{m_1} \times 
E_{m_2} \times \cdots \times E_{m_k}$, and the image $S_{\bm} = 
\psi_{\bm}(E_{\bm})$ in $SU_n$.  Then, $E_{\bm}$ is an open dense cell in 
$\tilde{S}_{\bm}$ with $\dim_{\R} E_{\bm} = \sum_{j = 1}^{k} (2m_j - 1) =  
2 | \bm | - \ell(\bm)$ for $| \bm | =\sum_{j = 1}^{k} m_j$ and $\ell(\bm) = 
k$, which we refer to as the length of $\bm$.  Also, the image $S_{\bm} = 
\psi_{\bm}(E_{\bm})$ consists of elements of $SU_n$ of Schubert type 
$\bm$.  Furthermore, $\overline{S_{\bm}} = \psi_{\bm}(\tilde{S}_{\bm})$.  
Then the results of Whitehead, Miller and Yokota together give the following 
Schubert decomposition of $SU_n$.
\begin{Thm}
\label{Thm3.5}
The Schubert decomposition of $SU_n$ has the following properties:
\begin{itemize}
\item[a)] $SU_n$is the disjoint union of the $S_{\bm}$ as $\bm = (m_1, 
\dots , m_k)$ varies over all increasing sequences with $1 < m_1$, $m_k 
\leq n$, and $0 \leq k \leq n-1$.
\item[b)] The map $\psi_{\bm} : E_{\bm} \to S_{\bm}$ is a homeomorphism.
\item[c)] $(\overline{S_{\bm}}\backslash S_{\bm}) \subset 
\cup_{\bm^{\prime}} S_{\bm^{\prime}}$, where the union is over all 
$S_{\bm^{\prime}}$ with $\dim S_{\bm^{\prime}} < \dim S_{\bm}$.
\item[d)] the Schubert cells $S_{\bm}$ are preserved under taking  
inverses, conjugates, and transposes.
\end{itemize}
\end{Thm}
We note that d) follows from Corollary \ref{Cor3.5}.  \par
Hence, the Schubert decomposition by the cells $S_{\bm}$ is a cell 
decomposition of $SU_n$.  The cells $S_{\bm}$ are referred to as the {\em 
Schubert cells of $SU_n$}.  We note that as $\overline{S_{\bm}}$ is the 
image of the \lq\lq singular manifold\rq\rq\, $\tilde{S}_{\bm}$ which 
has a Borel-Moore fundamental class, we can describe in \S \ref{S:sec5} the 
homology of $SU_n$ in terms of the images of these fundamental classes.
\begin{Remark}
\label{Rem3.6}
There is an analogous Schubert decomposition for $U_n$ where the Schubert 
symbols can include $m_1 = 1$.  
\end{Remark}
\par
\section{Schubert Decomposition for Symmetric Spaces }
\label{S:sec4} 
\par
For the Milnor fibers for the varieties of singular matrices, we have compact 
models which are symmetric spaces.  To give the Schubert decomposition of 
these, we use the results of Kadzisa and Mimura \cite{KM} which modifies 
the Schubert decomposition given for $SU_n$ to apply to the Cartan models 
for the symmetric spaces.  We have given the Schubert decomposition for 
$SU_n$ in the previous section so we will consider the form it takes for both 
$SU_n/SO_n$ and $SU_{2n}/Sp_n$. \par
We again use the standard flag $0 \subset \C \subset \C^2 \subset \dots 
\subset \C^n$ and the same notation for pseudo-rotations as in \S 
\ref{S:sec3}.  
\vspace{2ex}
\flushpar
\subsection*{Schubert Decomposition for $SU_n/SO_n$} \hfill
\par
We consider an element of the Cartan model $\cC_n^{(sy)}$ for 
$SU_n/SO_n$.  If $B \in \cC_n^{(sy)}$ we have that $B \in SU_n$ and $B = 
B^T$.  By Lemma \ref{Lem2.1}, there is an orthonormal basis of real 
eigenvectors $x_i$ for $B$.   Hence, each $<x_i> \in \RP^{n-1}$.  Then $B$ 
can be written as a product of pseudo-rotations about complexifications of 
real hyperplanes $\C<x_i>^{\perp}$.  We will refer to such a 
pseudo-rotation $A_{(\theta, x)}$ for a real vector $x$ as an 
$\R$-pseudo-rotation.  There are two problems in trying to duplicate the 
reasoning used for the Schubert decomposition for $SU_n$.  First, there is 
no analogue of Lemma \ref{Lem3.1} for products of 
$\R$-pseudo-rotations.  Second, it need not be true that the ordered 
product of $\R$-pseudo-rotations $A_{(\theta, x_i)}$ is an element of 
$\cC_n^{(sy)}$ if the vectors $x_i$ are not mutually orthogonal.  \par
The solution obtained by Kadzisa-Mimura is to use instead  \lq\lq ordered 
symmetric factorizations\rq\rq\, by $\R$-pseudo-rotations.  Specifically it 
will be a product resulting from the successive application of Cartan 
conjugates by $\R$-pseudo rotations, which always yields elements of 
$\cC_n^{(sy)}$.  \par
  Then, in describing the Schubert decomposition for $SU_n/SO_n$, we are 
giving a version of that contained in \cite{KM}, except we again define maps 
from products of cones on real projective spaces whose open cells give the 
cell decomposition.  \par
Given an increasing sequence $m_1 < m_2 < \dots < m_k$ with $1 <  
m_1$ and $m_k \leq n$, which we denote by $\bm = (m_1, m_2, \dots , 
m_k)$ we define a map 
$$ \psi_{\bm}^{(sy)} : (C \RP^{m_1 - 1}) \times (C \RP^{m_2 - 1}) \times 
\cdots \times (C \RP^{m_k - 1}) \longrightarrow SU_n\, , $$
with $CX = (I \times X)/(\{0\} \times X)$ for $I = [0, 1]$, denoting the cone 
on $X$.  
This is given as follows: \par
First, we define a simpler map for $m \leq n$, $I = [0, 1]$ and a real line 
$L \subset \R^m$, 
$\tilde {\psi}_m^{(sy)} : C\RP^{m-1} \to SU_n$ defined by $\tilde 
{\psi}^{(sy)}_m(t, L) = A_{(\pi t, L_{\C})}$, with $L_{\C}$ denoting the 
complexification of the real line $L$.  Note this factors through the cone as 
$A_{(0, L_{\C})} = Id$, independent of $L$.  We will henceforth abbreviate 
this to $A_{(\pi t, L)}$.  
Then, we extend this to a map 
$$ \tilde{\psi}^{(sy)}_{\bm} : \prod_{i = 1}^{k} (C\RP^{m_i-1})\,\,  \longrightarrow 
\,\, SU_n $$
 defined by
\begin{align}
\label{Eqn4.1}
\tilde{\psi}^{(sy)}_{\bm}((t_1, L_1), \dots  , (t_k, L_k)) \,\,  &= \,\, A_{(-\pi 
\tilde{t}, e_1)}\cdot {\psi}_{m_1}(t_1, L_1) \cdot  {\psi}_{m_2}(t_2, L_2) 
\cdots {\psi}_{m_k}(t_k, L_k)  \notag \\ 
 \,\,  &= \,\, A_{(-\pi \tilde{t}, e_1)}\cdot A_{(\pi t_1, L_1)} \cdot A_{(\pi 
t_2, L_2)} \cdots A_{(\pi t_k, L_k)} \, . 
\end{align}
where $\tilde{t} = \sum_{j = 1}^{k}t_j$.  We note that the first factor 
$A_{(-\pi \tilde{t}, e_1)}$ ensures the product is in $SU_n$ as in the 
splitting for \eqref{CD2.1}.  
Then we define
\begin{equation}
\label{Eqn4.2}
\psi_{\bm}^{(sy)}((t_1, L_1), \dots  , (t_k, L_k)) \,\,  = \,\, 
\tilde{\psi}_{\bm}^{(sy)}((t_1, L_1), \dots  , (t_k, L_k)) \cdot \left( 
\tilde{\psi}_{\bm}^{(sy)}((t_1, L_1), \dots  , (t_k, L_k))\right)^T 
\end{equation}
We note that the RHS is the Cartan conjugate of $I$ by 
$\tilde{\psi}_{\bm}((t_1, L_1), \dots  , (t_k, L_k)) \in SU_n$ and thus is in 
the Cartan model $\cC_n^{(sy)}$.  It can also be
obtained by successively applying to $I$ the Cartan conjugates by the 
$A_{(\pi t_j, L_j)}$, for $j = k, k-1, \dots , 1, 0$, where we let $A_{(\pi t_0, 
L_0)}$ denote $A_{(-\pi \tilde{t}, e_1)}$ (each of these are, strictly 
speaking, Cartan conjugates for $U_n$ but their product is in $SU_n$).  
\par
We observe that each $C\RP^{m-1}$ has an open dense cell 
$$ E_m^{(sy)} = (0, 1) \times \{x = (x_1, \dots, x_m, 0 , \dots 0) : (x_1, 
\dots, x_m) \in S^{m-1} \mbox{ and } x_m > 0 \} $$
which is of dimension $m$.  Also, if $x = (x_1, \dots, x_m, 0 , \dots 0)$ with 
$x_m > 0$, then $x \in_{\min} \C^m$.  \par 
\par
We now introduce some notation and denote $\tilde{S}_{\bm}^{(sy)} = 
(C\RP^{m_1 - 1}) \times (C\RP^{m_2 - 1}) \times \cdots \times (C\RP^{m_k - 
1})$, the cell 
$$ E_{\bm}^{(sy)} \,\, = \,\, E_{m_1}^{(sy)} \times E_{m_2}^{(sy)} \times 
\cdots \times E_{m_k}^{(sy)}\, ,$$
 and $S_{\bm}^{(sy)} = \psi_{\bm}(E_{\bm}^{(sy)})$.  Then, 
$E_{\bm}^{(sy)}$ is an open dense cell in $\tilde{S}_{\bm}^{(sy)}$ with 
$\dim_{\R} E_{\bm}^{(sy)} =  | \bm | \overset{def}{=} \sum_{j = 1}^{k} m_j 
$.  Also, the image  $S_{\bm}^{(sy)} = \psi_{\bm}(E_{\bm}^{(sy)})$ consists 
of elements of $SU_n$ of {\em real Schubert type} $\bm$.  Furthermore, 
$\overline{S_{\bm}^{(sy)}} = \psi_{\bm}^{(sy)}(\tilde{S}_{\bm}^{(sy)})$.  
Then the results of Kadzisa-Mimura \cite[Thm 6.7]{KM} give the following 
Schubert decomposition of $SU_n/SO_n$.
\begin{Thm}
\label{Thm4.5}
The Schubert decomposition of $SU_n/SO_n$ has the following properties:
\begin{itemize}
\item[a)] $SU_n/SO_n$is the disjoint union of the $S_{\bm}^{(sy)}$ as $\bm 
= (m_1, \dots , m_k)$ varies over all increasing sequences with $1 < m_1$, 
$m_k \leq n$, and $0 \leq k \leq n-1$.
\item[b)] The map $\psi_{\bm}^{(sy)} : E_{\bm}^{(sy)} \to S_{\bm}^{(sy)}$ 
is a homeomorphism.
\item[c)] $(\overline{S_{\bm}^{(sy)}}\backslash S_{\bm}^{(sy)}) \subset 
\cup_{\bm^{\prime}} S_{\bm^{\prime}}^{(sy)}$, where the union is over all 
$S_{\bm^{\prime}}^{(sy)}$ with $\dim S_{\bm^{\prime}}^{(sy)} < \dim 
S_{\bm}^{(sy)}$.
\end{itemize}
\end{Thm}
Hence, the Schubert decomposition by the cells $S_{\bm}^{(sy)}$ is a cell 
decomposition of $SU_n/SO_n$.  We refer to the cells $S_{\bm}^{(sy)}$ as  
the {\em symmetric Schubert cells of $SU_n/SO_n$}.  We also refer to the 
factorization given by \eqref{Eqn4.2} for elements $B$ of $S_{\bm}^{(sy)}$ 
as the {\em ordered symmetric factorization} and the corresponding 
Schubert symbol is denoted by $\bm^{(sy)}(B)$.  \par
\begin{Remark}
Unlike the case of $SU_n$, in general the $\tilde{S}_{\bm}^{(sy)}$ do not 
carry a top dimensional fundamental class.  In the case of a simple Schubert 
symbol $(m_1)$, since $L$ is real, $A_{(\pi, L)}$ is the complexification of a 
real reflection about the real hyperplane $L_{\C}^{\perp}$  and hence it is 
its own inverse and transpose.  This is independent of $L$.  Then, 
\begin{align}
\label{Eqn4.3}
 \psi_{(m_1)}^{(sy)}(\pi, L_1) \,\, &=  \,\,  A_{(-\pi , e_1)}\cdot A_{(\pi , 
L_1)}\cdot A_{(\pi , L_1)}^T \cdot A_{(-\pi , e_1)}^T  \notag \\
\,\, &=  \,\, A_{(-\pi , e_1)}\cdot A_{(\pi , L_1)}\cdot A_{(\pi , L_1)}^{-1} 
\cdot A_{(-\pi , e_1)}^{-1} \,\, = \,\, Id
\end{align}
Thus, $\psi_{(m_1)}^{(sy)}( \{1\} \times \RP^{m_1-1}) = Id$ and so factors 
to give a map $\psi_{(m_1)}^{(sy)} : S\RP^{m_1-1} \to \cC^{(sy)}_n$.  
Hence, for the simple Schubert symbol $(m_1)$, 
$\overline{E_{(m_1)}^{(sy)}} = \psi_{(m_1)}^{(sy)}(S\RP^{m_1-1})$ has a 
fundamental class which is the image of the fundamental class of 
$S\RP^{m_1-1}$.  
\par
For a general symmetric Schubert symbol $\bm = \bm^{(sy)} = (m_1, m_2, 
\dots , m_k)$, if $(SU_n/SO_n)^{(\ell)}$ denotes the $\ell$-skeleton of 
$SU_n/SO_n$, then $\psi ^{(sy)}_{\bm}$ composed with the projection does factor 
through to give a map
$$ \tilde{\psi}_{\bm}^{(sy)\, \prime} : \prod_{i = 1}^{k} S\RP^{m_i-1} \to 
(SU_n/SO_n)/(SU_n/SO_n)^{(|\bm|-1)}\, .  $$
The product again carries a fundamental class and in \S \ref{S:sec5} we see 
how these images in homology correspond to generators.  
\end{Remark}
\vspace{2ex}
\flushpar
\subsection*{Schubert Decomposition for $SU_{2n}/Sp_n$} \hfill
\par
For the Schubert decomposition for $SU_{2n}/Sp_n$ we will largely follow 
\cite[\S 7]{KM}; except that for the geometric properties of Milnor fibers we 
will emphasize the use of the quaternionic structure on $\C^{2n}$.  We 
already have the complex structure giving multiplication by $\bi$.  We extend 
it to $\mbH$ by defining multiplication by $\bj$ by $\bj x = J_n \bar{x}$ for 
$x \in \C^{2n}$ with $\bar{x}$ complex conjugation (so $\bk x = \bi\bj x$).  
Then, it is a standard check (see e.g. 
\cite[\S 1.4.4]{GW}) that this defines a quaternionic action so $\C^{2n} 
\simeq \mbH^n$.  For this quaternionic structure, each subspace $\C^{2m}$ 
spanned by $\{e_1, \dots , e_{2m}\}$ is a quaternionic subspace. \par
 Let $<x , y > = x^T\cdot \bar{y}$ (for column vectors $x$ and $y$) denote 
the Hermitian inner product on $\C^{2n}$.  It has the following directly 
verifiable properties:
\begin{itemize}
\item[i)] multiplication by $J_n$ is $\mbH$-linear;
\item[ii)] $<\bj x , \bj y > = \overline{<x , y >}$;  and
\item[iii)] (by ii)) both $<x , \bj x > = 0$ and $<\bj x , y > = - \overline{<x 
, \bj y >}$.
\end{itemize}
\par
An element $B$ of the Cartan model for $SU_{2n}/Sp_n$ is characterized 
from \eqref{Eqn2.2} by $(BJ_n)^T = -BJ_n$.  so that $BJ_n$ is an element 
of $SU_{2n}$ and is skew-symmetric. 
This has the following consequence, which is basically equivalent to \cite[Thm 
3.4]{KM}. 
\begin{Lemma}
\label{Lem4.6}
If $B \in \cC_{2n}^{(sk)}$, the Cartan model for $SU_{2n}/Sp_n$, then 
\begin{itemize}
\item[a)] $B \bj x = \bj B^*x$; and 
\item[b)] if $B$ satisfies the condition in a), then the eigenspaces of $B$ are 
$\mbH$-subspaces.
\end{itemize}
\end{Lemma}
\begin{proof}
For a), this is a simple calculation.
$$ B \bj x \, = \, BJ_n\bar{x} \, = \, - (BJ_n)^T\bar{x} \, = \, -J_n^T B^T 
\bar{x} \, = \, J_n \overline{\bar{B}^T x} \, = \, J_n \overline{B^* x} \, = \, 
\bj B^*x\, .$$
\par
For b), we observe that if $ B x = \gl x$, then as $B \in SU_{2n}$, $B^* = 
B^{-1}$ and $|\gl | = 1$ so 
$$ B \bj x \, = \, \bj B^*x \, = \, \bj B^{-1}x \, = \, \bj \gl^{-1}x \, = \, J_n 
\overline{\gl^{-1}x} \, = \, \gl J_n \bar{x} \, = \, \gl \bj x\, .$$
Thus, the $\gl$-eigenspace of $B$ is invariant under multiplication by $\bj$. 
\end{proof}
\par
We will refer to a $B \in U_{2n}$ which satisfies the condition in a) of Lemma 
\ref{Lem4.6} as being $\mbH$*-linear.  To factor such a matrix, we use a 
version of pseudo-rotation for $\mbH^n$.  Given a quaternionic line $L 
\subset \C^{2n}$, let $L^{\perp}$ be the quaternionic hyperplane orthogonal 
to $L$.  We define an $\mbH$-pseudo-rotation by an angle $\theta$, 
$\tilde{A}_{(\theta, L)}$ which is the identity on $L^{\perp}$ and is 
multiplication by $e^{\iti \theta}$ on $L$.  It is $\C$-linear and can be 
checked to be $\mbH$*-linear.  If $x \in L$ is a unit vector, then by property 
iii), $\{ x, \bj x\}$ is an orthonormal basis for $L$.  Then, 
$\tilde{A}_{(\theta, L)}$ can be written as a product of 
pseudo-rotations $A_{(\theta, x)}A_{(\theta, \bj x)}$, which commute.  By 
the properties of pseudo-rotations, we have the following properties of 
$\mbH$-pseudo-rotations.
\begin{itemize}
\item[i)]  $\tilde{A}_{(\theta, L)}^* =  \tilde{A}_{(\theta, L)}^{-1} = 
\tilde{A}_{(-\theta, L)}$;
\item[ii)]   $\overline{\tilde{A}_{(\theta, L)}} = \tilde{A}_{(-\theta, 
\bar{L})}$, where $\bar{L}$ is the $\mbH$-line generated by $\bar{x}$; and
\item[iii)] $\tilde{A}_{(\theta, L)}^T = \tilde{A}_{(\theta, \bar{L})}$;
\item[iv)] $\det(\tilde{A}_{(\theta, L)}) = e^{2 \iti \theta}$;
\item[v)]  If $L \perp L^{\prime}$ then $\tilde{A}_{(\theta, L)}$ and 
$\tilde{A}_{(\theta, L^{\prime})}$ commute; 
\item[vi)]  $\tilde{A}_{(\theta, L)}$ is $\mbH$*-linear. 
\end{itemize}
\begin{proof}  All of i) - v) follow directly from the properties of 
pseudo-rotations.  For vi) we observe that $\tilde{A}_{(\theta, L)}$ is 
characterized as a unitary matrix which has $L$ for the eigenspace for 
$e^{\iti \theta}$ and $L^{\perp}$ as the eigenspace for the eigenvalue $1$.  
Thus, for vi), as both $L$ and $L^{\perp}$ are $\mbH$-subspaces we see 
$\tilde{A}_{(\theta, L)} \equiv Id$ on $L^{\perp}$ and for $x \in L$, 
$$ \tilde{A}_{(\theta, L)}( \bj x) \,\, = \,\, e^{\iti \theta} \bj x\,\, = \,\, \bj 
e^{-\iti \theta} x \,\, = \,\, \bj \tilde{A}_{(\theta, L)}^{-1}(x)\, .$$ 
As $\tilde{A}_{(\theta, L)}^* = \tilde{A}_{(\theta, L)}^{-1}$, we see that 
$\tilde{A}_{(\theta, L)}( \bj x) = \bj \tilde{A}_{(\theta, L)}^*(x)$ on each 
summand $L$ and $L^{\perp}$; hence they are equal.  
\end{proof}
In addition, we can give a unique representation of $\tilde{A}_{(\theta, L)}$ 
as an ordered product of pseudo-rotations.
\begin{Lemma}
\label{Lem4.11}
Given an $\mbH$-line $L \subset_{\min} \C^{2m}$, there is a unique unit 
vector $x \in L \cap \C^{2m-1}$ of the form $x = (x_1, \dots , x_{2m-1}, 
0)$ with $x_{2m-1} > 0$ so that $\bj x = (\bar{x_2}, -\bar{x_1}, \bar{x_4}, 
-\bar{x_3}, \dots, 0, -x_{2m-1})$.  Hence, $\tilde{A}_{(\theta, L)}$ can be 
uniquely written 
$A_{(\theta, x)}\cdot A_{(\theta, \bj x)}$.
\end{Lemma}
\begin{proof}
As $\dim_{\C} L = 2$. $\dim_{\C} (L \cap \C^{2m-1}) = 1$.  It is $\geq 1$, 
and otherwise it would be $2$, i.e. $L \subset \C^{2m-1}$.  Then, under the 
$\mbH$-linear projection $p : \C^{2m} \to \C^{2m}/\C^{2m-2}$ the image 
of $L$, which is an $\mbH$-subspace would have $\C$-dimension $1$, a 
contradiction.  \par
As $\dim_{\C} (L \cap \C^{2m-1}) = 1$, and $L \not \subset \C^{2m-2}$, 
we may find a unit vector $x \in L$ of the form $x^{\prime} = 
(x_1^{\prime}, \dots , x^{\prime}_{2m-1}, 0)$ with $x^{\prime}_{2m-1} \neq 
0$.  Multiplying $x^{\prime}$ by an appropriate unit complex number we 
obtain $x$ with $x_{2m-1} > 0$.  Then, $\bj x$ is as stated and so is 
$\tilde{A}_{(\theta, L)}$.  
\end{proof}
\subsection*{Whitehead-Type Ordered Factorization} \hfill 
\par
For an $\mbH$*-linear $B \in U_{2n}$, we may initially factor it as a product 
of $\mbH$-pseudo-rotations in a manner similar to the symmetric case as 
follows.  Each eigenspace $V_{\gl}$ of $B$ with $\gl = e^{\iti \theta} \neq 
1$ is an $\mbH$-subspace.  We choose the smallest $m^{\prime}_1$ so that 
$V_{\gl} \cap \C^{2m^{\prime}_1} \neq 0$, and hence is an $\mbH$-line 
$L_1^{(\gl)}$.  We successively repeat this for $(L_1^{(\gl)})^{\perp} \cap 
V_{\gl}$ and obtain an orthogonal decomposition $V_{\gl} = L_1^{(\gl)} 
\oplus L_2^{(\gl)} \cdots L_{k^{\prime}}^{(\gl)}$ with $L_j^{(\gl)} 
\subset_{\min} \C^{2m_j^{\prime}}$ and $m_1^{\prime} < m_2^{\prime} < 
\cdots  < m_{k^{\prime}}^{\prime}$.  Each $L_j^{(\gl)}$ gives an 
$\mbH$-pseudo-rotation $\tilde{A}_{(\theta, L_j^{(\gl)})}$.   We may do this 
for each eigenvalue $\gl \neq 1$.  Because different $L_j$ are orthogonal, 
the corresponding $\mbH$-pseudo-rotations commute.  Thus, we may factor 
$B$ as a product of $\mbH$-pseudo-rotations

\begin{equation}
\label{Eqn4.6a}
 B \,\, = \,\, \tilde{A}_{(\theta_1, L_1)} \cdot \tilde{A}_{(\theta_2, L_2)} 
\cdots \tilde{A}_{(\theta_k, L_k)}  
\end{equation}
where $L_j \subset_{min} \C^{2m_j}$, $1 \leq m_1 \leq m_2 \leq \cdots 
\leq m_k$, and several $\theta_j$ may be equal.  However, this is not an 
ordered factorization as some of the $m_j$ may be equal. \par
We would like to apply an analogue of the Whitehead Lemma \ref{Lem3.1} to 
products of $\mbH$-pseudo-rotations.  However, it is not possible to do so 
remaining in the category of $\mbH$-pseudo-rotations.  For example, if $B 
\in U_{2n}$ then $B\cdot \tilde{A}_{(\theta, L)}\cdot B^{-1}$ is a unitary 
transformation with $B( L)$ as the eigenspace for $e^{\iti \theta}$ and 
$B(L^{\perp}) = (B(L))^{\perp}$ as the eigenspace for the eigenvalue $1$.  
While $B( L)$ is a $2$-dimensional complex space, it need not be an 
$\mbH$-subspace.  \par
However, there is an alternate way to proceed which uses Lemma 
\ref{Lem4.11}.  We may uniquely decompose each $\mbH$-pseudo-rotation 
in \eqref{Eqn4.6a} into a product of pseudo-rotations about orthogonal 
planes which thus all commute so that \eqref{Eqn4.6a} may be rewritten
\begin{equation}
\label{Eqn4.6b}
 B \,\, = \,\, A_{(\theta_1, x_1)} \cdot A_{(\theta_2, x_2)} \cdots 
A_{(\theta_k, x_k)} \cdot A_{(\theta_k, \bj x_k)} \cdots A_{(\theta_2, \bj 
x_2)} \cdot A_{(\theta_1, \bj x_1)}
\end{equation}
\par
Then, we can progressively apply Whitehead\rq s Lemma to the factors 
$A_{(\theta_j, x_j)}$ beginning with the highest $j$ and proceeding left  to 
the lowest to obtain an ordered factorization for the product involving the 
$A_{(\theta_j, x_j)}$.  Then for each application of Whitehead\rq s Lemma 
for these, there is a corresponding application of it for the $A_{(\theta_j, 
\bj x_j)}$ from the left proceeding to the right using the following lemma. 
\par
\begin{Lemma}
\label{Lem4.7}
Given a relation between pseudo-rotations  
\begin{equation}
\label{Eqn4.7a}
 A_{(\theta, x)} \cdot A_{(\theta^{\prime}, x^{\prime})}\,\, = \,\, 
A_{(\theta_1, x_1)} \cdot 
A_{(\theta_2, x_2)}  \, ,
\end{equation}
 there is a corresponding relation 
\begin{equation}
\label{Eqn4.7b}
A_{(\theta^{\prime}, \bj x^{\prime})} \cdot A_{(\theta, \bj x)}\,\, = \,\, 
 A_{(\theta_2, \bj x_2)}  \cdot A_{(\theta_1, \bj x_1)}\, .
\end{equation}
\end{Lemma}
\begin{proof}
First, apply the transpose to each side of \eqref{Eqn4.7a} and then 
conjugate with $J_n$ to obtain 
\begin{equation}
\label{Eqn4.7c}
(J_n \cdot A_{(\theta^{\prime}, x^{\prime})}^T \cdot J_n^{-1}) \cdot (J_n 
\cdot  A_{(\theta, x)}^T \cdot J_n^{-1})\,\, = \,\, 
(J_n \cdot A_{(\theta_2, x_2)}^T \cdot J_n^{-1}) \cdot (J_n \cdot 
A_{(\theta_1, x_1)}^T \cdot J_n^{-1})
\end{equation}
Then, for any pseudo-rotation $A_{(\theta, x)}$, 
\begin{equation}
\label{Eqn4.7d}
J_n\cdot A_{(\theta, x)}^T\cdot J_n^{-1} = J_n\cdot A_{(\theta, 
\bar{x})}\cdot J_n^{-1} = A_{(\theta, J_n \bar{x})} = A_{(\theta, \bj x)}\, .
\end{equation}
Thus, applying \eqref{Eqn4.7d} to each product in \eqref{Eqn4.7c} yields
\eqref{Eqn4.7b}.
\end{proof}
\par
Then, by applying Whitehead\rq s Lemma successively to appropriate 
adjacent pairs $A_{(\theta_j, x_j)} \cdot A_{(\theta_{j^{\prime}}, 
x_{j^{\prime}})}$ and Lemma \ref{Lem4.7} to the corresponding pairs 
$A_{(\theta_{j^{\prime}}, \bj x_{j^{\prime}})} \cdot A_{(\theta_j, \bj x_j)}$
we may rewrite 
\begin{equation}
\label{Eqn4.8a}
 B \,\, = \,\, A_{(\theta_1^{\prime}, x_1^{\prime})} \cdot 
A_{(\theta_2^{\prime}, x_2^{\prime})} \cdots A_{(\theta_k^{\prime}, 
x_k^{\prime})} \cdot A_{(\theta_k^{\prime}, \bj x_k^{\prime})} \cdots 
A_{(\theta_2^{\prime}, \bj x_2^{\prime})} \cdot A_{(\theta_1^{\prime}, 
\bj x_1^{\prime})}
\end{equation}
with the $A_{(\theta_j^{\prime}, x_j^{\prime})}$ in increasing order and the 
$A_{(\theta_j^{\prime}, \bj x_j^{\prime})}$ in decreasing order.  
\par
\subsection*{Kadzisa-Mimura Ordered Skew-Symmetric Factorization} \hfill
\par
In fact, this is the skew-symmetric factorization of $B \in \cC^{(sk)}_m$ 
given by Kadzisa-Mimura.  We further rewrite \eqref{Eqn4.8a} using the 
properties of pseudo-rotations 
$\gs(A_i^{-1}) = A_{(\theta_i, \bj x_i)}$.  Hence, $B$ in \eqref{Eqn4.8a} can 
be rewritten either as
\begin{equation}
\label{Eqn4.12a}
 B \,\, = \,\,  \left(A_{(\theta_1, x_1)} \cdot A_{(\theta_2, x_2)} \cdots 
A_{(\theta_k, x_k)} \cdot J_n \cdot A_{(\theta_k, x_k)}^T \cdots 
A_{(\theta_1,  x_1)}^T \right) \cdot J_n^{-1}  
\end{equation}
or alternatively for each $A_j = A_{(\theta_j, x_j)}$ as 
\begin{equation}
\label{Eqn4.12b}
 B \,\, = \,\, A_1 \cdot A_2 \cdots A_k \cdot \gs(A_k^{-1}) \cdots 
\gs(A_1^{-1})\, .
\end{equation}
which is a Cartan conjugate of $I$ and hence belongs to $F^{(sk)\, c}_m$.  
\par
What we have not yet considered is the skew-symmetric Schubert symbol 
associated to this factorization.  We shall do so in giving in the next section 
the Kadzisa-Mimura algorithm for obtaining the ordered 
skew-symmetric factorization from the full Whitehead ordered 
factorization.\par 
We next define the maps for the cell decomposition of $SU_{2n}/Sp_n$ via 
the Cartan Model $\cC_{2n}^{(sk)}$.  
In describing the Schubert decomposition for $SU_{2n}/Sp_n$, we are giving 
a version that modifies that contained in \cite{KM} to associate to the 
Borel-Moore fundamental classes of products of suspensions of quaternionic 
projective spaces the Borel-Moore fundamental classes of the \lq\lq 
Schubert cycles\rq\rq\, obtained as the closures of the Schubert cells.  
However, unlike the general and symmetric cases, we cannot directly do this 
by expressing the closures of Schubert cells as the images of the products 
of suspensions of quaternionic projective spaces.  Instead we proceed through 
intermediate spaces which are products of suspensions of complex 
projective spaces.  \par
For any $m > 0$, we define via the quaternionic structure on $\C^{2m} \simeq \mbH^m$ a map $\chi_m : \C P^{2m - 2} \to \mbH P^{m - 1}$ by $\chi_m(L) = L + \bj 
L$ for complex lines $L \subset \C^{2m-1}$.  For a quaternionic line $Q 
\subset_{\min} \mbH^m$, $Q$ has a unique element $x = (x_1, \dots, 
x_{4(m-1)}, x_{4m-3}, 0) \in S^{4m-3} \subset \C^{2m-1}$ with 
$x_{4m-3} > 0$.  Then, 
$$\bj \, x = (\bar{x}_2, -\bar{x}_1, \bar{x}_4, -\bar{x}_3, \dots, 
\bar{x}_{4(m-1)}, -\bar{x}_{4m-5}, 0, -x_{4m-3})\, .$$  
Hence, the set of such $Q$ are parametrized by the cell $E^{4m-4}$ in 
$S^{4m-3}$ with $x_{4m-3} > 0$ (since $x_{4m-3} = \sqrt{1 - \sum_{j = 
1}^{4(m-1)} |x_j|^2}$ \,).  However, this cell also parametrizes the open 
dense subset of $L \in \C P^{2m - 2}$ with $L \subset_{\min} \C^{2m-1}$. 
The map $\chi_m$ acts as the identity on these parametrized cells of 
dimension $4m-4$, and the complements have lower dimensions.  We may 
then take the suspension $S\chi_m: S\C P^{2m - 2} \to S\mbH P^{m - 1}$, 
which now is a homeomorphism on the cell $(0, 1) \times E^{4m-4}$ of 
dimension $4m-3$.  Thus, $S\chi_{m\, *}$ sends the Borel-Moore 
fundamental class of $S\C P^{2m - 2}$ to that of $S\mbH P^{m - 1}$.  
\par
Then, given an increasing sequence $1 < m_1 < m_2 < \dots < m_k \leq n$, 
which we denote by $\bm^{(sk)} = (m_1, m_2, \dots , m_k)$, we may form 
the product map 
$$\tilde{\chi}_{\bm}^{(sk)} \,\, = \,\, S\chi_{m_1} \times S\chi_{m_2}  
\times \cdots \times  S\chi_{m_k}\, .$$
which again sends the Borel-Moore fundamental class of the product $S\C 
P^{2m_1 - 2} \times \cdots \times S\C P^{2m_k - 2}$ to that of $S\mbH 
P^{m_1 - 1} \times \cdots \times S\mbH P^{m_k - 1}$.  \par 
Then, the correspondence we give between the fundamental homology 
classes of $S\mbH P^{m_1 - 1} \times \cdots \times S\mbH P^{m_k - 1}$ 
and the Schubert cycles will be via the fundamental homology classes of 
$S\C P^{2m_1 - 2} \times \cdots \times S\C P^{2m_k - 2}$.  
\par
We do so by defining a map 
$$ \psi_{\bm}^{(sk)} : S \C P^{2m_1 - 2} \times S \C P^{2m_2 - 2} \times 
\cdots \times S \C P^{2m_k - 2} \longrightarrow \cC_m^{(sk)} . $$  
This is given as follows: \par  
$$\tilde{\psi}_m^{(sk)} : (I \times \C P^{2m_1-2}) \times (I \times \C 
P^{2m_2-2}) \times \cdots \times (I \times \C P^{2m_k-2}) \longrightarrow 
SU_n$$ 
is defined by 
\begin{align}
\label{Eqn4.13}
\tilde{\psi}_{\bm}^{(sk)}((t_1, L_1), \dots  , (t_k, L_k)) \,\,  &= \,\, A_{(-
2\pi \tilde{t}, e_1)}\cdot A_{(2\pi t_1, L_1)} \cdot  A_{(2\pi t_2, L_2)} 
\cdots A_{(2\pi t_k, L_k)}  \notag \\ 
 \,\,  &\cdot A_{(2\pi t_k,\, \bj L_k)} \cdots A_{(2\pi t_2,\, \bj L_2)} \cdot 
A_{(2\pi t_1,\, \bj L_1)} \cdot A_{(-2\pi \tilde{t}, - e_3)}\, , 
\end{align}
where $\tilde{t} = \sum_{j = 1}^{k}t_j$.  We note that the product is of the 
form \eqref{Eqn4.8a} and hence \eqref{Eqn4.12b}.  Also, the first and last 
factors $A_{(-2\pi \tilde{t}, e_1)}$ and  $A_{(-2\pi \tilde{t}, -e_3)}$ ensure 
the product is in $SU_n$ as in the splitting for \eqref{CD2.1}.  \par
Since $A_{(0, L)} = A_{(2\pi, L)} = I_n$ independent of a complex line $L 
\subset \C^{2m-1}$, \eqref{Eqn4.13} descends to a map 
$$\psi_m^{(sk)} : S\C P^{2m_1-2} \times S \C P^{2m_2-2} \times \cdots \times 
S \C P^{2m_k-2} \longrightarrow \cC_m^{(sk)}\, .$$
\par
As remarked above, each $S \C P^{2m_j-2}$ has an open dense cell of 
dimension $4m_j-3$ which we denote by
\begin{align*}
 E_{m_j}^{(sk)}\,\, &= \,\, (0, 1) \times \{x = (x_1, \dots, x_{4(m_j-1)}, 
x_{4m_j-3}, 0 , \dots 0)  \\ 
&\qquad \qquad : (x_1, \dots, x_{4(m_j-1)}, x_{4m_j-3}), 0) \in S^{4m_j-3} 
\mbox{ and } x_{4m_j-3} > 0 \} 
\end{align*}
 and we conclude $\mbH<x> \, \subset_{\min} \C^{2m_j}$.  \par 
We now introduce some notation and denote 
$$\tilde{S}_{\bm}^{(sk)} = S \C P^{2m_1 - 2} \times S \C P^{2m_2 - 2} 
\times \cdots \times S \C P^{2m_k - 2}\, .$$
Also, we consider the corresponding cell $E_{\bm}^{(sk)} = E_{m_1}^{(sk)} 
\times E_{m_2}^{(sk)} \times \cdots \times E_{m_k}^{(sk)}$, and the image 
$S_{\bm}^{(sk)} = \psi_{\bm}^{(sk)}(E_{\bm}^{(sk)})$ in $\cC_{2n}^{(sk)}$.  
Then, $E_{\bm}^{(sk)}$ is an open dense cell in $\tilde{S}_{\bm}^{(sk)}$ with 
$\dim_{\R} E_{\bm}^{(sk)} = \sum_{j = 1}^{k} (4m_j - 3) =  4 | \bm^{(sk)} | 
- 3k = 4 | \bm^{(sk)} | - 3\ell(\bm^{(sk)})$ for $| \bm^{(sk)} | =\sum_{j = 
1}^{k} m_j$ (and $\ell(\bm^{(sk)}) = k$).  Also, the image $S_{\bm}^{(sk)} = 
\psi_{\bm}^{(sk)}(E_{\bm}^{(sk)})$ consists of elements of $\cC_{2n}^{(sk)}$ 
of skew Schubert type $\bm$.  Furthermore, $\overline{S_{\bm}^{(sk)}} = 
\psi_{\bm}^{(sk)}(\tilde{S}_{\bm}^{(sk)})$.  Then the results of 
Kadzisa-Mimura \cite[Thm 8.7]{KM} give the following Schubert 
decomposition of $SU_{2n}/Sp_n$.
\begin{Thm}
\label{Thm4.13}
The Schubert decomposition of $SU_{2n}/Sp_n$ has the following properties 
via the diffeomorphism $SU_{2n}/Sp_n \simeq \cC_{2n}^{(sk)}$:
\begin{itemize}
\item[a)] $SU_{2n}/Sp_n$ is the disjoint union of the $S_{\bm}^{(sk)}$ as 
$\bm = \bm^{(sk)} = (m_1, \dots , m_k)$ varies over all increasing 
sequences with $1 < m_1 < \cdots < m_k \leq n$, and $0 \leq k \leq n-1$.
\item[b)] The map $\psi_{\bm}^{(sk)} : E_{\bm}^{(sk)} \to S_{\bm}^{(sk)}$ 
is a homeomorphism.
\item[c)] $(\overline{S_{\bm}^{(sk)}}\backslash S_{\bm}^{(sk)}) \subset 
\cup_{\bm^{\prime}} S_{\bm^{\prime}}^{(sk)}$, where the union is over all 
$S_{\bm^{\prime}}^{(sk)}$ with $\dim S_{\bm^{\prime}}^{(sk)} < \dim 
S_{\bm}^{(sk)}$.
\end{itemize}
\end{Thm}
Hence, the Schubert decomposition by the cells $S_{\bm}^{(sk)}$ gives a 
corresponding cell decomposition of $SU_{2n}/Sp_n$.  The cells 
$S_{\bm}^{(sk)}$ will be referred to as the {\em skew-symmetric Schubert 
cells of $SU_{2n}/Sp_n$ or $\cC_{2n}^{(sk)}$}.  We note that 
$\overline{S_{\bm}^{(sk)}}$ has a Borel-Moore fundamental class which we 
refer to as a {\em skew-symmetric Schubert cycle}. It is the image of the 
Borel-Moore fundamental class of the \lq\lq singular manifold\rq\rq 
$\tilde{S}_{\bm}^{(sk)}$.  It corresponds to the Borel-Moore fundamental 
class of the associated product of suspensions of quaternionic projective 
spaces.  We describe in \S \ref{S:sec5} the homology of $SU_{2n}/Sp_n$ 
and $\cC_{2n}^{(sk)}$ in terms of these skew-symmetric Schubert cycles.  
Furthermore, for $m = 2n$ the relation of $\cC_m^{(sk)}$ with $F_m^{(sk)\, c}$ allows us to give a Schubert decomposition for the Milnor fiber.  
\par
\begin{Remark}
\label{Rem4.14}
If in the initial factorization of $B \in \cC^{(sk)}_{2n}$ given in \eqref{Eqn4.6a} 
into a product of $\mbH$-pseudo-rotations, the orders for all of the 
$L^{(\gl_{\ell})}_j$ are all distinct then $1 < m_1 < m_2 < \cdots < m_k$.  
By the commutativity of the $\mbH$-pseudo-rotations, we may arrange 
them in increasing order and obtain \eqref{Eqn4.8a} without using 
Whitehead\rq s Lemma.  Hence, the skew-symmetric Schubert symbol is 
given by $\bm^{(sk)} = (m_1, m_2, \cdots , m_k)$, which would be the 
corresponding Schubert symbol in the quaternionic Grassmannian.  In general, 
the use of Whitehead\rq s Lemma has the effect of twisting the $\mbH$-lines 
which then again reappear from the form of the skew-symmetric factorization.
\end{Remark} 
\par
\section{Schubert Decomposition for Milnor Fibers }
\label{S:sec5} 
\par
In this section we apply the results giving the Schubert decomposition for 
the associated symmetric spaces providing compact models for the global 
Milnor fibers.  We first give the form that the Schubert decomposition gives 
for the specific Cartan models, and extending these to the Milnor fibers 
themselves.  Second, in doing this we give an algorithm due to 
Whitehead and Kadzisa-Mimura for identifying for a given matrix in the global 
Milnor fiber the Schubert cell to which it belongs.  Third, we will see the form 
that the Schubert decomposition takes for the global Milnor fibers using 
Iwasawa decomposition.  
\subsection*{Whitehead-Kadzisa-Mimura Algorithm for Identifying 
Schubert Cells} \hfill 
\par
The algorithm given by Kadzisa-Mimura \cite{KM} for the ordered 
factorizations of matrices in the various Cartan models uses the ordered 
factorization for $SU_m$ based on the work of Whitehead \cite{W} as 
developed by Miller \cite{Mi} and Yokota \cite{Y}.  They cleverly combine the 
uniqueness of the factorization for $U_m$ (and $SU_m$) and the Cartan 
conjugacy for the Cartan models to give the symmetric, respectively 
skew-symmetric, factorizations for the cases of $SU_m/SO_m$ and for $m 
= 2n$, $SU_{2n}/Sp_n$.   We explain this algorithm as it will apply to the 
compact models for global Milnor fibers and then for the global Milnor fibers 
themselves. \par
\par
An element of any of the Cartan models is a matrix $B \in SU_m$ for 
appropriate $m$.  Thus, by Lemma \ref{Lem3.3} we may obtain an ordered 
factorization by pseudo-rotations except with decreasing order for $B$.  
\begin{equation}
\label{Eqn5.1}
B \, \, = \,\,  A_{k} \cdot A_{k-1} \cdots A_{1}\, ,
\end{equation}
where $A_j = A_{(\theta_j, x_j)}$ with the $\{x_j \}$ a set of unit vectors 
with $x_j \in_{\min} \C^{m_j}$ and $1 \leq m_1 < m_2 < \cdots < m_k \leq 
m$, and $\theta_i \not \equiv 0  \, \mod 2\pi$ for each $i$.  In addition, if 
$m_1 = 1$ then the Schubert symbol is $\bm = (m_2, \dots , m_k)$.  Now 
from \eqref{Eqn5.1} we describe how to obtain either the symmetric or 
skew-symmetric ordered factorizations as obtained by Kadzisa-Mimura.  
\par
\subsubsection*{Ordered Symmetric Factorizations for $\cC^{(sy)}$} 
\par
As $B \in \cC^{(sy)}$, $\gs(B^{-1}) = B$.  Hence, as $\gs(B^{-1}) = 
\overline{B^{-1}} = B^T$, we obtain from \eqref{Eqn5.1}
$$  A_{k} \cdot A_{k-1} \cdots A_{1}   \, \, = \,\,  A_{1}^T \cdot A_{2}^T 
\cdots A_{k}^T \, . $$
As each $A_j = A_{(\theta_j, x_j)}$, $A_j^T = A_{(\theta_j, \bar{x}_j)}$ is a 
pseudo-rotation with $\bar{x}_j \in_{\min} \C^{m_j}$.  Thus, it follows by 
Lemma \ref{Lem3.4} that $A_1 = A_1^T$ and $x_1$ is real.  
Let $C_1 = A_{(\frac{\theta_1}{2}, x_1)}$.  We can write $A_1 = C_1 \cdot 
C_1$, and as $A_{(\theta_1, x_1)}$ is a pseudo-rotation about a real 
hyperplane, so is $C_1$.  Hence, $C_1 = C_1^T$ and $\gs(C_1) = C_1^*$.   
Then, from \eqref{Eqn5.1} since
\begin{equation}
\label{Eqn5.2}
B \, \, = \,\,  A_{k} \cdot A_{k-1} \cdots A_{1}\, ,
\end{equation} 
we have 
\begin{align}
\label{Eqn5.3}
C_1^*\cdot B \cdot \gs(C_1) \, \, &= \,\, (C_1^*\cdot A_{k}\cdot A_{k-1} 
\cdots A_{2} \cdot C_1) \cdot C_1 \cdot \gs(C_1)   \notag  \\
 \, \, &= \,\, (C_1^*\cdot A_{k}\cdot C_1)\cdot (C_1^*\cdot A_{k-1}\cdot 
C_1) \cdots (C_1^*\cdot A_{2} \cdot C_1)    \notag  \\
\, \, &= \,\,  A_{k}^{(2)} \cdot A_{k-1}^{(2)} \cdots A_{2}^{(2)}
\end{align}
where each $A_{j}^{(2)} = C_1^*\cdot A_{j}\cdot C_1$ is again a 
pseudo-rotation $A_{(\theta_j, x_j^{(2)})}$, with $x_j^{(2)} =  
C_1^{-1}(x_j)$ satisfying $x_j^{(2)}  \in_{\min} \C^{m_j}$ as $C_1 \equiv 
Id$ on $(\C^{m_1})^{\perp}$.  \par
Also, the LHS of \eqref{Eqn5.3} is the Cartan conjugate of the symmetric 
matrix $B$ and so is still symmetric (and in $SU_n$),
except now it is a product of $k-1$ pseudo-rotations with Schubert symbol 
$(m_k, \dots , m_2)$.  Thus we can inductively repeat the argument to 
write.
$$  C_j^* \cdots C_2^* \cdot C_1^*\cdot B \cdot \gs(C_1)\cdot 
\gs(C_2)\cdots \gs(C_j) \, \, = \,\, A_{k}^{(j+1)} \cdot A_{k-1}^{(j+1)} 
\cdots A_{j+1}^{(j+1)}  $$
which has Schubert symbol $(m_{j+1}, \dots , m_k)$.  After $k-1$ steps we 
obtain
\begin{equation}
\label{Eqn5.4} 
C_{k-1}^* \cdots C_2 ^*\cdot C_1^*\cdot B \gs(C_1) \cdot 
\gs(C_2)\cdots \gs(C_{k-1}) \, \, = \,\, A_{k}^{(k)}\, ,  
\end{equation}
with $A_{k}^{(k)} = A_{(\theta_k, x_k^{(k)})}$ for $x_k^{(k)} \in_{\min} 
\C^{m_k}$.  The last step then allows us to rewrite \eqref{Eqn5.4} as
\begin{equation}
\label{Eqn5.5} 
 B \, \, = \,\, C_1 \cdots C_{k-1} \cdot C_k\cdot \gs(C_k^*)\cdot \gs(C_{k-
1}^*)\cdots \gs(C_1^*)\, ,  
\end{equation}
which gives the ordered symmetric factorization.  \par
We obtain as a corollary of the algorithm
\begin{Corollary}
\label{Cor5.6}
If $B \in F_m^{(sy)\, c} = \cC^{(sy)}_m$, and has increasing Schubert symbol 
$\bm = (m_1, \dots , m_k)$, then the symmetric factorization has the same 
Schubert symbol $\bm^{(sy)} = \bm$.
\end{Corollary}
\subsubsection*{Ordered Skew-symmetric Factorizations for $\cC^{(sk)}_m$} 
\par
The algorithm for $\cC^{(sk)}_m$, with $m = 2n$, is very similar and depends on the following 
lemma, see \cite[Lemma 7.2]{KM}. 
\begin{Lemma}
\label{Lem5.7}
If $B \in (U_{2n} \cap Sk_{m}(\C))\cdot J_n^{-1}$, with $m = 2n$, has a 
factorization as in \eqref{Eqn5.1}, then: $k$ is even, $m_1$ is odd, $m_2 = 
m_1 + 1$, and $A_2 = \gs(A_1^*)$.
\end{Lemma}
Here $\gs(A) = J_n\cdot \overline{A} \cdot  J_n^{-1}$ and $A_1 = 
A_{(\theta_1, x_1)}$ with $x_1 \in_{\min} \C^{m_1}$, for which we may 
arrange $x_1 = (x_{1, 1}, \dots , x_{1, m_1})$ with $x_{1, m_1} > 0$.  
Then, by properties of pseudo-rotations 
 $A_2 = \gs(A_1^*) = A_{(\theta_1, \bj x_1)}$ (hence, $A_2\cdot A_1$ is 
an $\mbH$-pseudo-rotation and $A_1$ and $A_2$ commute).  We may then 
rewrite \eqref{Eqn5.1} as 
\begin{align}
\label{Eqn5.8}
A_1^*\cdot B \cdot \gs(A_1) \, \, &= \,\, A_1^*\cdot A_{k}\cdot A_{k-1} 
\cdots  A_{3} \cdot A_1 \cdot \gs(A_1^*) \cdot \gs(A_1)   \notag  \\
 \, \, &= \,\, (A_1^*\cdot A_{k}\cdot A_1)\cdot (A_1^*\cdot A_{k-1}\cdot 
A_1) \cdots  (A_1^*\cdot A_{3} \cdot A_1)    \notag  \\
\, \, &= \,\,  A_{k}^{(2)} \cdot A_{k-1}^{(2)} \cdots A_{3}^{(2)}
\end{align}
where each $A_{j}^{(2)} = A_1^*\cdot A_{j}\cdot A_1$ is again a 
pseudo-rotation $A_{(\theta_j, x_j^{(2)})}$, with $x_j^{(2)} =  
A_1^{-1}(x_j)$ satisfying $x_j^{(2)}  \in_{\min} \C^{m_j}$ as $A_1 \equiv 
Id$ on $(\C^{m_1})^{\perp}$.  \par
Also, the LHS of \eqref{Eqn5.8} is the Cartan conjugate of 
$B$ for which $B\cdot J_n$ is skew-symmetric (and in $U_{2n}$); and so it  
also has these properties, except now it is a product of $k-2$ 
pseudo-rotations with Schubert symbol $(m_k, \dots , m_3)$.  Thus we can 
inductively repeat the argument.  After $\frac{k}{2}$ steps we obtain a 
factorization in the form
\begin{align}
\label{Eqn5.9} 
 B \, \, &= \,\, A_{(\theta_1, x_1^{\prime})} \cdots A_{(\theta_r, 
x_r^{\prime})}\cdot \gs(A_{(\theta_r, x_r^{\prime})}^*) \cdots 
\gs(A_{(\theta_1, x_1^{\prime})}^*)\, ,  \notag  \\ 
\, \, &= \,\,  A_{(\theta_1, x_1^{\prime})} \cdots A_{(\theta_r, 
x_r^{\prime})}\cdot A_{(\theta_r, \,\bj x_r^{\prime})} \cdots 
A_{(\theta_1, \,\bj x_1^{\prime})}\, .
\end{align}
Here $k = 2r$, and each $\mbH <  x_r^{\prime} > \subset_{\min} 
\C^{2m_j}$.  This gives the ordered skew-symmetric factorization.  
By \eqref{Eqn4.7d} we may write each 
$A_{(\theta_j, \bj x_j^{\prime})} = J_n\cdot A_{(\theta_j, 
x_j^{\prime})}^T \cdot J_n^{-1}$, and then by \eqref{Eqn4.12a} we may alternately 
write \eqref{Eqn5.9} in the form 
\begin{equation}
\label{Eqn5.9b} 
B \, \, = \,\, A_{(\theta_1, x_1^{\prime})} \cdots A_{(\theta_r, 
x_r^{\prime})}\cdot J_n \cdot A_{(\theta_r, x_r^{\prime})}^T \cdots 
A_{(\theta_1, x_1^{\prime})}^T \cdot J_n^{-1}\, .
\end{equation}
\par
We obtain as a corollary of the algorithm.
\begin{Corollary}
\label{Cor5.10}
If $B \in \cC^{(sk)}_m = F_m^{(sk)\, c}\cdot J_n^{-1}$ (with $m = 2n$), then it 
has an increasing Schubert symbol of the form 
$\bm = (2m_1-1, 2m_1 , 2m_2-1, 2m_2 , \dots , 2m_r-1, 2m_r)$ with 
$1 < m_1 < m_2 , \cdots < m_r \leq n$.  Then 
the ordered skew-symmetric factorization has the skew-symmetric 
Schubert symbol $\bm^{(sk)} = (m_1, m_2 , \dots , m_r)$. 
\end{Corollary}
\par
To use the preceding results for the global Milnor fibers, we use in each case 
the Iwasawa decomposition, which is given for $SL_n$ by the 
Gram-Schmidt process, to determine the Schubert cell decomposition.
\par
\subsection*{Global Milnor Fibers for the Variety of Singular 
$m \times m$-Matrices} \hfill 
\par
This is the simplest case and was essentially covered in Proposition 
\ref{Prop1.1}.  Given $B \in F_m$, the global Milnor fiber, we have $F_m = 
SL_m(\C)$.  To obtain its representation in the Iwasawa decomposition 
$SL_m(\C) = SU_m\cdot A_m\cdot N_m$ where $A_m$ denotes the group 
of diagonal matrices with positive entries, and $N_m$ is the nilpotent group 
of upper triangular complex matrices with $1$\rq on the diagonal.  We may 
apply the Gram-Schmidt process to the columns of $B$ to obtain $B = 
A\cdot C$, where $A$ is unitary and $C$ is upper triangular  with positive 
entries on the diagonal.  As $\det(B)= 1$, $\det(A)$ is a unit complex 
number, and $\det(C) > 0$; it follows that both $\det(A)= \det(C) = 1$; 
thus, $C$ belongs to $\Sol_m = A_m\cdot N_m$. Then by applying the 
method of \S 3 for giving an ordered factorization for $A$ gives the 
Schubert symbol for $A$, which we shall also use for $B$.  Thus, we may 
describe the Schubert decomposition for the global Milnor fiber $F_m$ as 
follows.
\begin{Thm}
\label{5.13}
The Schubert decomposition of the global Milnor fiber $F_m$ for the variety 
of $m \times m$ general complex matrices is given, 
via the diffeomorphism with $SL_m(\C)$, by the disjoint union of the {\em 
Schubert cells} $S_{\bm}\cdot Sol_m$ where the $S_{\bm}$ are the 
Schubert cells of $SU_m$ for all Schubert symbols $\bm = (m_1, \dots , 
m_k)$ with $1 < m_1 < \dots < m_k \leq m$. 
\end{Thm}
\par
\subsection*{Global Milnor Fibers for the Variety of Singular 
$m \times m$-Symmetric Matrices} \hfill
\par
If $B \in F_m^{(sy)}$, then we want to relate $B$ to a matrix $C \in F_m^{(sy)\, 
c} = SU_m \cap Sym_m(\C) = \cC^{(sy)}_m$.  As $B$ is symmetric and 
$\det(B) = 1$, as in \cite[Table 1]{D3} we may diagonalize the quadratic 
form $X^T\cdot B\cdot X$, for column vectors $X$ so there is a $C \in 
SL_m(\C)$ so that $(CX)^T\cdot B\cdot CX = X^T\cdot X$. Thus, $C^T\cdot 
B\cdot C = I_m$ or 
$B = (C^{-1})^T\cdot C^{-1}$.  Then, by Iwasawa decomposition $C^{-1} = 
A\cdot E$, with $A \in SU_m$ and $E \in Sol_m$.  Then, $B = 
E^T\cdot(A^T\cdot A)\cdot E$, and $A^T\cdot A \in \cC_m^{(sy)}$.  If 
$\bm = (m_1, m_2, \dots , m_k)$ is the Schubert symbol for $\tilde{A} = 
A^T\cdot A$, it is also the symmetric Schubert symbol and so $\tilde{A} =  
A^T\cdot A \in S^{(sy)}_{\bm}$ and conversely.  \par
We let $Sol_m^T$ denote the group of lower triangular complex matrices $E$ 
with positive entries on the diagonal and $\det(E) = 1$.  Then, there is the 
action of $Sol_m^T$ on $\cC_m^{(sy)}$ as follows : 
$$Sol_m^T \times \cC_m^{(sy)} \to \cC_m^{(sy)}\qquad \mbox{sending} 
\qquad (E,\tilde{A}) \mapsto E\cdot \tilde{A}\cdot E^T\, .  $$  
Then, the action applied to each Schubert cell $S^{(sy)}_{\bm}$ gives by 
Proposition \ref{Prop1.1} the Schubert cell for $F_m^{(sy)}$ which we denote 
by $Sol_m^T\cdot (S^{(sy)}_{\bm})$.  Combining this with Theorem 
\ref{Thm4.5}  we obtain  
\begin{Thm}
\label{Thm5.14}
The Schubert decomposition of the global Milnor fiber $F_m^{(sy)}$ for the 
variety of $m \times m$ symmetric complex matrices is given by 
the disjoint union of the {\em symmetric Schubert cells} $Sol_m^T\cdot 
(S^{(sy)}_{\bm})$ for $S_{\bm}^{(sy)}$ the symmetric Schubert cells of 
$SU_m/SO_m$ for all symmetric Schubert symbols $\bm^{(sy)} = (m_1, 
\dots , m_k)$ with $1 < m_1 < \dots < m_k \leq m$.  \par
Furthermore, the preceding algorithm using ordered factorization gives the 
symmetric Schubert symbol for a given matrix in $F_m^{(sy)}$.  
\end{Thm}
\par
\subsection*{Global Milnor Fibers for the Variety of Singular 
$m \times m$ Skew-Symmetric Matrices} \hfill
\par
For the case of $B \in F_m^{(sk)}$ with $m = 2n$, we follow an analogous 
argument to the preceding. We first want to relate $B$ to a matrix $C \in 
F_m^{(sk)\, c} = SU_m \cap Sk_m(\C)$, and then use the relation $F_m^{(sk)\, 
c}\cdot J_n^{-1}  = \cC^{(sk)}_m$ to determine the skew-symmetric 
factorization for $C\cdot J_n^{-1}$ to determine its skew-symmetric 
Schubert type.  \par
As $B$ is skew-symmetric with $\Pf(B) = 1$, as in \cite[Table 1]{D3} we 
may block diagonalize the quadratic form $X^T\cdot B\cdot X$, for column 
vectors $X$ so there is a $C \in SL_m(\C)$ so that $(CX)^T\cdot B\cdot CX 
= X^T\cdot J_n\cdot X$. Thus, $C^T\cdot B\cdot C =  J_n$ or 
$B = (C^{-1})^T\cdot J_n\cdot C^{-1}$.  Then, we again apply Iwasawa 
decomposition $C^{-1} = A\cdot E$, with $A \in SU_m$ and $E \in Sol_m$.  
Then, $B = E^T\cdot(A^T\cdot J_n\cdot A)\cdot E$, and $\tilde{A} = 
A^T\cdot J_n\cdot A \in  SU_m \cap Sk_m(\C)$.  It follows $\tilde{A}\cdot 
J_n^{-1} \in \cC_m^{(sk)}$.  The Schubert symbol $\bm = (2m_1-1, 2m_1, 
2m_2-1, 2m_2, \dots , 2m_k-1, 2m_k)$ for $\tilde{A}\cdot J_n^{-1}$ is 
obtained from the ordered factorization of $\tilde{A}\cdot J_n^{-1}$.  By 
\eqref{Eqn5.9b}, this may be alternatively written as a skew-symmetric 
factorization of $\tilde{A}$
\begin{equation}
\label{Eqn5.15b} 
\tilde{A} \, \, = \,\, A_{(\theta_1, x_1^{\prime})} \cdots A_{(\theta_{k}, 
x_{k}^{\prime})}\cdot J_n \cdot A_{(\theta_{k}, x_{k}^{\prime})}^T 
\cdots A_{(\theta_1, x_1^{\prime})}^T\, .
\end{equation}
By Corollary \ref{Cor5.10}, $\bm^{(sk)} = (m_1, m_2, \dots , m_k)$ is the 
skew-symmetric Schubert symbol.  Then, under the map $\cC_m^{(sk)} \to 
F_m^{(sk)}$ given by right multiplication by $J_n$, i.e.  
$\tilde{A}\cdot J_n^{-1} \mapsto \tilde{A} \in SU_m \cap Sk_m(\C) = 
F_m^{(sk)}$, we have $S^{(sk)}_{\bm}$ mapping diffeomorphically to 
$S^{(sk)}_{\bm}\cdot J_n \subset F_m^{(sk)}$.  
Hence, we again use the action of $Sol_m^T$ but on $F_m^{(sk)}$ given by : 
$$Sol_m^T \times F_m^{(sk)} \to F_m^{(sk)}\qquad \mbox{sending} \qquad 
(E,\tilde{A}) \mapsto E\cdot \tilde{A}\cdot E^T\, .  $$  
Then, from the action applied to each Schubert cell $S^{(sk)}_{\bm}$ after 
right multiplication by $J_n$ gives by Proposition \ref{Prop1.1} the Schubert 
cell for $F_m^{(sk)}$ which we denote by $Sol_m^T\cdot 
(S^{(sk)}_{\bm}\cdot J_n)$.  Combining this with Theorem \ref{Thm4.5}  we 
obtain  
\begin{Thm}
\label{5.15}
The Schubert decomposition of the global Milnor fiber $F_m^{(sk)}$ for the 
variety of $m \times m$ skew-symmetric complex matrices (with $m = 2n$) 
is given by the disjoint union of the {\em skew-symmetric Schubert cells} 
$Sol_m^T\cdot (S^{(sk)}_{\bm}\cdot J_n)$ corresponding to the 
skew-symmetric Schubert cells $S^{(sk)}_{\bm}$ of $\cC_m^{(sk)}$, for all 
skew-symmetric Schubert symbols $\bm^{(sk)} = (m_1, \dots , m_k)$ with 
$1 < m_1 < \dots < m_k \leq n$.  \par
Furthermore, the preceding algorithm using ordered factorization gives the 
associated skew-symmetric Schubert symbol for a given matrix in 
$F_m^{(sk)}$.  
\end{Thm}
\par
\section{Representation of the Dual Classes in Cohomology}
\label{S:sec6} 
Having given the Schubert decomposition for the global Milnor fibers in terms 
of the corresponding Cartan models, we now consider how the Schubert 
decomposition corresponds to the (co)homology of the global Milnor fibers as 
given in \cite{D3}, which was deduced from that of the corresponding 
symmetric spaces.  We will refer to the closures of the Schubert cells in 
each case as {\em Schubert cycles} of the appropriate type.  We shall see 
that for both the general and skew-symmetric cases the Schubert cycles 
are cycles whose fundamental classes define 
$\Z$-homology classes.  For the symmetric case, the symmetric Schubert 
cycles are only $\mod 2$-cycles which define unique $\Z/2\Z$-homology 
classes.  The situation is somewhat similar to that for real Grassmannians 
where the $\Z/2\Z$-cohomology classes correspond to real Schubert cycles, 
while the rational classes are more difficult to identify in terms of the 
Schubert decomposition. 
\par
This identification is made using the standard method (see e.g. \cite[Chap. 
IX, \S 4]{Ma}) for computing the (co)homology of a finite CW-complex $X$ 
with skeleta $\{ X^{(k)}\}$ with coefficient ring $R$ using the finite algebraic 
complex $C_k(\{ X^{(k)}\}) = H_k(X^{(k)}, X^{(k-1)}; R)$, with boundary map 
given by the boundary map for the exact sequence of a triple.  Then, 
$\rk_R(C_k(\{ X^{(k)}\}))$ equals the number of cells $q_k$ of dimension 
$k$.  Thus, $\rk_R H_k(X; R) \leq q_k$ with equality iff the closures of the 
cells of dimension $k$ give a free set of generators for $H_k(X; R)$.  
Likewise the cohomology is computed from the complex $C^k(\{ X^{(k)}\}) = 
H^k(X^{(k)}, X^{(k-1)}; R)$ using the coboundary map for the exact sequence 
of a triple in cohomology. 
\par
  
\subsection*{Milnor Fiber for the Variety of Singular 
$m \times m$-Matrices} \hfill 
\par
We consider the Schubert decomposition for $F_m$ obtained from that for the compact model $F_m^{c} = SU_m$ as a result of Theorem \ref{5.13}.  
Then, the homology of $SU_m$ can be computed from the algebraic complex 
with basis formed from the Schubert cells $S_{\bm}$.  By a result of Hopf, 
the homology of $SU_m$ (which is isomorphic as a graded $\Z$-module to its 
cohomology) is given as a graded $\Z$-module by

$$  H_*(SU_m; \Z) \,\,  \simeq \,\, \gL^*\Z \langle s_3, s_5, \dots , 
s_{2m-1} \rangle \, . $$
where $s_{2j-1}$ has degree $2j-1$.  Then, a count shows that $H_q(SU_n; 
\Z)$ is spanned by $s_{2m_1-1}\cdot s_{2m_2 - 1}\cdots s_{2m_k-1}$ 
where $1 < m_1 < m_2 < \cdots < m_k \leq m$ and  $q = \sum_{j = 1}^{k} 
(2m_j -1)$.  This equals the number of Schubert cells $S_{\bm}$ of real 
dimension $q$.  Thus, each $\overline{S_{\bm}}$ defines a $\Z$-homology 
class of dimension $\dim_{\R} S_{\bm}$.  Together they form a basis for 
$H_{q}(SU_m; \Z)$.  Also, $\psi_{\bm}(\tilde{S}_{\bm}) = \overline{S_{\bm}}$ 
and $\tilde{S}_{\bm}$ has a top homology class in $H_{q}(\tilde{S}_{\bm}; 
\Z)$ for $q = \dim_{\R}(\tilde{S}_{\bm})$, which we can view as a 
fundamental class for $\tilde{S}_{\bm}$ for Borel-Moore homology.  We have 
a similar dimension count in cohomology, so that the duals of the classes 
$\overline{S_{\bm}}$ via the Kronecker pairing give a $\Z$-basis for 
cohomology.  \par 
Then, as $F_m^{c} = SU_m$ and the inclusion $\iti_m : F_m^{c} 
\hookrightarrow F_m$  is a homotopy equivalence, we obtain the following
\begin{Thm}
\label{Thm6.1}
The homology $H_*(F_m; \Z)$ has for a free $\Z$-basis the fundamental 
classes of the Schubert cycles, given as images $\iti_{m\, *}\circ 
\psi_{\bm\, *}([\tilde{S}_{\bm}]) = \psi_{\bm\, *}(\tilde{S}_{\bm}) = 
\overline{S_{\bm}}$ as we vary over the Schubert decomposition of 
$SU_m$.  The Kronecker duals of these classes give the $\Z$-basis for the 
cohomology 
$$  H^*(SU_m; \Z) \,\,  \simeq \,\, \gL^*\Z \langle e_3, e_5, \dots , 
e_{2m-1} \rangle \, . $$ 
Moreover, the Kronecker duals of the {\em simple Schubert classes} 
$S_{(m_1)}$ are homogeneous generators of the exterior algebra 
cohomology. 
\end{Thm}
\begin{proof}
The preceding discussion establishes all of the theorem except for the last 
statement about the generators of the cohomology algebra.  We prove this 
by induction on $m$.  It is trivially true for $m = 1, 2$.  Suppose it is true for 
$m < n$ and let $i_{n-1} : SU_{n-1} \hookrightarrow SU_n$ denote the 
natural inclusion $A \mapsto \bigl( \begin{smallmatrix} A & 0 \\ 0 & 1 
\end{smallmatrix}\bigr)$.  The Schubert decomposition preserves the 
inclusion so that any $S_{\bm}$ for $\bm = (m_1, m_2, \cdots , m_k)$ with  
$m_k < n$ is contained in the image of $i_{n-1}$ and so is also a Schubert 
cell for $SU_{n-1}$; while if $m_k = n$, then $S_{\bm}$ is in the complement 
of the image of 
$SU_{n-1}$.  Thus, if the result is true for $SU_{n-1}$, the Kronecker 
duals to the simple $S_{(m_1)}$ with $m_1 < n$ restrict via $i_{n-1}^*$ to 
the Kronecker duals of the  $S_{(m_1)}$ with $m_1 < n$ viewed as Schubert 
cells of $SU_{n-1}$.  Thus, they map to the generators of the exterior 
algebra $\gL^*\Z <e_3, e_5, \cdots e_{2n-3}>$.  Also, the Kronecker dual 
to any $S_{\bm}$ with $m_k = n$ is zero on any Schubert cell of $SU_{n-1}$ 
so by a counting argument the kernel of $i_{n-1}^*$, which is the ideal 
generated by $e_{2n-1}$, is spanned by the Kronecker duals of the Schubert 
cells with $m_k = n$.  \par 
Now there is a unique Schubert class of this type of degree $2n-1$, and 
hence its Kronecker dual is the added generator which together with the 
others for $S_{(m_1)}$ with $m_1 < n$ generate $H^*(SU_n ; \Z)$.
\end{proof}
\par
There is also the question of identifying the Kronecker dual of the Schubert 
cycle $[\overline{S}_{\bm}]$ for $\bm = (m_1, m_2, \cdots , m_k )$, which 
we denote by $e_{\bm}$.  We claim it is given up to sign by the cohomology 
class $e_{2m_1-1}\cdot e_{2m_2 - 1}\cdots e_{2m_k-1}$ (where the 
products denote cup-products).  
We show this using the product structure of the group $SU_m$ to give a 
product representation for the closures of Schubert cells together with the 
Hopf algebra structure of $H^*(SU_m)$.  \par 
We let $\overline{S_{\bm}}\cdot \overline{S_{\bm^{\prime}}}$ denote the 
group product in $SU_m$ of the closures of Schubert cells 
$\overline{S_{\bm}}$ and $\overline{S_{\bm^{\prime}}}$.  We also use the 
simpler notation $S_{m_1}$ to denote the Schubert cell $S_{\bm}$ when 
$\bm = (m_1)$.  In particular, we emphasize that 
$$ S_{m_1} \,\, = \,\,  \{ A_{(-\theta, e_1)}\cdot A_{(\theta, x_1)} : 
\theta \in (0, 2\pi), x_1 \in_{\min} \C^{m_1}\}\, . $$ 
\par
First, as result of Lemma \ref{Lem3.1} we obtain the following version of a 
Lemma due to J.H.C. Whitehead (see e.g. \cite[Lemma 4.2]{KM} or 
\cite[Lemma 2.2]{Mi}).  \par
\begin{Lemma}
\label{Lem6.2}
For Schubert cells in $\cC_m$ for $SU_m$, 
\begin{itemize}
\item[1)] If $1 < m_1 < m_2\leq m$ then
$$  \overline{S_{m_2}}\cdot \overline{S_{m_1}} \,\, = \,\, 
\overline{S_{m_1}}\cdot \overline{S_{m_2}} \,\, = \,\,  \overline{S_{(m_1, 
m_2)}} \, . $$
\item[2)] If $1 < m^{\prime} \leq m$, then 
$$  \overline{S_{m^{\prime}}}\cdot \overline{S_{m^{\prime}}} \,\, 
\subseteq \,\,  \overline{S_{(m^{\prime}-1, m^{\prime})}} \, . $$
\end{itemize}
\end{Lemma}
We note that this differs slightly from the above referred to Lemmas as 
each element in $S_{m_1}$ is a product of two pseudo-rotations, one of 
which is $A_{(-\theta, e_1)}$.  However, by the Lemma, this 
pseudo-rotation can also be interchanged with other $A_{(\theta, x_j)}$, and 
combined via multiplication with other $A_{(-\theta^{\prime}, e_1)}$.  
We also note in the Lemma that 
$\dim_{\R} S_{(m^{\prime}-1, m^{\prime})} \leq 2\cdot \dim_{\R} 
S_{m^{\prime}} - 2$.  \par
We can inductively repeat this to obtain 
\begin{Lemma}
\label{Lem6.3}
For Schubert cells $S_{m_j}$ in $\cC_ m$ (for $SU_m$):
\begin{itemize}
\item[1)] If $\bm = (m_1, m_2, \dots , m_r)$ then
$$  \overline{S_{\bm}} \,\, = \,\, \overline{S_{m_1}}\cdot 
\overline{S_{m_2}} \cdots \overline{S_{m_r}} \, . $$
\item[2)] If $\bm = (m_1, m_2, \dots , m_r)$ and $\bm^{\prime} = 
(m_1^{\prime}, m_2^{\prime}, \dots , m_{r^{\prime}}^{\prime})$ with 
$\{m_1, m_2, \dots , m_r\} \cap \{m_1^{\prime}, m_2^{\prime}, \dots , 
m_{r^{\prime}}^{\prime}\} = \emptyset$ then 
$$  \overline{S_{\bm}}\cdot \overline{S_{\bm^{\prime}}} \,\, = \,\,  
\overline{S_{\bm^{\prime\prime}}} \, . $$
where $\bm^{\prime\prime}$ is the union of $\bm$ and $\bm^{\prime}$ in 
increasing order.
\item[3)] If $\bm = (m_1, m_2, \dots , m_r)$ and $\bm^{\prime} = 
(m_1^{\prime}, m_2^{\prime}, \dots , m_{r^{\prime}}^{\prime})$ with 
$\{m_1, m_2, \dots , m_r\} \cap \{m_1^{\prime}, m_2^{\prime}, \dots , 
m_{r^{\prime}}^{\prime}\} \neq \emptyset$ then 
$$  \overline{S_{\bm}}\cdot \overline{S_{\bm^{\prime}}} \,\, \subset \,\,  
\cC_m^{(q)} \, . $$
where $q \leq \dim_{\R} S_{\bm} + \dim_{\R} S_{\bm^{\prime}} - 2$.
\end{itemize}
\end{Lemma}
\begin{proof}
For 1) we consider a product in $S_{m_1}\cdot S_{m_2} \cdots S_{m_r}$ 
which has the form 
\begin{equation}
\label{Eqn6.4} 
B \, \, = \,\, (A_{(-\theta_1, e_1)}\cdot A_{(\theta_1, x_1)}) \cdot (A_{(-
\theta_2, e_1)}\cdot A_{(\theta_2, x_2)}\cdots (A_{(-\theta_r, 
e_1)}\cdot A_{(\theta_r, x_r)})
\end{equation}   
where each $x_j \in_{\min} \C^{m_j}$.  
Then, we may repeatedly apply the Whitehead lemma to move each 
$A_{(-\theta_j, e_1)}$ to the left and obtain a factorization in the form
\begin{equation}
\label{Eqn6.5} 
B \, \, = \,\, A_{(-\tilde{\theta}, e_1)}\cdot A_{(\theta_1, x_1^{\prime})} 
\cdot A_{(\theta_2, x_2^{\prime}}) \cdots A_{(\theta_r, x_r^{\prime})})
\end{equation}
 where $\tilde{\theta} = \sum_{j = 1}^{r} \theta_j$ and each $x_j^{\prime} 
\in_{\min} \C^{m_j}$.  Hence, $B \in S_{\bm}$.  Conversely we can reverse 
the process beginning with $B$ in \eqref{Eqn6.5} and obtain a factorization 
as in \eqref{Eqn6.4}.  This gives the equality for the Schubert cells.  Since 
the closures are compact, we obtain the equality of 1) by taking closures of 
the Schubert cells.  \par
Given 1) we may write 
\begin{equation}
\label{Eqn6.6}
S_{\bm}\cdot S_{\bm^{\prime}}\, \, = \,\,  (S_{m_1}\cdot S_{m_2} \cdots 
S_{m_r}) \cdot (S_{m_1^{\prime}}\cdot S_{m_2^{\prime}} \cdots 
S_{m_{r^{\prime}}^{\prime}})
\end{equation}
If $\{m_1, m_2, \dots , m_r\} \cap \{m_1^{\prime}, m_2^{\prime}, \dots , 
m_{r^{\prime}}^{\prime}\} = \emptyset$, then we can repeatedly apply a) 
of the Whitehead Lemma to move an element of $S_{m_j^{\prime}}$ across 
an element of $S_{m_i}$ when $m_i > m_j^{\prime}$ while preserving the 
order of the $m_i$\rq s and $m_j^{\prime}$\rq s.  We arrive at an ordered 
factorization with increasing order $\bm^{\prime\prime}$, which is the union 
of $\bm$ and $\bm^{\prime}$ in increasing order.  Taking closures of the 
Schubert cells then gives 2).  \par
Finally, for 3), we may begin with \eqref{Eqn6.6}.  There are smallest 
$m_{\ell} = m^{\prime}_{k}$.  Then, if $m^{\prime}_j  < m^{\prime}_{k}$ then it differs from all $m_{i}$.  Hence, we can first move the elements in 
$S_{m_j^{\prime}}$ across all of those in $S_{m_{i}}$ as in the previous case by 2) of Lemma \ref{Lem6.3}.   Next, we can move elements in $S_{m_{k^{\prime}}}$ across those in $S_{m_j}$ as long as $m_j > m_{\ell}$.  Then, we arrive at a factorization where we have successive terms in $S_{m_{\ell}}$ and $S_{m_{k^{\prime}}}$ with $m_{\ell} = m^{\prime}_{k}$.  Then, we may 
apply b) of the Whitehead lemma (or 2) of Lemma \ref{Lem6.2}) and obtain a new 
pair in $S_{\tilde{m}}$ and $S_{m_{\ell}}$ with $\tilde{m} \leq m_{\ell}-1$.  
This has the effect of reducing the sum of the Schubert symbol values in the 
product by at least $1$.  Also, further application of the Whitehead Lemma will 
not increase the sum.  Hence, by further application of the Whitehead Lemma 
we obtain a product in the union of Schubert cells of dimension 
$q \leq \dim_{\R} S_{\bm} + \dim_{\R} S_{\bm^{\prime}} - 2$.
  Thus, it lies in the $q$-skeleton of $\cC_m$.  This gives 3) when 
we take closures.
\end{proof}
\par
Now we will use the Hopf structure of $H^*(SU_n)$ to relate the 
fundamental classes from the Schubert decomposition with the cohomology 
classes via the Kronecker pairing.  Let $\mu : SU_n \times SU_n \to SU_n$ 
denote the multiplication map.  Then, we can use Lemma \ref{Lem6.3} to 
determine the effect of $\mu_*$ for homology using the complex $C_k(\{ 
X^{(k)}\})$ and then the coproduct map $\mu^*$ for the Hopf algebra.  We 
obtain as a corollary of Lemma \ref{Lem6.3}.
\begin{Corollary}
\label{Cor6.7}
We let $s_{\bm}$ denote the homology class obtained from $\psi_{\bm\, 
*}([\tilde{S}_{\bm}])$ with restriction to positive orientation for $E_{\bm}$.  
For $\bm = (m_1, m_2, \dots , m_r)$ and $\bm^{\prime} = (m_1^{\prime}, 
m_2^{\prime}, \dots , m_{r^{\prime}}^{\prime})$ we let $\itm = \{ m_1, 
m_2, \dots , m_r\} \cap \{ m_1^{\prime}, m_2^{\prime}, \dots , 
m_{r^{\prime}}^{\prime}\}$ and let $\bm^{\prime\prime} = 
(m_1^{\prime\prime}, m_2^{\prime\prime}, \dots , 
m_{r^{\prime\prime}}^{\prime\prime})$ denote the union of the elements of 
$\bm$ and $\bm^{\prime}$ written in increasing order.  Then,
\begin{equation}
\label{Eqn6.8}
\mu_*(s_{\bm} \otimes s_{\bm^{\prime}})\,  = \, \begin{cases}
\gevar_{\bm, \bm^{\prime}}\cdot s_{\bm^{\prime\prime}}& \text{ if } \itm 
= \emptyset, \\
0& \text{ if } \itm \neq \emptyset\, .
\end{cases}
\end{equation}
where $\gevar_{\bm, \bm^{\prime}}$ is the sign of the permutation which 
moves $(\bm, \bm^{\prime})$ to increasing order.
\end{Corollary}
The reason for the factor $\gevar_{\bm, \bm^{\prime}}$ is that each 
interchange of two factors $S_{(m_1)}$ and $S_{(m_2)}$ will change the 
orientation by a factor $(-1)^{(2m_1 -1)(2m_2 - 1)} = -1$. \par 
From the corollary we obtain a formula for the coproduct $\mu^*$ in terms 
of the (Kronecker) dual basis $\{e_{\bm}\}$ in cohomology to Schubert basis 
for homology $\{s_{\bm}\}$.
\begin{equation}
\label{Eqn6.9} 
\mu^*(e_{\bm})\,\, = \,\, \sum (-1)^{\deg(e_{\bm^{\prime}}) 
\deg(e_{\bm^{\prime\prime}})} \gevar_{\bm^{\prime}, 
\bm^{\prime\prime}}\cdot e_{\bm^{\prime}} \otimes 
e_{\bm^{\prime\prime}} \, ,
\end{equation}
where the sum is over all disjoint $\bm^{\prime}$ and $\bm^{\prime\prime}$ 
whose union in increasing order gives $\bm$ (and the terms $(-
1)^{\deg(e_{\bm^{\prime}}) \deg(e_{\bm^{\prime\prime}})}$ arise from 
the property $(\varphi \otimes \psi)(\gs \otimes \nu) = (-1)^{\deg(\varphi) \deg(\psi)} 
\varphi(\gs) \psi(\nu)$).  Since $S_{\bm}$ is a product of odd 
dimensional cells, $\deg(e_{\bm^{\prime}}) (= \dim_{\R} S_{\bm}) \equiv 
\ell(\bm)\, \mod 2$ and the sign in \eqref{Eqn6.9} equals $(-
1)^{\ell(\bm^{\prime}) \ell(\bm^{\prime\prime})}$.  Also, note the sum 
includes the empty symbol which denotes the Schubert cell consisting of just 
$I_n$.  In the case of the simple Schubert symbol $(m_1)$ we obtain
$$ \mu^*(e_{(m_1)})\,\, = \,\, e_{(m_1)} \otimes 1 \, + \, 1 \otimes 
e_{(m_1)} \, .$$
Hence, all of the $e_{(m_1)}$ are independent primitive classes.  
Then there is the following relation between the generators of $H^*(SU_n)$ 
and the Schubert classes.
\begin{Thm}
\label{Thm6.9}
$H^*(SU_n)$ is a free exterior algebra with generators $e_{(m)}$ of degrees 
$2m-1$, for $m = 2, \dots, n$.  Moreover the Kronecker dual to $s_{\bm}$ for 
$\bm = (m_1, m_2, \dots , m_r)$ is $e_{\bm} = (-
1)^{\gb(\bm)} e_{(m_1)} e_{(m_2)} \dots e_{(m_r)}$.  
where $\gb(\bm) = \binom{\ell(\bm)}{2}$ (where we denote $\binom{1}{2} = 
0$).  
\end{Thm}
\begin{proof}
We already have established the first statement about the algebra 
generators in Theorem \ref{Thm6.1}.  We note that it also follows from the 
Hopf algebra structure.  Since the $e_{(m)}$ , for $m = 2, \dots, n$ are 
primitive generators of degree $2m-1$, and $H^*(SU_n)$ is a Hopf algebra 
which is a free exterior algebra on generators of degrees $2m-1$ for $m = 
2, \dots, n$, it follows by a theorem of Hopf-Samuelson that $H^*(SU_n)$ is 
the free exterior algebra generated by the primitive elements $e_{(m)}$ , 
for $m = 2, \dots, n$.  \par
We furthermore claim that the Kronecker dual to the Schubert class 
$s_{\bm}$ for $\bm = (m_1, m_2, \dots , m_r)$ is given by 
$(-1)^{\gb(\bm)} e_{(m_1)} e_{(m_2)} \dots e_{(m_r)}$, which will follow 
from $e_{\bm} = (-1)^{\ell(\bm^{\prime})} e_{(m_1)} e_{\bm^{\prime}}$  
for $\bm^{\prime} = (m_2, m_3, \dots m_r)$.  We prove this by induction 
on $r$.  It is already true for $r = 1$.  Next, consider the case of $\bm = 
(m_1, m_2)$; then $\gevar_{(m_1), (m_2)} = 1$, $\gevar_{(m_2), (m_1)} = 
-1$ and $(-1)^{\ell(m_1) \ell(m_2)} = -1$.  Then, from 
\eqref{Eqn6.9} 
\begin{equation}
\label{Eqn6.10} 
\mu^*(e_{(m_1, m_2)})\,\, = \,\, e_{(m_1, m_2)} \otimes 1 \,   -  
e_{(m_1)} \otimes e_{(m_2)}  \, + \, e_{(m_2)} \otimes e_{(m_1)} \, + \, 1 
\otimes e_{(m_1, m_2)}\, .
\end{equation}
Also, as $\mu^*$ is an algebra homomorphism, 
\begin{align}
\label{Eqn6.11} 
\mu^*(e_{(m_1)}\cdot e_{(m_2)})\,\, &= \,\,  \mu^*(e_{(m_1}) \cdot 
\mu^*(e_{(m_2)})     \notag \\     
\,\, &= \,\, \left( e_{(m_1)} \otimes 1 \, + \, 1 \otimes e_{(m_1)}\right) 
\cdot \left( e_{(m_2)} \otimes 1 \, + \, 1 \otimes e_{(m_2)}\right)  \notag 
\\ 
\,\, &= \,\,  e_{(m_1)}\cdot e_{(m_2)} \otimes 1 \, + \, e_{(m_1)} \otimes 
e_{(m_2)} \, - \, e_{(m_2)} \otimes e_{(m_1)}    \notag \\
& \qquad \qquad \qquad  \qquad \qquad \qquad \qquad \, + \,  1 \otimes 
e_{(m_1)}\cdot e_{(m_2)}\, , 
\end{align}
where the signs on the RHS result from both $ e_{(m_1)}$ and $e_{(m_2)}$ 
having odd degree.  Adding \eqref{Eqn6.11} and \eqref{Eqn6.10} we 
obtain
\begin{align}
\label{Eqn6.11b} 
\mu^*(e_{(m_1, m_2)} + e_{(m_1)}\cdot e_{(m_2)})\, \, &= \, \, \left( 
e_{(m_1, m_2)} + e_{(m_1)}\cdot e_{(m_2)}\right) \otimes 1 \, \, \, + \,  
\notag  \\
& \qquad \qquad \qquad  \qquad \qquad 1 \otimes \left( e_{(m_1, m_2)} + 
e_{(m_1)}\cdot e_{(m_2)}\right)\, . 
\end{align}
This implies that if $e_{(m_1, m_2)} + e_{(m_1)}\cdot e_{(m_2)} \neq 0$, 
then it is a primitive element independent from the other primitive elements 
$e_{(m)}$.  This contradicts the Hopf-Samuelson theorem.  Thus, 
 $e_{(m_1, m_2)} = - e_{(m_1)}\cdot e_{(m_2)}$.  \par
Suppose by induction the result holds for $k < r$.  
Then, for
$\bm = (m_1, \dots , m_r)$, let $\bm^{\prime} = (m_2, \dots , m_r)$.  
First, by \eqref{Eqn6.9} we have  
\begin{equation}
\label{Eqn6.12a} 
\mu^*(e_{\bm})\,\, = \,\, e_{\bm} \otimes 1 \, + \, 1 \otimes e_{\bm}  \, 
+ \,     \sum (-1)^{\ell(\bm^{\prime}) \ell(\bm^{\prime\prime})} 
\gevar_{\bm^{\prime}, \bm^{\prime\prime}}\cdot e_{\bm^{\prime}} 
\otimes e_{\bm^{\prime\prime}} \, ,
\end{equation}
where the sum is over all $\bm^{\prime} = (m_1^{\prime}, m_2^{\prime}, 
\dots , m_k^{\prime})$ and $\bm^{\prime\prime} = (m_1^{\prime\prime}, 
m_2^{\prime\prime}, \dots , m_{k^{\prime}}^{\prime\prime})$ which are 
both nonempty, disjoint, and whose union in increasing order is $\bm$.  
Then, by induction we obtain
\begin{align}
\label{Eqn6.12}
 \mu^*(e_{(m_1)}\cdot e_{\bm^{\prime}}) \,\, &= \,\, 
\mu^*(e_{(m_1})\cdot \mu^*(e_{\bm^{\prime}})  \notag  \\
\,\, &= \,\, \left( e_{(m_1)} \otimes 1 \, + \, 1 \otimes e_{(m_1)}\right) 
\cdot \left( e_{\bm^{\prime}} \otimes 1 \, + \, 1 \otimes 
e_{\bm^{\prime}} + \right.  \notag  \\ 
&  \qquad  \left. \sum (-1)^{\ell(\bm^{\prime\prime}) 
\ell(\bm^{\prime\prime\prime})} \gevar_{\bm^{\prime\prime}, 
\bm^{\prime\prime\prime}}\cdot e_{\bm^{\prime\prime}} \otimes 
e_{\bm^{\prime\prime\prime}}\right)  
\end{align}
where the sum is over $\bm^{\prime\prime}$ and 
$\bm^{\prime\prime\prime}$ which are nonempty, disjoint and whose union 
in increasing order is $\bm^{\prime}$.  
In the sum on the RHS of \eqref{Eqn6.12a}, we have in addition to the terms 
$e_{\bm} \otimes 1$ and  $1 \otimes  e_{\bm}$ the four following types of terms : \flushpar
{\it Four Types of Terms in \eqref{Eqn6.12a}}: \hfill
\par
\begin{itemize}
\item[i)]
$ (-1)^{\ell(\bm^{\prime})} \gevar_{(m_1), \bm^{\prime}}\cdot e_{(m_1)} 
\otimes e_{\bm^{\prime}}\,  = \, (-1)^{\ell(\bm^{\prime})} e_{(m_1)} 
\otimes e_{\bm^{\prime}}$
\item[ii)]
$(-1)^{\ell(\bm^{\prime})} \gevar_{\bm^{\prime}, (m_1)}\cdot 
e_{\bm^{\prime}} \otimes e_{(m_1)}\, = \, e_{\bm^{\prime}} \otimes 
e_{(m_1)}$
\item[iii)]
$(-1)^{\ell(\bm^{\prime\prime})\ell(\bm^{\prime\prime\prime})} 
\gevar_{\bm^{\prime\prime}, \bm^{\prime\prime\prime}}\cdot 
e_{\bm^{\prime\prime}} \otimes e_{\bm^{\prime\prime\prime}}$  \qquad 
\text{ with $m_1$ in } $\bm^{\prime\prime}$
\item[iv)] 
$ (-1)^{\ell(\bm^{\prime\prime})\ell(\bm^{\prime\prime\prime})}
 \gevar_{\bm^{\prime\prime}, \bm^{\prime\prime\prime}}\cdot  
e_{\bm^{\prime\prime}} \otimes e_{\bm^{\prime\prime\prime}}$ 
 \qquad \text{ with $m_1$ in } 
$\bm^{\prime\prime\prime}$
\end{itemize}
 For comparison, we have in addition to the terms 
$(e_{(m_1)} e_{\bm^{\prime}}) \otimes 1$ and  $1 \otimes (e_{(m_1)} 
e_{\bm^{\prime}})$ the corresponding terms from \eqref{Eqn6.12} which have the 
following types:
\flushpar
{\it Corresponding Four Types of Terms in \eqref{Eqn6.12}}: \hfill
\begin{itemize}
\item[i)]
$e_{(m_1)} \otimes e_{\bm^{\prime}}$
\item[ii)]
$(-1)^{\ell(\bm^{\prime})} e_{\bm^{\prime}} \otimes e_{(m_1)} $
\item[iii)]
$(-1)^{\ell(\bm^{\prime\prime})\ell(\bm^{\prime\prime\prime})} 
\gevar_{\bm^{\prime\prime}, \bm^{\prime\prime\prime}}\cdot (e_{(m_1)} 
e_{\bm^{\prime\prime}}) \otimes e_{\bm^{\prime\prime\prime}}$
\item[iv)] 
$(-1)^{\ell(\bm^{\prime\prime})} 
(-1)^{\ell(\bm^{\prime\prime})\ell(\bm^{\prime\prime\prime})}
 \gevar_{\bm^{\prime\prime}, \bm^{\prime\prime\prime}}\cdot  
e_{\bm^{\prime\prime}} \otimes (e_{(m_1)} 
e_{\bm^{\prime\prime\prime}})$
\end{itemize}
\par 
In the first two cases for \eqref{Eqn6.12}, we can view them as a 
decomposition of $\bm$ either as $(\{m_1\}, \bm^{\prime})$ or 
$(\bm^{\prime}, \{m_1\})$.  We see that the corresponding coefficients for 
i) and ii) for \eqref{Eqn6.12} and \eqref{Eqn6.12a} differ by a factor 
$(-1)^{\ell(\bm^{\prime})}$.  The corresponding terms in iii) and iv) for 
\eqref{Eqn6.12} can be viewed as a decomposition  either as $(\{m_1\} \cup 
\bm^{\prime\prime}, \bm^{\prime\prime\prime})$ or 
$(\bm^{\prime\prime}, \{m_1\} \cup \bm^{\prime\prime\prime})$.  The 
corresponding coefficients will also be shown to differ by the same factor 
$(-1)^{\ell(\bm^{\prime})}$.  \par
For example, for iv) let $\tilde{\bm}^{\prime\prime\prime} = \{m_1\} \cup 
\bm^{\prime\prime\prime}$.  Then, $\gevar_{\bm^{\prime\prime}, 
\tilde{\bm}^{\prime\prime\prime}} = (-1)^{\ell(\bm^{\prime\prime})} 
\gevar_{\bm^{\prime\prime}, \bm^{\prime\prime\prime}}$; 
$\ell(\tilde{\bm}^{\prime\prime\prime}) = \ell(\bm^{\prime\prime\prime}) 
+ 1$; and  by the induction hypothesis 
$e_{\tilde{\bm}^{\prime\prime\prime}} = (-
1)^{\ell(\bm^{\prime\prime\prime})} e_{(m_1)}\cdot 
e_{\bm^{\prime\prime\prime}}$.  Then, substituting these values in iv) for 
\eqref{Eqn6.12} yields 
\begin{align}
\label{6.13a}
 &(-1)^{\ell(\bm^{\prime\prime})} 
(-1)^{\ell(\bm^{\prime\prime})\ell(\bm^{\prime\prime\prime})}
 \gevar_{\bm^{\prime\prime}, \bm^{\prime\prime\prime}}\cdot  
e_{\bm^{\prime\prime}} \otimes (e_{(m_1)} 
e_{\bm^{\prime\prime\prime}}) \,\, = \,\,  \notag  \\ 
&(-1)^{\ell(\bm^{\prime\prime})} 
(-1)^{\ell(\bm^{\prime\prime})\ell(\tilde{\bm}^{\prime\prime\prime})} 
(-1)^{\ell(\bm^{\prime\prime})} (-1)^{\ell(\bm^{\prime\prime})} 
(-1)^{\ell(\bm^{\prime\prime\prime})} \gevar_{\bm^{\prime\prime}, 
\tilde{\bm}^{\prime\prime\prime}}\cdot e_{\bm^{\prime\prime}} \otimes 
e_{\tilde{\bm}^{\prime\prime\prime}}  \notag  \\ 
&= \,  
(-1)^{\ell(\bm^{\prime\prime})\ell(\tilde{\bm}^{\prime\prime\prime})} 
 (-1)^{\ell(\bm^{\prime\prime})} (-1)^{\ell(\bm^{\prime\prime\prime})} 
\gevar_{\bm^{\prime\prime}, \tilde{\bm}^{\prime\prime\prime}} \cdot 
e_{\bm^{\prime\prime}} \otimes e_{\tilde{\bm}^{\prime\prime\prime}}  
\notag  \\
&= \,   (-1)^{\ell(\bm^{\prime})} 
\left((-1)^{\ell(\bm^{\prime\prime}) 
\ell(\tilde{\bm}^{\prime\prime\prime})} 
\gevar_{\bm^{\prime\prime}, \tilde{\bm}^{\prime\prime\prime}} \cdot 
e_{\bm^{\prime\prime}} \otimes 
e_{\tilde{\bm}^{\prime\prime\prime}}\right)
\end{align}
A similar, but somewhat simpler, argument works for the terms iii).  
\par
Then, we proceed as in the previous case.  We compute $\mu^*(e_{\bm} - (-
1)^{\ell(\bm^{\prime})} e_{(m_1)} e_{\bm^{\prime}})$ from 
\eqref{Eqn6.12} and \eqref{Eqn6.12a} and by the above all terms of types i) 
- iv) cancel so we obtain
\begin{align}
\label{Eqn6.13b}
\mu^*(e_{\bm} - (-1)^{\ell(\bm^{\prime})} e_{(m_1)} e_{\bm^{\prime}}) 
\,\, &= \,\, (e_{\bm} - (-1)^{\ell(\bm^{\prime})} e_{(m_1)} 
e_{\bm^{\prime}}) \otimes 1 \,\, + \,\,  \notag   \\  
&\quad\quad\quad\quad\quad\quad 1 \otimes (e_{\bm} - (-
1)^{\ell(\bm^{\prime})} e_{(m_1)} e_{\bm^{\prime}})
\end{align}
This again implies that $e_{\bm} - (-1)^{\ell(\bm^{\prime})} e_{(m_1)} 
e_{\bm^{\prime}}$ is a primitive element if it is nonzero.  Hence, it is zero 
and so $e_{\bm} = (-1)^{\ell(\bm^{\prime})} e_{(m_1)} e_{\bm^{\prime}}$.  
Repeated inductive application of this implies that for $\bm = (m_1, m_2, 
\dots , m_r)$
$$e_{\bm} \,\, = \,\,  (-1)^{\gb(\bm)} e_{(m_1)}\cdot  e_{(m_2)} \cdots 
e_{(m_r)}\,  .$$
with $\gb(\bm) = 1 + 2 + \cdots + (r-1) = \binom{\ell(\bm)}{2}$.  
\end{proof}
\par
As a consequence we have determined the Poincar\'{e} duals to the Schubert 
classes.
\begin{Corollary}
\label{Cor6.13}
For each Schubert symbol $\bm = (m_1, m_2, \dots , m_r)$ let the ordered 
complement in $\{2, 3, \dots , n\}$ be denoted by $\bm^{\prime} = 
(m_1^{\prime}, m_2^{\prime}, \dots, m_{n-1-r}^{\prime})$.  
\begin{itemize}
\item[i)] The Poincar\'{e} dual to the Schubert class $\left[ 
\overline{S_{\bm}}\right]$ in $F_n^{c}$ and to the Schubert class $\left[ 
\overline{S_{\bm}}\cdot Sol_n \right]$ in $F_n$ is given by $$(-
1)^{(\gb(\bn)+\gb(\bm))}\gevar_{\bm, \bm^{\prime}}\, e_{(m_1^{\prime})}\cdot  e_{(m_2^{\prime})} \cdots e_{(m_{n-1- r}^{\prime})}$$ 
for $\bn = (2, 3, \dots , n)$.  
\item[ii)] For Schubert symbols $\bm$ and $\bm^{\prime}$ such that 
$\ell(\bm) + \ell(\bm^{\prime}) = n-1$, the intersection pairing satisfies
\begin{equation}
\label{Eqn6.13c}
\langle [\overline{S_{\bm}}], [\overline{S_{\bm^{\prime}}}]\rangle \,  = \, 
\begin{cases}
(-1)^{(\gb(\bn)+\gb(\bm)+\gb(\bm^{\prime}))}\gevar_{\bm, \bm^{\prime}} & \text{ if $\bm^{\prime}$ 
is the ordered }  \\
 & \text{ complement to $\bm$} , \\
0& \text{ otherwise } \, .
\end{cases}
\end{equation}
\end{itemize}

\end{Corollary}
\begin{proof}
By Theorem \ref{Thm6.9}, the Kronecker dual to $\left[ \overline{S_{\bm}} 
\right]$ is given by $e_{\bm} = (-1)^{\gb(\bm)} e_{(m_1)}\cdot  e_{(m_2)} 
\cdots e_{(m_r)}$.  Also, the fundamental class for $\left[ SU_n \right]$ 
with orientation given by $\left[ \overline{S_{\bn}} \right]$ has Kronecker 
dual $(-1)^{\gb(\bn)} e_{(2)}\cdot  e_{(3)} \cdots e_{(n)}$.  Then, the 
Poincar\'{e} dual to $\left[ \overline{S_{\bm}} \right]$ is given by a 
cohomology class $\nu$ such that 
$e_{\bm} \cup \nu = (-1)^{\gb(\bn)} e_{(2)}\cdot  e_{(3)} \cdots e_{(n)}$.  
This is satisfied by 
$$\nu = (-1)^{(\gb(\bn)+\gb(\bm))}\gevar_{\bm, \bm^{\prime}}\, e_{(m_1^{\prime})}\cdot  
e_{(m_2^{\prime})} \cdots e_{(m_{n-1- r}^{\prime})}\, .$$
  \par
In the case of the Schubert class $\left[ \overline{S_{\bm}}\cdot Sol_n 
\right]$ in $F_n$, we note that $\overline{S_{\bm}}$ is the transverse 
intersection of $F_n^{c} = SU_n$ with $\overline{S_{\bm}}\cdot Sol_n$ in 
$F_n$ and that the inclusion $\iti_n : F_n^{c} \hookrightarrow F_n$ is a homotopy 
equivalence.  Hence, by a fiber square argument, the Poincar\'{e} dual in 
$H^*(F_n; \Z)$ to the fundamental class of $\overline{S_{\bm}}\cdot Sol_n$ for 
Borel-Moore homology, agrees via $\iti_n^*$ with that for the fundamental class of 
$\overline{S_{\bm}}$ in 
$H^*(F_n^{c}; \Z)$.  \par
The consequence for the intersection pairing follows from the above and
\begin{equation}
\label{Eqn6.13d}
\langle [\overline{S_{\bm}}], [\overline{S_{\bm^{\prime}}}] \rangle \, \, = \, 
\, \langle e_{\bm} \cup e_{\bm^{\prime}}, \left[ \overline{S_{\bn}} \right] 
\rangle 
\end{equation}
\end{proof}
\par
\subsection*{Milnor Fiber for the Variety of Singular 
$m \times m$-Skew-Symmetric Matrices} \hfill
\par
We second consider the case of the global Milnor fiber $F_m^{(sk)}$ for 
skew-symmetric matrices with $m = 2n$.  Then, the homology of $SU_{2n}/Sp_n$ 
can be computed from the algebraic complex with basis formed from the 
Schubert cells $S_{\bm}^{(sk)}$.  By a result of Cartan (see e.g. Mimura-Toda 
\cite[Theorem 6.7]{MT}) the homology of $SU_{2n}/Sp_n$ (which is 
isomorphic as a graded $\Z$-module to its cohomology) is given as a graded 
$\Z$-module by
\begin{equation}
\label{Eqn6.14a}
  H_*(SU_{2n}/Sp_n; \Z) \,\,  \simeq \,\, \gL^*\Z \langle s_5, s_9, \dots , 
s_{4n-3} \rangle \, . 
\end{equation}
where $s_{4j-3}$ has degree $4j-3$.  By the universal coefficient theorem 
this holds as well as a vector space over a field $\bk$ of characteristic zero.  
\begin{Thm}
\label{Thm6.14}
The homology $H_*(F_m^{(sk)\, c}; \Z)$ for $m = 2n$ has for a free 
$\Z$-basis the fundamental classes of the skew-symmetric Schubert cycles,  
$\iti_{m\, *}\circ \psi_{\bm\, *}^{(sk)}([\tilde{S}_{\bm}^{(sk)}]) = 
\psi_{\bm\, *}^{(sk)}(\tilde{S}_{\bm}^{(sk)}) = \overline{S_{\bm}^{(sk)}}$ 
as we vary over the Schubert decomposition of $\cC^{(sk)}_m \simeq 
SU_{2n}/Sp_n$.  Moreover, the Kronecker duals of the simple 
skew-symmetric Schubert cycles $\overline{S_{(m_1)}^{(sk)}}$ give 
homogeneous exterior algebra generators for the cohomology. \par 
This likewise extends to $H_*(F_m^{(sk)}; \Z)$ ($m = 2n$) for Borel-Moore 
homology with basis given by the fundamental classes of the global 
skew-symmetric Schubert cycles 
$Sol_m^T\cdot (\overline{S^{(sk)}_{\bm}}\cdot J_n)$ for $F_m^{(sk)}$.  
The Poincar\'{e} duals of these classes form a $\Z$-basis for the cohomology
$$  H^*(F_m^{(sk)}; \Z) \,\,  \simeq \,\, \gL^*\Z \langle e_5, e_9, \dots , 
e_{4n-3} \rangle \, . $$
\end{Thm}
\begin{proof}
The proof follows the same lines as that of Theorem \ref{Thm6.1}. 
Then, a count from \eqref{Eqn6.14a} shows that $H_{q}(SU_{2n}/Sp_n; \Z)$ 
is spanned by $s_{4m_1-3}\cdot s_{4m_2 - 3}\cdots s_{4m_k-3}$ where $1 
< m_1 < m_2 < \cdots < m_k \leq n$ and  $q = \sum_{j = 1}^{k} (4m_j -3)$.  
By Theorem \ref{5.15} this equals the number of 
skew-symmetric Schubert cells $S_{\bm}^{(sk)}$ of real dimension $q$.  
Thus, each $\psi_{\bm}^{(sk)}(\tilde{S}_{\bm}^{(sk)}) = 
\overline{S_{\bm}^{(sk)}}$ defines a $\Z$-homology class of dimension 
$\dim_{\R} S_{\bm}^{(sk)}$.  Together they form a basis for 
$H_{q}(SU_{2n}/Sp_n; \Z)$.  That the Kronecker duals of the simple 
Schubert cycles $S_{(m_1)}^{(sk)}$ give algebra generators for the 
cohomology follows by the same argument used in Theorem \ref{Thm6.1}.
\par  
As $\tilde{S}_{\bm}^{(sk)}$ has a top homology class in 
$H_{q}(\tilde{S}_{\bm}^{(sk)}; \Z)$ for $q = \dim_{\R}(\tilde{S}_{\bm}^{(sk)})$, 
we can view it as a fundamental class for $\tilde{S}_{\bm}^{(sk)}$ for 
Borel-Moore homology.  As $F_m^{(sk)\,c} \simeq 
\cC_m^{(sk)} \simeq SU_{2n}/Sp_n$ by multiplication by $J_n$ and the 
inclusion $\iti_m : F_m^{(sk)\, c} \hookrightarrow F_m^{(sk)}$ is a homotopy 
equivalence, we conclude that these classes form a $\Z$-basis for the 
cohomology via $H^{*}(F_m^{(sk)}; \Z) \simeq H^{*}(F_m^{(sk)\, c}; \Z)$.  Their 
Poincar\'{e} duals then form a $\Z$-basis for the the Borel-Moore homology.
\end{proof}
\par
Again there is the question of explicitly identifying the Kronecker dual of the 
fundamental class $\psi_{\bm\, *}^{(sk)}([\tilde{S}_{\bm}^{(sk)}])$ with a 
cohomology class as a polynomial in the cohomology algebra generators 
$e_{4j-3}$,  $j = 2, \dots , n$, and as a consequence explicitly identifying the 
generators for the cohomology algebra.  We shall comment on this 
after next considering the symmetric case.  
\flushpar
\subsection*{Milnor Fiber for the Variety of Singular 
$m \times m$-Symmetric Matrices} \hfill
\par
We next consider the case of $F_m^{(sy)}$.  Again the line of reasoning will be 
similar to the two preceding cases with the crucial difference that the 
(co)homology has two different forms for coefficients $\Z/2\Z$ or a field of 
characteristic zero. 
There is the compact model $F_n^{(sy)\, c} \simeq \cC_n^{(sy)} \simeq 
SU_n/SO_n$ for $F_n^{(sy)}$.  
Then, the homology of $SU_n/SO_n$ can be computed from the algebraic 
complex with basis formed from the Schubert cells $S_{\bm}^{(sy)}$.  By a 
result of Borel and Hopf, see e.g. \cite{Bo} and see \cite{KM}, the homology 
of $SU_n/SO_n$ with $\Z/2\Z$-coefficients (which is isomorphic as a graded 
$\Z/2\Z$-vector space to its cohomology) is given as a graded vector space 
over the field $\Z/2\Z$  
$$  H_*(SU_n/SO_n; \Z/2\Z) \,\,  \simeq \,\, \gL^*\Z/2\Z \langle s_2, s_3, 
\dots , s_{n} \rangle \, . $$
where $s_{j}$ has degree $j$.  
A count shows that 
$$ \dim_{\Z/2\Z} H_*(SU_n/SO_n; \Z/2\Z) \, = \,  2^{n-1}\, .$$
 This is the same as the number of Schubert cells $S_{\bm}^{(sy)}$, for 
$1 < m_1 < \dots < m_k \leq n$ in the cell decomposition of $SU_n/SO_n$.  
Thus, the Schubert cycles $\overline{S_{\bm}^{(sy)}}$, which are 
$\mod 2$-homology cycles, give a $\Z/2\Z$-basis for the homology 
$H_*(SU_n/SO_n; \Z/2\Z)$.  
In particular the $\mod 2$-homology cycles $\overline{S_{\bm}^{(sy)}}$ for which $|\bm | = q$ give a $\Z/2\Z$-basis for $H_{q}(SU_n/SO_n; \Z/2\Z)$ for each $q \geq 0$.  \par  
Thus, we conclude by an analogous argument to that used in the preceding 
two cases
\begin{Thm}
\label{Thm6.15}
The homology $H_*(F_n^{(sy)\, c}; \Z/2\Z)$ has for a $\Z/2\Z$-basis the 
$\Z/2\Z$ fundamental classes of the symmetric Schubert cycles $\left[ 
\overline{S_{\bm}^{(sy)}} \right]$ as we vary over the Schubert 
decomposition of $\cC_n^{(sy)} \simeq SU_n/SO_n$ for all symmetric 
Schubert symbols $\bm^{(sy)} = (m_1, \dots , m_k)$ with $1 < m_1 < \dots 
< m_k \leq n$.  Moreover, the Kronecker duals of the simple symmetric 
Schubert cycles $\overline{S^{(sy)}_{(m_1)}}$ are algebra generators for the 
exterior cohomology algebra with $\Z/2\Z$-coefficients.  \par
This extends to $H_*(F_n^{(sy)}; \Z/2\Z)$ with $\Z/2\Z$-basis given by the 
Borel-Moore $\mod 2$-cycles given by the global symmetric Schubert cycles 
$\left[ Sol_m^T\cdot (S^{(sy)}_{\bm}) \right]$ for $S^{(sy)}_{\bm}$ over the 
symmetric Schubert symbols $\bm^{(sy)}$.  
The Poincar\'{e} duals of these classes form a $\Z/2\Z$-basis for the 
cohomology.  
$$  H^*(F_m^{(sy)}; \Z/2\Z) \,\,  \simeq \,\, \gL^*\Z/2\Z \langle e_2, e_3, 
\dots , e_{n} \rangle \, . $$
\end{Thm}
\par 
There are several points to be made regarding this result and that for 
skew-symmetric matrices.  \par
First, unlike the cases of $SU_n$ and $SU_{2n}/Sp_n$, the closure of the 
Schubert cells are not the images of Borel-Moore homology classes of 
singular manifolds.  As mentioned earlier, if we consider instead the quotient 
space $F_m^{(sy)\, c}/(F_m^{(sy)\, c})^{(q-1)}$, and $|\bm | = q$, then the 
composition of the map  
$$ \tilde{\psi}_{\bm}^{(sy)} : \prod_{i = 1}^{k} (C\RP^{m_i-1})\,\,  
\longrightarrow \,\, SU_n/SO_n \simeq F_m^{(sy)\, c} $$
with the quotient map 
$pr_{q} : F_m^{(sy)\, c} \to F_m^{(sy)\, c}/(F_m^{(sy)\,c})^{(q-1)}$ factors through 
to give a map
$$pr_{q}\circ \tilde{\psi}_{\bm}^{(sy)}: \prod_{i = 1}^{k} S\RP^{m_i-1} \,\,  
\longrightarrow \,\, F_m^{(sy)\, c}/(F_m^{(sy)\, c})^{(q-1)}\, . $$
As $pr_{q} : (F_m^{(sy)\, c}, (F_m^{(sy)\, c})^{(q-1)}) \to (F_m^{(sy)\, 
c}/(F_m^{(sy)\, c})^{(q-1)}, *)$, for $*$ the point representing $(F_m^{(sy)\, 
c})^{(q-1)}$ in the quotient, is a relative homeomorphism,
$$pr_{q\, *} : H_{q}(F_m^{(sy)\, c}, (F_m^{(sy)\, c})^{(q-1)}; \Z/2\Z) \,\, \simeq 
\,\,  H_{q}(F_m^{(sy)\, c}/(F_m^{(sy)\, c})^{(q-1)}, *; \Z/2\Z)\, . $$
Then, the closure $\overline{S_{\bm}^{(sy)}}$ corresponds via the 
isomorphism to the image of the fundamental class of $\prod_{i = 1}^{k} 
(S\RP^{m_i-1})$ under $pr_{q\, *}\circ \tilde{\psi}_{\bm\, *}^{(sy)}$. 
\par
Moreover, as noted earlier for the simple Schubert symbol $(m_1)$, there is 
a factored map 
$ \tilde{\psi}_{(m_1)}^{(sy)} : S\RP^{m_i-1} \to SU_n/SO_n \simeq 
F_m^{(sy)\, c} $ with image $\overline{S_{(m_1)}^{(sy)}}$, giving it a 
Borel-Moore fundamental homology class for $\Z/2\Z$-coefficients.  
\par
However, for cohomology with rational coefficients, see e.g. \cite[Chap. 3, 
Thm 6.7 (2)]{MT} or Table 1 in \cite{D3}, many of these Schubert cells do 
not contribute homology classes.  This is similar to the situation for oriented 
Grassmannians for $\Z/2\Z$ versus rational coefficients.  This relation 
extends further.  Over $SU_n/SO_n$ is a natural $n$-dimensional real 
oriented vector bundle  $E_n = (SU_n \times_{SO_n} \R^n)$ where $\R^n$ 
has the natural representation of $SO_n$.  This bundle can be viewed 
geometrically as the set of oriented real subspaces $V \subset \C^n$ with 
$\dim_{\R} V = n$ such that $\C\langle V\rangle \,= \C^n$.  Then, by e.g. 
\cite[Chap. 3, Thm 6.7 (3)]{MT} the cohomology of $SU_n/SO_n$, already 
quoted in Theorem \ref{Thm6.15} has $e_j = w_j(E_n)$, the $j$-th 
Stiefel-Whitney class.  This bundle pulls-back by the homotopy equivalence 
$SU_n/SO_n \simeq F_n^{(sy)\, c} \simeq F_n^{(sy)}$ to give an $n$-dimensional 
real oriented vector bundle, which we denote by $\tilde{E}_n$ and then
$$ H^*(F_n^{(sy)}; \Z/2\Z) \,\, \simeq  \gL^* \Z/2\Z < w_2, w_3, \dots , 
w_n >  $$
where $w_j = w_j(\tilde{E}_n)$ for each $j =  2, 3, \dots , n$.
We will see in the next section that this algebra naturally pulls back to a 
characteristic subalgebra of Milnor fibers for general symmetric matrix 
singularities generated by the Stiefel-Whitney classes of the pull-back of 
$\tilde{E}_n$ to the Milnor fiber.
\par
Although both 
$$H^*(F_n^{(sy)}; \Z/2\Z) \simeq H^*(SU_n/SO_n; \Z/2\Z)\quad \text{ and } 
\quad H^*(F_{2n}^{(sk)}; \Z) \simeq H^*(SU_{2n}/Sp_n; \Z)$$ are exterior 
algebras, neither is a Hopf algebra.  Hence, the full argument given for 
$H^*(F_{n}; \Z)$ for the relation between the cohomology and the Schubert 
decomposition cannot be given using Hopf algebra methods.  However, it does 
suggest the following conjecture is true and constitutes work in progress.  
\par
\vspace{1ex} 
\flushpar
{\bf Conjecture:}  For both $F_n^{(sk)\, c}$ and $F_n^{(sy)\, c}$, the Kronecker duals 
to the Schubert classes $S_{(\bm)}^{(sk)}$, resp. $S_{(\bm)}^{(sy)}$ for 
Schubert symbols $\bm^{(sk)}$, or $\bm^{(sy)} = (m_1, m_2, \dots , m_r)$ 
are given up to sign by $e_{(m_1)} \cdot e_{(m_2)} \cdots e_{(m_r)}$ in the 
corresponding cohomology algebra. \par

\section{Characteristic Subalgebra in the Cohomology of General Matrix 
Singularities}
\label{S:sec7} 
\par
In the preceding section we have identified for the Milnor fibers $F_m$, 
$F_m^{(sy)}$, and $F_m^{(sk)}$ (for m = 2n), their cohomology and the 
decomposition of their homology using the Schubert decomposition.  We see 
how this applies  to the structure of Milnor fibers of general matrix 
singularities of each of these types.  \par
Let $M$ denoting any one of the three spaces of complex $m \times m$ 
matrices which are general $M_{m, m}(\C)$, symmetric $Sym_m(\C)$, or 
skew-symmetric $Sk_m(\C)$ with $m = 2n$.  Also, let $\cD_m$, resp. 
$\cD_m^{(sy)}$, or $\cD_m^{(sk)}$ denote the variety of singular matrices 
of the corresponding type.  We suppose that each type is defined by $H : M 
\to \C$, which denotes either the determinant $\det$ for $\cD_m$ or 
$\cD_m^{(sy)}$, or the Pfaffian $\Pf$ for $\cD_m^{(sk)}$.  \par  
\subsection*{Matrix Singularities of a Given Type} \hfill
\par
A matrix singularity of any of the given types is defined by a holomorphic 
germ $f_0 : \C^s, 0 \to M, 0$, and the singularity is defined by $X_0 = 
f_0^{-1}(\cV), 0$ where $\cV$  denotes the appropriate variety of singular 
matrices.  
We impose an additional condition on $f$ which can take several forms based 
on forms of $\cK$-equivalence preserving $\cV$.  There is the equivalence 
defined using the parametrized action by points in $\C^s$ of the group $G = 
GL_m(\C)$ acting by $C \mapsto A\cdot C\cdot A^T$ in the  symmetric or 
skew-symmetric cases.  For the general $m \times m$ matrix case, the 
action of $G = GL_m(\C)$ acting by left multiplication suffices for studying 
the Milnor fiber.  However, for the general equivalence studying the pull-back 
of $\cD_m$ the action is given by $G = GL_m(\C) \times GL_m(\C)$ acting 
by $C \mapsto A\cdot C\cdot B^{-1}$ .  We denote the equivalence for any 
of the general, symmetric, or 
skew-symmetric cases as $\cK_M$-equivalence.  The second equivalence 
allows the action of germs of diffeomorphisms of $\C^s \times M, (0, 0)$ of 
the form $\varphi(x, y) = (\varphi_1(x), \varphi_2(x, y))$ which preserve 
$\C^s \times \cV$, and is denoted $\cK_{\cV}$ equivalence.  The third is a 
subgroup of $\cK_{\cV}$ which preserves the defining equation of $\C^s 
\times \cV$,  $H \circ pr_M$, with $pr_M$ denoting projection onto $M$.  It 
is denoted $\cK_{H}$.  See for example \cite{DP2}, \cite{D2}, or \cite{D1} 
for more details about the groups of equivalence and their relations and the 
properties of germs which have finite codimension for one of these 
equivalences.  In particular, for the three classes of varieties of singular 
matrices, $\cK_{\cV}$ and $\cK_{M}$ equivalences agree.  \par
If $f_0$ has finite $\cK_{\cV}$-codimension, then it may be deformed to 
$f_t$ which is transverse to $\cV$ in a neighborhood $B_{\gevar}(0)$ of $0 
\in \C^s$ for $t \neq 0$.  Then it is shown in \cite{DM} that one measure of 
the vanishing topology of $X_0$ is by the\lq\lq singular Milnor fiber\rq\rq 
$\tilde{X}_t = f_t^{-1}(\cV) \cap B_{\gevar}(0)$.  It is homotopy equivalent 
to a bouquet of real spheres of dimension $s-1$.  If $s < \codim_M 
(\sing(\cV))$, then this is the usual Milnor fiber of $\cV_0$.   This condition 
requires $s < 4$, resp. $3$, resp. $6$, for the three types of matrices.  
\par 
In the special case that $\cV$ is a free divisor and holonomic in the sense of 
Saito \cite{Sa} and satisfies a local weighted homogeneity condition 
\cite{DM} or is a free divisor and H-holonomic \cite{D1}, then the singular 
Milnor number is given by the length of the normal space 
$N \cK_{H \, e} f_0$, which is a determinantal module.  \par
For the three classes of varieties of singular matrices, the varieties are not 
free divisors.  Nonetheless, when $s \leq \codim_M (\sing(\cV))$, Goryunov 
and Mond \cite{GM} give a formula for the Milnor number which adds a 
correction term for the lack of freeness given by an Euler characteristic of 
a Tor complex.  Instead, Damon-Pike \cite{DP3} give a formula valid for all 
$s$ but which is presently restricted to a limited range of matrices.  It is 
given by a sum of terms which are lengths of determinantal modules, based 
on placing the varieties in a tower of free divisors \cite{DP2}.  
\par
\subsection*{Cohomology Structure of Milnor Fibers of General Matrix 
Singularities} \hfill
\par
We explain how the results in earlier sections provide information about the 
cohomology of the Milnor fiber for a matrix singularity $X_0$ for all $s$.  
\par
We consider the defining equation $H : \C^N, 0 \to \C, 0$ for $\cV$, where 
$M \simeq \C^N$ for each case.  For $\cV$ there exists $0 < \gd << \eta$ 
such that for balls $B_{\gd} \subset \C$ and$B_{\eta} \subset \C^N$  (with 
all balls centered $0$), we let $\cF_{\gd} = H^{-1}(B_{\gd}) \cap B_{\eta}$ 
so $H : \cF_{\gd} \to B_{\gd}$ is the Milnor fibration of $H$, with Milnor fiber 
$\cV_w = H^{-1}(w) \cap B_{\eta}$ for each $w \in B_{\gd}$.  By continuity, 
there is an $\gevar > 0$ so that $f_0(B_{\gevar}) \subset 
\cF_{\gd}$.  By possibly shrinking all three values, 
$H \circ f_0 : f_0^{-1}(\cF_{\gd}) \cap  B_{\gevar} \to B_{\gd}$ is the 
Milnor fibration of $H \circ f_0$.  Also, by the parametrized transversality 
theorem, for almost all $w \in B_{\gd}$, $f_0$ is transverse to $\cV_w$ and 
so the Milnor fiber of $H \circ f_0$ is given by $$X_w \,\, = \,\, (H \circ 
f_0)^{-1}(w) \cap B_{\gevar} \,\, = \,\, f_0^{-1}(\cV_w) \cap B_{\gevar}\, 
.$$  
\par
Thus, if we denote $f_0 | X_w = f_{0, w}$, then in cohomology with 
coefficient ring $R$, $f_{0, w}^* : H^*(\cV_w ; R) \to H^*(X_w ; R)$.  For any 
of the three types of matrices with $(*)$ denoting $( )$ for general matrices, 
$(sy)$ for symmetric matrices, or $(sk)$ for skew-symmetric matrices, we let
 $$\cA^{(*)}(f_0; R) \,\, \overset{def}{=} \,\, f_{0, w}^* (H^*(\cV_w ; R))\, ,$$
 which we refer to as the {\em characteristic subalgebra} of the cohomology 
of the Milnor fiber $H^*(X_w ; R)$ of $X_0$.  This is an algebra over $R$, and 
the cohomology of the Milnor fiber of the matrix singularity $X_0$ is a graded
module over $\cA^{(*)}(f_0; R)$ (both with coefficients $R$).  \par
 By Theorems \ref{Thm6.1} and 
\ref{Thm6.14} for the $m \times m$ general case or skew-symmetric 
case (with $m = 2n$),  for $R = \Z$-coefficients (and hence 
for any coefficient ring $R$) $\cA^{(*)}(f_0; R)$ is the quotient ring of a 
free exterior $R$-algebra on generators $e_{2j-1}$, for $j = 2, 3, \dots , m$, 
resp. $e_{4j-3}$ for $j = 2, 3, \dots , n$.  
For the $m \times m$ symmetric case there are two important cases where 
either $R = \Z/2\Z$ or is a field of characteristic zero.  In the first case, by 
Theorem \ref{Thm6.15}, $\cA^{(*)}(f_0; \Z/2\Z)$ is the quotient ring of a 
a free exterior algebra on generators $e_j = w_j(\tilde{E}_m)$, for $j = 2, 3, \dots , m$, 
for $w_j(\tilde{E}_m)$ the Stiefel-Whitney classes of the real oriented 
$m$-dimensional vector bundle $\tilde{E}_m$ on the Milnor fiber of $\cD_m^{(sy)}$.  
Hence, $\cA^{(*)}(f_0; \Z/2\Z)$ is a subalgebra generated by the 
Stiefel-Whitney classes of the pull-back vector bundle $f_{0, w}^*(\tilde{E}_m)$ on 
$X_w$.  \par
For the coefficient ring $R = \bk$ a field of characteristic zero, the 
symmetric case breaks-up into two cases depending on whether $m$ is even 
or odd (see \cite[(2),Thm. 6.7, Chap. 3]{MT} or Table 1 of \cite{D3}).  
\begin{equation}
\label{Eqn7.1}
  H^*(F_m^{(sy)}; \bk) \,\,  \simeq \,\, \begin{cases}
\gL^*\bk \langle e_5, e_9, \dots , e_{2m-1} \rangle \,  & \text{ if $m = 
2k+1$ }  \\
\gL^*\bk \langle e_5, e_9, \dots , e_{2m-3} \rangle \{1, e_m\}& \text{if 
$m = 2k$ } \, 
\end{cases}
\end{equation}
where $e_m$ is the Euler class of $\tilde{E}_m$. Hence, in both cases they are 
graded modules over an exterior algebra.  Hence, the Milnor fiber of $X_0$ has 
cohomology over a field of characteristic zero which, via the characteristic subalgebra
 is a graded module over the exterior algebra in either case of \eqref{Eqn7.1}.  \par
We summarize these cases with the following.
\begin{Thm}
\label{Thm7.2}
Let $f_0 : \C^s, 0 \to M, 0$ be a matrix singularity of finite $\cK_M$-codimension 
for $M$ the space of $m \times m$ matrices which are either general, symmetric, or 
skew-symmetric (with $m = 2n$).  Let $\cV$ denote the variety of singular matrices.  
Then, 
\begin{itemize}
\item[i)] The cohomology (with coefficients in a ring $R$) of the Milnor fiber 
of $X_0 = f_0^{-1}(\cV)$ has a graded module structure over the characteristic 
subalgebra $\cA^{(*)}(f_0; R)$ of $f_0$.  
\item[ii)]  In the general and skew-symmetric cases, $\cA^{(*)}(f_0; R)$ is a quotient of the free $R$-exterior algebra with generators given in 
Theorems \ref{Thm6.1} and \ref{Thm6.14} . 
\item[iii)]  In the symmetric case with $R = \Z/2\Z$, $\cA^{(sy)}(f_0; \Z/2\Z)$ 
is the quotient of the free exterior algebra over $\Z/2\Z$ on the Stiefel-Whitney classes of the real oriented vector bundle $\tilde{E}_m$ on the Milnor fiber of $\cV$.
\item[iv)]  In the symmetric case with $R = \bk$, a field of characteristic 
zero, $\cA^{(sy)}(f_0; \bk)$ is a quotient of the $\bk$-algebras in each of the cases 
in \eqref{Eqn7.1}.
\end{itemize}
\end{Thm}
\par
In light of this theorem there are several problems to be solved for 
determining the cohomology of the Milnor fiber of the matrix singularity 
$X_0$ for coefficients $R$.  \par
\vspace{1ex} 
\flushpar
{\it Questions for the Cohomology of the Milnor Fibers of Matrix Singularities} \par
\begin{itemize}
\item[1)]  Determine the characteristic subalgebras as the images of the exterior algebras by determining which monomials map to nonzero elements in $H^*(X_w ; R)$.  
\item[2)]  Find the non-zero monomials in the image by geometrically identifying the 
pull-backs of the Schubert classes. 
\item[3)]  For the symmetric case with $\Z/2\Z$-coefficients, compute the 
Stiefel-Whitney classes of the pull-back of the vector bundle $\tilde{E}_m$.
\item[4)] Determine a set of module generators for the cohomology of 
the Milnor fibers as modules over the characteristic subalgebras.
\end{itemize}
\subsection*{Transversality to Schubert Cycles} \hfill
\par
We can give a first step for these using transversality.  We let $M$ denote 
one of the spaces of $m \times m$ matrices with variety of singular 
matrices denoted by $\cV$.  There is a transitive action on $SL_m(\C)$ on 
the global Milnor fibers of the varieties of singular matrices in all three 
cases.  We let $S_{\bm}^{(*)}$ denote the Schubert cell in the global Milnor 
fiber of the corresponding type.  For each Schubert class $S_{\bm}^{(*)}$ 
and $A \in SL_m(\C)$, we let $A\cdot S_{\bm}^{(*)}$ denote the image 
under the action of $A$.  Also, we let the germ $f_1 = A^{-1} \cdot f_0$ denote 
the germ obtained by applying the constant matrix $A^{-1}$ to $f_0(x)$ 
independent of $x$.  This action preserves the global Milnor fibers of $\cV$.  
Then, deforming either the Schubert cells or $f_0$ by multiplication by $A$ 
yields the following.
\begin{Lemma}
\label{Lem7.3}
Given $f_0 : \C^s, 0 \to M, 0$ of finite $\cK_M$-codimension, for 
almost all $A \in SL_m(\C)$ the germ $f_0$ is transverse to $A\cdot 
S_{\bm}^{(*)}$ for all Schubert cells $S_{\bm}^{(*)}$ in a Milnor fiber 
$\cV_w$ of $\cV$.  Then, for $f_1 = A^{-1} \cdot f_0$ and  $e_{\bm}^{\prime}$ 
the Poincar\'{e} dual to $[\overline{S_{\bm}^*}]$, $f_1^*(e_{\bm}^{\prime})$ is the Poincar\'{e} dual of $[f_1^{-1}(\overline{S_{\bm}^{(*)}})]$.  \par
Then, $f_1$ is $\cK_M$-equivalent to $f_0$, and $f_{0 w}^* = 
f_{1 w}^*$.
\end{Lemma}
\begin{proof}
\par  
As $SL_m(\C)$ is path-connected, the action of $A$ is homotopic to the 
identity.  Let $A_t$ be such a path from $I_m$ to $A$.  Hence,  $[A_t\cdot 
\overline{S_{\bm}^{(*)}}] = [\overline{S_{\bm}^{(*)}}]$ for all $t$.  \par 
Next, by the parametrized transversality theorem and the transitive acton 
of $SL_m(\C)$ on the global Mlnor fiber, it follows that $f_0$ is transverse 
to $A\cdot \overline{S_{\bm}^{(*)}}$ for almost all $A \in SL_m(\C)$.  As 
there are only a finite number of Schubert cells, then for almost all $A$ this 
simultaneously holds for all of the Schubert cells $S_{\bm}^{(*)}$.  
For such an $A$ with $f_1 = A^{-1} \cdot f_0$, it follows that $f_1 = A \cdot f_0$ is 
transverse to all of the Schubert cells.  If $e_{\bm}^{\prime}$ denotes the Poincar\'{e} dual to $[\overline{S_{\bm}^{(*)}}]$, it is also the Poincar\'{e} dual to 
$[A\cdot\overline{S_{\bm}^{(*)}}]$.  Thus, by a fiber square argument 
$f_{1 w}^*(e_{\bm}^{\prime})$ is the Poincar\'{e} dual to 
$[f_{1 w}^{-1}(A\cdot \overline{S_{\bm}^{(*)}})]$.  \par  
Lastly, the family $f_t = A^{-1}_t \cdot f_0$ is a $\cK_M$-constant family so 
that $f_1 = A^{-1} \cdot f_0$ is $\cK_M$-equivalent  to $f_0$ and $f_{0 w}^* = 
f_{1 w}^*$. 
\end{proof}
\par
\begin{Remark}
\label{Rem7.4}
As a simple consequence of this lemma, we may replace $f_0$ by the $\cK_M$-equivalent $f_1 = A^{-1} \cdot f_0$ transverse to $\overline{S_{\bm}^{(*)}}$.  If $s < \frac{1}{2}\codim_{\R} (S_{\bm}^{(*)}) + 1$, then $f_{1 w}^{-1}(A\cdot \overline{S_{\bm}^{(*)}})$ is empty.  Hence $f_{0 w}^*(e_{\bm}^{\prime}) = 0$.  
\end{Remark} 
\par
\subsection*{Module Structure for  the Milnor Fibers} \hfill
\par
We make several remarks regarding these questions concerning the module structure.
  These involve two cases at opposite extremes, namely $s < \codim_M (\sing(X_0))$ or
 $f_0$ is the germ of a submersion.  In the first case when $s < \codim_M (\sing(\cV))$,
 $X_0$ has an isolated singularity, and the singular Milnor fiber for $f_0$ is the Milnor fiber for $X_0$, so the Milnor number and singular Milnor number agree.  Also, 
$f_{0 w}^*(e_{\bm}^{\prime}) = 0$ for all $e_{\bm}^{\prime}$ of positive degree; thus 
$\cA^{(*)}(f_0, R) = H^0(X_w; R) \simeq R$.  As the Minor fiber is homotopy equivalent 
to a CW-complex of real dimension $s-1$, the corresponding classes which occur for 
the Milnor fiber will have a trivial module structure over $\cA^{(*)}(f_0, R)$.  \par
Second, if $f_0$ is the germ of a submersion, then the Milnor fiber has the form $\cV_w \times \C^k$, where $k = s - \dim_{\C} M$ and so has the same cohomology, so   
we conclude that $f_0^* : H^*(\cV_w ; R)  \simeq H^*(X_w ; R)$ so $\cA^{(*)}(f_0, R) = H^*(X_w ; R)$.  Also, there are no singular vanishing cycles.  Thus, for these two cases there is the following expression for the cohomology of the Milnor fiber, where the 
second summand has trivial module structure shifted by degree $s-1$.
\begin{equation}
\label{Eqn7.12}
   \quad H^*(X_{w}; R) \, \simeq \cA^{(*)}(f_0, R) \oplus R^{\mu}[s-1]   
\end{equation}
where $\mu = \mu_{\cV}(f_0)$ for $\cV = \cD_m^{(*)}$ the corresponding variety of singular matrices.  \par
We ask whether this holds in general or at least for a large class of matrix singularities. 
\par 
\vspace{1ex} 
\flushpar
{\it Question: How generally valid is \eqref{Eqn7.12} for matrix singularities of the three types?} 
\par
\vspace{1ex} 
\par
For this question, we note that for the case of $2 \times 3$ complex 
matrices with $\cV$ denoting the variety of singular matrices and $s = 5$, 
the matrix singularities define Cohen-Macaulay $3$-fold singularities.  A 
stabilization of these singularities gives a smoothing and Milnor fiber.  In 
\cite[Thm. 8.4]{DP3} is given an algebraic formula for the vanishing Euler 
characteristic, which becomes the difference of the Betti numbers 
$b_3 - b_2$ of the Milnor fiber.  While specific calculations in the Appendix of 
\cite{DP3} show that the vanishing Euler characteristic typically increases in 
families with the $\cK_V$-codimension, it is initially not clear how this 
increase is distributed as changes of $b_3$ and $b_2$.  Surprisingly, 
Fr\"{u}hbis-Kr\"{u}ger and Zach \cite{FZ}, \cite{Z} show that for a large class of 
such singularities that $b_2 = 1$.  This suggests it may be possible to 
identify certain classes of $m \times m$ matrix singularities for which there are 
contributions from $\cA^{(*)}(f_0, R)$ for the topology of the Milnor fiber.  This is 
a fundamental question whose answer along with the preceding ones will clarify 
our understanding of the full cohomology of the Milnor fibers of matrix singularities.
\par
\section{Extensions to Exceptional Orbit Varieties, Complements, and Links}
\label{S:sec8} 
\par
We indicate in this section how the methods of the previous sections can be 
extended to exceptional orbit hypersurfaces for prehomogeneous vector 
spaces in the sense of Sato, see \cite{So} and \cite{SK}.  This includes 
equidimensional prehomogeneous spaces, see \cite{D3}, in the cases of both 
block representations of solvable linear algebraic groups \cite{DP2} and the 
discriminants for quivers of finite type in the sense of Gabriel, see \cite{G}, 
\cite{G2}, represented as linear free divisors by Buchweitz-Mond \cite{BM}.  \par 
Second, we can also apply the preceding methods to the complements of 
exceptional orbit hypersurfaces arising as the varieties of singular $m 
\times m$ matrices just considered and the equidimensional prehomogeneous 
spaces just described.  Third, in \cite{D3}, the cohomology of the link of one 
of these singularities is computed as a shift of the (co)homology of the 
complement.  Thus, the Schubert classes for the complement correspond to 
cohomology classes in the link. However, we explain how the multiplicative 
cohomology structure of the complement contains more information than the 
cohomology of the link. \par
\subsection*{Exceptional Orbit Hypersurfaces for the Equidimensional Cases} 
\hfill
\par
\subsubsection*{Block Representations of Linear Solvable Algebraic Groups} 
\hfill
\par
First, for the case of block representations of solvable linear algebraic 
groups, in \cite[Thm 3.1]{DP} the complement was shown to be a 
$K(\pi, 1)$-space where $\pi$ is a finite extension of $\Z^n$ (for $n$ the rank of 
the solvable group) by the finite isotropy group of the action on the open orbit.  
The solvable group is an extension of an algebraic torus by a unipotent group 
which is contractible.  The resulting cell decomposition follows from that for 
the torus times the unipotent group.  Thus, the decomposition is that modulo 
the finite group.  In important cases of (modified) Cholesky-type 
factorization for the three types of matrices and also $m \times (m+1)$ 
matrices the finite group is either the identity or $(\Z/2\Z)^n$ and the 
resulting quotient is shown, see \cite[Thm 3.4]{DP}, to still be the extension 
of a torus by a (contractible) unipotent group.  \par 
Thus, for these cases the cell decomposition follows from the product 
decomposition for the complex torus times the unipotent group, which has as 
a compact model a compact torus of the same rank.  Moreover, by 
\cite[Thm 4.1]{DP}, the cohomology with complex coefficients is an exterior 
algebra which has as generators $1$-forms defined from the defining 
equation of the exceptional orbit hypersurface.  \par
Also, by \cite[Thm 3.2]{DP} the Milnor fiber is again a 
$K(\pi^{\prime}, 1)$-space with $\pi^{\prime}$ a subgroup of $\pi$ (for the 
complement) with quotient $\Z$.  Again, by \cite[Thm 3.4]{DP} for the cases 
of (modified) Cholesky-type factorization of matrices, it is also true that the 
Milnor fiber for these cases is the extension of a torus, except of one lower 
rank, by the unipotent group.  Likewise the cohomology with complex 
coefficients of the Milnor fiber is again an exterior algebra which has one 
fewer generator, as the result of a quotient by a single specified relation.  
\par
\subsubsection*{Discriminants of Quivers of Finite Type} \hfill
\par
The quivers are defined by a finite ordered graph $\gG$ having for 
each vertex $v_i$ a space $\C^{n_i}$ and for each directed edge from 
$v_i$ to $v_j$ a linear map $\varphi_{i j} : \C^{n_i} \to \C^{n_j}$.  Those 
quivers of finite type were classified by Gabriel \cite{G}, \cite{G2}.  The 
discriminants for the quiver spaces of finite type were shown by 
Buchweitz-Mond \cite{BM} to be linear free divisors.  As such these discriminants 
are exceptional orbit hypersurfaces for the action of the group $G = ( \prod_{i = 
1}^{k} GL_{n_i}(\C)) / \C^*$ where $k = | \gG |$.    Since each 
$GL_{n_i}(\C)$ topologically factors as $SL_{n_i}(\C) \times \C^*$, then the 
complement is diffeomorphic to $(\prod_{i = 1}^{k} SL_{n_i}(\C)) \times 
(\C^*)^{k-1}$.  The earlier results for the Schubert decomposition for each 
$SL_{n}(\C)$ via its maximal compact subgroup $SU_n$ and the product cell 
decomposition for $(\C^*)^{k-1}$ gives a {\em product Schubert cell 
decomposition} for the complement. \par
The Milnor fiber has an analogous form $( \prod_{i = 1}^{k} SL_{n_i}(\C)) 
\times (\C^*)^{k-2}$, and a {\em product Schubert cell decomposition} for 
the Milnor fiber.  \par 
The cohomology of the complement is given by \cite[(5.11)]{D3} as an 
exterior algebra on a specific set of  generators.  The cohomology of the 
Milnor fiber is also an exterior algebra except with 
one fewer degree $1$ generator, see \cite[(Thm 5.4)]{D3}.  Furthermore, by 
Theorem \ref{Thm6.1} relating the Schubert decomposition for $SL_{n}(\C)$ 
via its maximal compact subgroup $SU_n$ with the cohomology classes, we 
conclude that for both the complement and the Milnor fiber of the 
discriminant of the space of quivers, the closures of the product Schubert 
cells provide a set of generators for the homology.  \par
\subsection*{Complements of the Varieties of Singular Matrices} \hfill
\par
We can likewise give a Schubert decomposition for the complements of the 
varieties of $m \times m$ matrices which are general, symmetric or 
skew-symmetric.  We note that in \cite{D3} the complements were given as 
$GL_m(\C)$ for the general matrices, $GL_m(\C)/O_m(\C)$ for the 
symmetric matrices, and $GL_{2n}(\C)/Sp_n(\C)$ for the skew-symmetric 
case with $m = 2n$.  These have as compact models the symmetric spaces 
$U_m$, resp. $U_m/O_m$, resp. $U_{2n}/Sp_n$.  Each of these 
has a Schubert decomposition given in \cite{KM}.  As remarked in \S 
\ref{S:sec3}, $U_m$ has a Schubert decomposition by cells $S_{\bm}$ for 
$\bm = (m_1, m_2, \dots , m_r)$, where $m_1$ may equal $1$ and it is not 
required that  $\sum_{i = 1}^{r} \theta_i \equiv 0 \, \mod \, 2\pi$. \par 
 Second, in \cite[\S 5]{KM} is given a Schubert decomposition for 
$U_m/O_m$ using for the symmetric Schubert cell $S_{\bm}^{(sy)}$ the 
symmetric factorization into pseudo-rotations except again $\bm^{(sy)} = 
(m_1, m_2, \dots , m_r)$, where $m_1$ may equal $1$ and it is not required 
that $\sum_{i = 1}^{r} \theta_i \equiv 0 \, \mod \, \pi$.  \par
Third,  in \cite[\S 7]{KM} is given a Schubert decomposition for for 
$U_{2n}/Sp_n$ using for the skew-symmetric Schubert cell 
$S_{\bm}^{(sk)}$ the skew-symmetric factorization into pseudo-rotations 
except again $\bm^{(sk)} = (m_1, m_2, \dots , m_r)$, where $m_1$ may 
equal $1$ and  it is not required that  $\sum_{i = 1}^{r} \theta_i \equiv 0 \, 
\mod \, 2\pi$. \par  
In the case of $U_m$ and $U_{2n}/Sp_n$ the cohomology with integer 
coefficients is an exterior algebra with an added generator of degree $1$; 
and for $U_m/O_m$ the cohomology with $\Z/2\Z$ coefficients is an 
exterior algebra with an added generator of degree $1$.  Hence, a counting 
argument analogous to that for the Milnor fibers show that the closure of 
each Schubert class gives a homology generator for the complement.
\subsubsection*{Complements of the Varieties of Singular $m \times n$ 
Matrices} \hfill
\par
The varieties of singular $m \times n$ complex matrices, $\cV_{m, n}$, with 
$m \neq n$ were not considered earlier because they do not have Milnor 
fibers.  However, the methods do apply to the complement and link as a 
result of work of J. H. C. Whitehead \cite{W}.  Let $M = M_{m, n}(\C)$ denote 
the space of $m \times n$ complex matrices.  We consider the case where 
$m > n$.  The other case $m < n$ is equivalent by taking transposes.  The 
left action of $GL_m(\C)$ acts on $M$ with an open orbit consisting of the 
matrices of rank $n$.  This is the complement to the variety $\cV_{m, n}$ 
of singular matrices and can be described as the ordered set of $n$ 
independent vectors in $\C^m$.  Then, the Gram-Schmidt procedure replaces 
them by an orthonormal set of $n$ vectors in $\C^m$.  This is the Stiefel 
variety $V_n(\C^m)$ and the Gram-Schmidt procedure provides a strong 
deformation retract of the complement $M \backslash \cV_{m, n}$  onto the 
Stiefel variety $V_n(\C^m)$.  Thus, the Stiefel variety is a compact model 
for the complement.  Whitehead \cite{W} computes both the (co)homology of 
the Stiefel variety using a Schubert decomposition which he gives.  The 
cohomology for integer coefficients of the complement of the variety 
$\cV_{m, n}$ is given by: 
\begin{equation}
\label{Eqn8.1}
 H^*(M_{m, n} \backslash \cV_{m, n}; \Z) \,\, \simeq \,\, \gL^*\Z\langle 
e_{2(m-n)+1}, e_{2(m-n)+3}, \dots , e_{2m-1}\rangle   
\end{equation}
with degree of $e_j$ equal to $j$.  
Again the Schubert decomposition gives for the closure of each Schubert cell 
a homology generator.  
\subsection*{Cohomology of the Links and Schubert Decomposition of the 
Complement} \hfill
\par
Consider an exceptional orbit variety $\cE$ of a prehomogeneous vector 
space $V$ of $\dim_{\C} V = N$.  Suppose there is a compact manifold $K 
\subset V \backslash \cE$ oriented for a coefficients field $\bk$, which is a 
compact model for the complement $V \backslash \cE$.  Then the 
cohomology of the link $L(\cE)$ is given, see \cite[Prop. 1.9]{D3}, by the 
following formula \par
\vspace{1ex} 
\flushpar
{\em Cohomology of the Link $L(\cE)$: } \par
\begin{equation}
\label{Eqn8.2}   
\widetilde{H}^*(L(\cE); \bk) \,\, \simeq   \widetilde{H^*(K; \bk)}\left[ 2N - 2 
- \dim_{\R} K \right] \, , 
\end{equation}
where the graded vector space 
$\widetilde{H^*(X; \bk)}\left[ r\right]$ will denote the vector space $H^*(X; 
\bk)$, truncated at the top degree and shifted upward by degree $r$.  
Furthermore, to a basis of vector space generators of $H_q(K; \bk)$, $q < 
\dim_{\R}K$, there corresponds by Alexander duality a basis of vector space 
generators of $H^{2N-2-q}(K; \bk)$.  \par
As a consequence of this and the preceding established relations between 
the Schubert decomposition (or product Schubert decomposition) of the 
complement and the homology, we obtain the following conclusions.
\begin{Thm}
\label{Thm8.4}
For the following exceptional orbit varieties $\cE$ there are the following 
relations between the Schubert (or product Schubert) decomposition for a 
compact model of the complement and the cohomology of the link obtained 
by shifting the cohomology of the compact model (for coefficients a field of 
characteristic zero $\bk$ unless otherwise stated).  
\begin{itemize}
\item[1)]  For the equidimensional solvable case for (modified) 
Cholesky-type factorizations of $m \times m$ matrices of all three types or 
$(m + 1) \times m$ matrices, the cohomology of the link is given by the 
shifted cohomology of the compact model torus, see \cite[Thm 4.5]{D3}.  
The closures of the cells of the product cell decomposition of nonmaximal 
dimension give a homology basis which correspond to the cohomology basis 
of the link after the shift.  
\item[2)] For the discriminant of the quiver space for a quiver of finite type, 
the cohomology of the link is the shifted cohomology of the compact model 
described above with shift given by \cite[Thm. 5.4]{D3}.  The closures of 
cells of the product Schubert decomposition of nonmaximal dimension for 
the complement give a homology basis which correspond after the shift to 
the cohomology basis for the link.
\item[3)] For the varieties of singular $m \times m$ complex matrices, in 
the general case or the skew-symmetric case with $m$ even, the cohomology of 
the link is the shifted cohomology of the compact 
symmetric spaces $U_m$, resp. $U_{2n}/Sp_n$ ($m = 2n$) given above with 
shift given in \cite[Table 2]{D3}.  The closures of the Schubert cells of 
nonmaximal dimension in each case give a homology basis which corresponds 
to the cohomology basis of the link after the shift.  
\item[4)] For the varieties of singular $m \times m$ complex symmetric 
matrices, the shifted cohomology of $H^*(U_m/O_m ; \Z/2\Z)$, described 
above, gives the cohomology of the link for $\Z/2\Z$-coefficients, where the 
shift is $\binom{m+1}{2} -2$.  The closures of the Schubert cells of 
nonmaximal dimension in the Schubert decomposition give a basis of 
$\Z/2\Z$-homology classes corresponding to the cohomology basis of the 
link after the shift.  For coefficients in a field $\bk$ of characteristic zero, the 
cohomology of $U_m/O_m$, is an exterior algebra which depends on whether 
$m$ is even or odd and the shifts are given in \cite[Table 2]{D3}, without a 
direct relation with the Schubert decomposition.  
\item[5)]
For the variety of singular $m \times n$ complex matrices, $\cV_{m, n}$ 
(with $m > n$), the cohomology of the compact model, the Stiefel variety 
$V_n(\C^m)$, for the complement is given by \eqref{Eqn8.1}.  The 
cohomology of the link is given in \eqref{Eqn8.2} as the upper truncated and 
cohomology $H^*(M_{m, n} \backslash \cV_{m, n},\bk)$ shifted by $n^2-2$ 
(as a graded vector space).  The closures of the Schubert cells of 
nonmaximal dimension give a homology basis for the cohomology of the link 
after the shift.
\end{itemize}
\end{Thm}
\par
\subsection*{Complements of the Varieties of Matrix Singularities} \hfill
\par
Given a matrix singularity $f_0 : \C^s, 0 \to M, 0$ with $\cV \subset M$ the 
variety of singular matrices and $X_0 = f_0^{-1}(\cV)$.  Here $M$ can 
denote any of the spaces of matrices and of any sizes.  In the preceding, we 
indicated how the cohomology of the link $L(\cV)$ is expressed as an upper 
truncated and shifted cohomology of the complement $M \backslash \cV$.  
Because of the shift, we showed in \cite{D3} that the cohomology product 
structure is essentially trivial.  Thus, the link is a stratified real analytic set 
whose structure depends upon much more than just the group structure of 
the (co)homology.  On the other hand, we showed in \cite{D3} that the 
cohomology structure of the complement is an exterior algebra, and hence 
contributes considerably more that just the vector space structure of the 
cohomology of the link.  This extra cohomology structure captures part of 
the additional structure. \par
Consequently, for the matrix singularity, using the earlier notation, we note 
that there is a map of complements $f_0 : (B_{\gevar} \backslash X_0) \to 
(B_{\gd} \backslash \cV)$.  Also, $B_{\gd} \backslash \cV \simeq  M 
\backslash \cV$, which has a compact model given by either a symmetric 
space or a Stiefel manifold.  Thus, the cohomology of the complement 
$H^*(B_{\gevar} \backslash X_0 ; R)$ is a module over the characteristic 
subalgebra which is the image of $H^*(B_{\gd} \backslash \cV ; R)$ under 
$f_0^*$.  In turn, this is an exterior algebra.  Hence, the multiplicative 
structure considerably adds to the group structure that would result from 
the link.  This is just as for the Milnor fiber described earlier.


\begin{thebibliography}{M-VI}
\bibitem[Bo]{Bo}  Borel, A. {\em Sur les cohomologie des espaces fibr\'{e}s 
principaux et des espaces homogenes de groupes de Lie compacts}
Annals Math. {\bf 57}, (1953) 115--207.

\bibitem[Bn]{Bn}  Bredon, G. E. {\em Introduction to Compact 
Transformation Groups}, Academic Press, 1972.

\bibitem[Br]{Br}  Brieskorn, E. V. {\em Beispiele sur Differentialtopologie 
von Singularit\"{a}ten}, Invent. Math. {\bf vol 2} (1966) 1--14.

\bibitem[Br2]{Br2} \bysame {\em Die Monodromie der isolierten 
Singularit\"{a}ten von Hyperfl\"{a}chen}, Manuscripta Math. {\bf vol 2} 
(1970) 103--161.

\bibitem[Br3]{Br3} \bysame {\em Singular Elements of Semi-simple 
Algebraic Groups}, Proc. Intl Congress Math., Nice {\bf vol 2} 
Gauthiers-Villars, Paris  (1970) 279--284.

\bibitem[Br4]{Br4} \bysame {\em Sur les Groupes de Tresses (apr\'{e}s 
Arnold)}, S\'{e}minaire Bourbaki (1971/72) Springer Lect. Notes in Math. 
{\bf 317}, (1973).

\bibitem[Br5]{Br5} \bysame {\em Die Hierarchie der $1$-unimodularen 
Singularit\"{a}ten}, Manuscripta Math. {\bf vol 27} (1979) 183--219.

\bibitem[BM]{BM}  Buchweitz, R. O. and Mond, D. {\em Linear Free Divisors 
and Quiver Representations}, in {\rm Singularities and Computer Algebra},  
London Math. Soc. Lect. Ser.  {\bf vol 324}  Cambridge Univ. Press, 2006, 
41--77.

 \bibitem[D1]{D1} Damon, J.  {\em Higher Multiplicities and Almost 
Free Divisors and Complete Intersections}, Memoirs Amer. Math. Soc. 
{\bf 123 no 589} (1996).

\bibitem[D2]{D2} \bysame {\em Nonlinear Sections of Non-isolated 
Complete Intersections}, New Developments in Singularity Theory, Eds. D. 
Siersma, C. T. C. Wall, V. Zakalyukin, Nato Science Series {\bf 21}, Kluwer 
Acad. Publ.  (2001), 405--445.

\bibitem[D3]{D3} \bysame {\em Topology of Exceptional Orbit 
Hypersurfaces of Prehomogeneous Spaces}, Journal of Topology {\bf 9 no. 3} 
(2016) 797--825.

 \bibitem[DM]{DM}  Damon, J. and Mond, D. {\em $\cA$--Codimension 
and the Vanishing Topology of Discriminants},  Invent. Math. {\bf 106} 
(1991), 217--242.  

\bibitem[DP]{DP}   Damon, J. and Pike, B.  {\em Solvable Group 
Representations and Free Divisors whose Complements are $K(\pi, 1)$\lq 
s},  Proc. Conference on Singularities and Generic Geometry, in Top. and its 
Appl. {\bf 159} (2012) 437--449.

\bibitem[DP2]{DP2} \bysame {\em Solvable Groups, Free 
Divisors and Nonisolated Matrix Singularities I: Towers of Free Divisors},  
Annales de l\rq Inst. Fourier {\bf 65 no. 3} (2015)  1251--1300.

\bibitem[DP3]{DP3} \bysame {\em Solvable Groups, Free
Divisors and Nonisolated Matrix Singularities II: Vanishing Topology}, 
Geom. and Top. {\bf 18 no. 2} (2014)  911--962.

\bibitem[FZ]{FZ} Fr\"{u}hbis-Kr\"{u}ger, A. and Zach, M.  {\em On the 
Vanishing Topology of Isolated Cohen-Macaulay codimension 2 Singularities}, 
(2015) preprint: arXiv:1501:01915.

\bibitem[G]{G} Gabriel, P.  {\em Unzerlegbare Darstellungen I},  Manuscripta 
Math. {\bf 6} (1972) 71--103.

\bibitem[G2]{G2} \bysame  {\em R\'{e}presentations Ind\'{e}composables},  
S\'{e}minaire Bourbaki Expos\'{e} 444 (1974) in Springer Lect Notes. 
{\bf  431} (1975) 143--169.

\bibitem[GW]{GW} Goodman, R. and Wallach, N.  {\em Representations and 
Invariants of the Classical Groups},  Cambridge Univ. Press.

\bibitem[GM]{GM} Goryunov, V. and Mond, D.  {\em Tjurina and Milnor
Numbers of Matrix Singularities},  J. London Math. Soc. {\bf 72 (2)} 
(2005), 205--224.

\bibitem[GMNS]{GMNS} Granger, M., Mond, D., Nieto-Reyes, A., and  Schulze, 
M. {\em Linear Free Divisors and the Global Logarithmic Comparison 
Theorem},  Annales de l\rq Inst. Fourier {\bf 59} (2009) 811--850.

\bibitem[KM]{KM} Kadzisa, H. and Mimura, M. {\em Cartan Models and Cellular 
Decompositions of Symmetric Riemannian Spaces},  Top. and its Appl. 156 
(2008) 348--364. 

\bibitem[KMs]{KMs} Kato, M. and Matsumoto, Y. {\em On the Connectivity of 
the Milnor Fiber of a Holomorphic Function at a Critical Point}, 
Manifolds-Tokyo 1973 (Proc. Int\rq l. Conf., Tokyo, 1973), Univ. Tokyo 
Press (1975) 131--136.

\bibitem[MT]{MT} Mimura, M. and Toda, H. {\em Topology of Lie Groups I and 
II}, Translations of Math. Monographs {\bf 91}, Amer. Math. Soc. (1991). 

\bibitem[Ma]{Ma}  Massey, W. {\em A basic Course in Algebraic Topology} 
Springer Graduate texts {\bf 127}, 1991.

\bibitem[Mi]{Mi} Miller, C.E. {\em The Topology of Rotation Groups} Annals of 
Math {\bf (2) vol 57} (1953), 90--114.

\bibitem[MS]{MS}  Milnor, J. and Stasheff, J. {\em Characteristic Classes}, 
Annales of Math. Studies {\bf 76}, Princeton Univ. Press, 1974.

\bibitem[Sa]{Sa} Saito, K. {\em Theory of Logarithmic Differential 
Forms and Logarithmic Vector Fields}, J. Fac. Sci. Univ. Tokyo Sect. 
Math. {\bf 27}(1980), 265--291.

\bibitem[So]{So}  Sato, M. {\em Theory of Prehomogeneous Vector Spaces 
(algebraic part)}, --English translation of Sato's lectures from notes by T. 
Shintani,\, Nagoya Math. J. \textbf{120} (1990), 1--34.

\bibitem[SK]{SK} Sato, M. and Kimura, T. {\em A Classification of Irreducible 
Prehomogeneous Vector Spaces and Their Relative Invariants}, Nagoya Math. 
Jour. {\bf 65}, (1977) 1--155.

\bibitem[Sp]{Sp}  Spanier, E. H. {\em Algebraic Topology},  McGraw-Hill Publ.  
(1966).

\bibitem[W]{W} Whitehead, J. H. C. {\em On the Groups $\pi_r(V_{n, m})$ 
and Sphere Bundles},  Proc. London Math. Soc. {\bf (2) vol 48}, (1944), 
243--291. 

\bibitem[Y]{Y}  Yokota, I. {\em On the Cellular Decompositions of Unitary 
groups}, J. Inst. Polytech. Osaka City Univ. Ser. A. {\bf 7} (1956) 39--49.

\bibitem[Z]{Z}  Zach, M. {\em Vanishing cycles of smoothable isolated
Cohen-Macaulay codimension $2$ singularities of type $2$}, preprint 2016.

\end{thebibliography}
\end{document}